\documentclass[a4paper,11pt,twoside]{article}

\usepackage[nodayofweek]{datetime}

\usepackage[utf8]{inputenc}
\usepackage[T1]{fontenc}
\usepackage[english]{babel}

\usepackage{charter}

\usepackage{xcolor}
\usepackage{verbatim}

\usepackage{hyperref}

\usepackage[top=3.cm, bottom=4.0cm, left=2.2cm, right=2.2cm]{geometry}
\usepackage{amsmath}
\usepackage{amsthm}	
\usepackage{amsfonts}	
\usepackage{amssymb,bbm}
\usepackage{shuffle}
\usepackage{mathrsfs}
\usepackage{pifont}

\usepackage{enumerate}
\usepackage{enumitem}
\usepackage[nobysame, alphabetic,initials]{amsrefs}

\usepackage{stmaryrd}
\usepackage{appendix}

\usepackage{graphicx}

\usepackage{matlab-prettifier}

\usepackage{
						ulem		
} 

\numberwithin{equation}{section}

\theoremstyle{plain}
\newtheorem{theorem}{Theorem}[section]
\newtheorem{lemma}[theorem]{Lemma}
\newtheorem{proposition}[theorem]{Proposition}

\newtheorem{definition}[theorem]{Definition}

\newtheorem{remark}[theorem]{Remark}
\newtheorem{example}[theorem]{Example}
\newtheorem{assumption}[theorem]{Assumption}

\allowdisplaybreaks

\newcommand{\bE}{\mathbb{E}}
\newcommand{\bF}{\mathbb{F}}

\newcommand{\bH}{\mathbb{H}}

\newcommand{\bN}{\mathbb{N}}
\newcommand{\bP}{\mathbb{P}}

\newcommand{\bR}{\mathbb{R}}

\newcommand{\bW}{\mathbb{W}}

\newcommand{\bZ}{\mathbb{Z}}

\newcommand{\cA}{\mathcal{A}}
\newcommand{\cB}{\mathcal{B}}
\newcommand{\cC}{\mathcal{C}}

\newcommand{\cE}{\mathcal{E}}
\newcommand{\cF}{\mathcal{F}}
\newcommand{\cG}{\mathcal{G}}
\newcommand{\cH}{\mathcal{H}}

\newcommand{\cL}{\mathcal{L}}
\newcommand{\cM}{\mathcal{M}}

\newcommand{\cP}{\mathcal{P}}

\newcommand{\cT}{\mathcal{T}}

\newcommand{\cV}{\mathcal{V}}

\newcommand{\cX}{\mathcal{X}}

\newcommand{\scI}{\mathbf{I}}
\newcommand{\scJ}{\mathbf{J}}
\newcommand{\scK}{\mathbf{K}}

\newcommand{\rx}{\mathbf{X}}

\newcommand{\rZ}{\mathbf{Z}}

\newcommand{\RKHS}{\cH}
\newcommand{\fWIC}{\widehat{\cH}}

\DeclareMathOperator{\spn}{span}
\DeclareMathOperator{\lin}{Lin}

\newcommand{\oSlash}{{\mbox{\o}}}

\newcommand{\newX}{X^\dagger}

\newcommand{\1}{\mathbbm{1}}

\usepackage{pgf,tikz}
\usetikzlibrary{arrows}
\usetikzlibrary{graphs,graphs.standard,quotes}

\tikzstyle{vertex} = [fill, shape=circle,inner sep=2pt,]
\tikzstyle{edge} = [fill, line width = 0.5pt]

\usepackage{enumerate}

\title{A Mimicking Theorem for processes driven by \\ fractional Brownian motion}
\author{
	\normalsize Kevin Hu \\[8pt]
	\small kevin\_hu@brown.edu
	\and
	\normalsize Kavita Ramanan \\[8pt]
	\small kavita\_ramanan@brown.edu
	\and
	\normalsize William Salkeld \\[8pt]
	\small william\_salkeld@brown.edu
}

\begin{document}
	
	\maketitle
	\setcounter{tocdepth}{1}
	\tableofcontents
	
	\newpage
	\section{Introduction}
	
	This paper focuses on answering the question ``Given a stochastic process, can we find a stochastic differential equation with a weak solution that has the same law as the law of the stochastic process?'' The existence of such a stochastic differential equation is referred to as a \emph{mimicking theorem} since the weak solution to the SDE mimics the properties of the stochastic process. 

    The first such mimicking theorem was proved in \cite{Gyongy1986Mimicking}: unlike subsequent and our present result, the coefficients of the mimicking SDE are functions of time and space, and the proof relies heavily on PDE techniques. More recently, \cite{Brunick2013Mimicking} proved a mimicking theorem where the coefficients are dependent on the entire path of the solution. This paper was motivated by calibration in financial mathematics and the desire to identify an appropriately rich class of stochastic differential equations for a financial model given asset price data. As asset prices data takes the form of a path, it is natural that the resulting stochastic processes are path-dependent, and therefore non-Markovian. 

    As we demonstrate below, mimicking theorems have in recent years found a wide range of applications in reinforced learning \cite{Jia2023QLearning}, mean-field games and control problems \cites{Lacker2023Closed,Daudin2023On} and finance \cite{Beiglbock2022From}. Our interest is focused on two applications: \emph{quantitative propagation of chaos} following the ideas of \cites{lacker2022Hierarchies, Lacker2023Sharp} and separately in the derivation of marginal dynamics for locally interacting stochastic differential equations \cite{lacker2020marginal}. 

    We prove a new mimicking theorem for a class of processes that are not It\^o diffusions. This distinguishes our work from all previous results and demonstrates that mimicking theorems can be proved without needing to reach for semi-martingale or PDE techniques. Following on from our first paper \cite{Hu2023Locally}, we consider stochastic processes with an additive Gaussian noise. The typical example of a Gaussian noise that we consider is \emph{fractional Brownian motion}, but we emphasise that our results hold for a rich class of Gaussian processes that satisfy the so-called ``securely locally non-deterministic'' property. 

    Originally motivated by the study of the locally interacting stochastic differential equations with an additive fractional Brownian motion, our desire was to prove a mimicking theorem that could be substituted into the techniques of \cite{lacker2020marginal} as part of the derivation of the so-called ``local equation''. What we realised was that such stochastic processes were a Volterra kernel convolution away from a semi-martingale, that we could apply a mimicking theorem to. Further. there was another Volterra kernel that transformed the semi-martingale to a stochastic process that had the same distribution as our original process. Essentially, this is the technique used to prove Theorems \ref{theorem:fbm-Mimicking} and \ref{brownianFilteringThm}. 

    By studying this problem in more detail, we found that this bijective transformation was only possible because the Hilbert space filtration of reproducing kernels for a fractional Brownian motion is isomorphic to the Hilbert space filtration of reproducing kernels for a Brownian motion. Further, this property also lead to powerful entropy results which were also critical to proving subsequent results in this work. We refer the interested reader to \cite{Hu2023Locally} which includes a more thorough study of the necessary Gaussian theory to prove these results. 
	
	\subsection*{Quantitative propagation of chaos}

    Consider some collection of stochastic differential equations
    \begin{equation*}
		dX_t^{i, n} = \bigg( b_0\big( t, X^{i,n} \big) + \tfrac{1}{n-1}\sum_{\substack{j=1 \\j \neq i}}^n b\big( t, X^{i,n}, X^{j,n} \big) \bigg) dt + dZ_t^{i,n}, \quad i = 1, ..., n.
	\end{equation*}
    \emph{Propagation of chaos} is the property that as the number of equations is taken to be large, the random empirical measure $\tfrac{1}{N} \sum_{i=1}^n \delta_{X^{i, n}}$ in the drift term can be replaced in the limit by a deterministic measure $\mu$ that corresponds to the law of the resulting dynamics. This \emph{McKean-Vlasov} stochastic differential equation 
    \begin{equation*}
		dX_t = \bigg( b_0\big( t, X \big) + \Big\langle \mu, b\big( t, X, \cdot \big) \Big\rangle \bigg) dt + dZ_t \quad \mu = \bP \circ \big( X \big)^{-1}
	\end{equation*}
    is known as the mean-field limit and such equations are frequently used to model the dynamics of large systems of equations that interact via an empirical measure. For a more comprehensive survey of the literature, we refer the reader to \cite{Jabin2017Mean}. 

    \emph{Quantitative propagation of chaos} is the discipline of establishing a rate of convergence for the empirical measure converging to the associated law of the McKean-Vlasov equation. Many results proving quantitative propagation of chaos rates find their origin in \cite{Sznitman}, where synchronous coupling arguments and Lipschitz coefficients (in both spatial variables and measure variables) lead to estimates with regard the \emph{Wasserstein metric} of the form
    \begin{equation*}
        \bW^{(2)}\big[ P^{n, k}, \mu^{\otimes k} \big] = O \big( \tfrac{k}{n} \big). 
    \end{equation*}
    Another popular technique is to demonstrate that the \emph{relative entropy} $\bH\big[ P^{n, n}\big| \mu^{\otimes n} \big]$ is uniform for $n \in \bN$. From this, one can conclude that
    \begin{equation*}
        \big\| P^{n, k} - \mu^{\otimes k} \big\|_{\mathbb{TV}}^2 = O\big( \tfrac{k}{n} \big). 
    \end{equation*}
    This work focuses on the techniques of \cite{lacker2022Hierarchies}, which are extended in \cite{Lacker2023Sharp}, to prove the sharp rates of convergence of the form
    \begin{equation*}
        \bH\big[ P^{n, k} \big| \mu^{\otimes k} \big] = O\Big( \big( \tfrac{k}{n} \big)^2 \Big), 
    \end{equation*}
    see Theorems \ref{qPropChaosfBM} and \ref{theorem:qPropChaos} below. Our results, which can be found in Section \ref{section:QPoC} are distinct from previous results in that we consider interacting stochastic differential equations that are driven by additive Gaussian noises rather than Brownian motion. The entropy estimates and mimicking theorem that we prove in Section \ref{section:Mimicking} are central to the techniques used in the proof. We also provide examples (in the form of fractional Ornstein-Uhlenbeck process, see Example \ref{example:fOU-prcess}) to demonstrate that the rate of convergence we prove is sharp. 

    \subsection*{Locally interacting processes}

    Another distinct collection of stochastic differential equations that we consider in Section \ref{section:LocalEq} are
    \begin{equation*}
    	\left.
    	\begin{aligned}
	    	&dX_t^u = b\Big( t, X^u[t], \mu^u\big( X[t] \big) \Big) dt + dZ_t^u
    		\quad
    		X_0^u \sim \lambda_0
    		\\
    		&\mu^u(X[t]) = \tfrac{1}{\kappa} \sum_{v\in N_u} \delta_{X^u[t]} 
    	\end{aligned}
    	\right\}
    	\quad
    	u \in V, 
        \quad 
        (V, E) \in \cG. 
    \end{equation*}
    These processes interact with one another based on some underlying graph structure where the stochastic differential equation associated to a vertex $u \in V$ has dynamics that depend only on itself and the dynamics of the neighbouring stochastic processes $N_u = \big\{v \in V: \{u, v\} \in E \big\}$. Such stochastic processes arise in statistical physics \cites{Dereudre2003Interacting, Redig2010Short}, mathematical finance \cite{Nadtochiy2020Mean} and oscillator synchronisation \cite{Medveded2019Continuum}.
    
    As an application of the Malliavin calculus results relating to the fundamental martingale proved in \cite{Hu2023Locally}, the authors demonstrate that such stochastic processes form a \emph{2-Markov random field} on pathspace. In this paper, we go further in combination with our mimicking theorem and entropy estimates from Section \ref{section:Mimicking} to describe a McKean-Vlasov equation whose law is equal to the marginal dynamics of some small neighbourhood of the graph. Building on the ideas of \cite{lacker2020marginal}, we derive the so-called \emph{local equation} (see Theorem \ref{theorem:LocalEquation}) and extend previous results to prove a much broader weak uniqueness result for the law of the local equation (see Theorem \ref{theorem:Local-Uniq}). 
    
    Other works that consider such locally interacting equations with Brownian motion include \cites{Lacker2019Local,lacker2020Locally}. 

    \subsection*{Summary}

    In Section \ref{section:MainResultsMim}, we provide an accessible summary of the results we prove in this paper. These results are specific to \emph{fractional Brownian motion}, which is the canonical example of the Gaussian processes that we use as additive noise. 

    In Section \ref{subsec:Filtering}, we summarise and extend results from \cite{Hu2023Locally} to explain how a fractional Brownian motion and other Gaussian processes that are said to be ``securely locally non-deterministic'' can be transformed via convolution with a Volterra kernel into a Brownian motion. 

    In Section \ref{section:Mimicking}, we apply this theory of Volterra kernels to consider the \emph{fundamental semimartingale} associated to a collection of stochastic processes. Next, we use this representation to prove a mimicking theorem (see Theorem \ref{brownianFilteringThm}), and entropy and transport estimates (see Proposition \ref{fBmEntropy} and Proposition \ref{proposition:TalagrandIQ}). 

    In Section \ref{section:QPoC}, we prove quantitative propagation of chaos results for a collection of stochastic differential equations driven by additive Gaussian noises as an application of Theorem \ref{brownianFilteringThm}. These results adapt the techniques of \cite{lacker2022Hierarchies}, although the technical results needed to accommodate the additive Gaussian noises (which include Proposition \ref{fBmEntropy}) are far from immediate.  

    Separately, in Section \ref{section:LocalEq} we consider collections of locally interacting stochastic differential equations with independent additive Gaussian noises indexed by a $\kappa$-regular graph. Using Theorem \ref{brownianFilteringThm}, we derive a $(\kappa+1)\times d$-dimensional McKean-Vlasov equation that has the same law as the marginal dynamics of a collection of stochastic processes associated to a graph ball of radius $1$. We show weak existence, and under additional assumptions weak uniqueness for this McKean-Vlasov equation. Further, these assumptions are more general than the associated results proved in \cite{lacker2020marginal} and rely on the establishment of a conditional Talagrand inequality (see Proposition \ref{proposition:Talagrand2}) which involves a number of technical adjustments the the proof method of Proposition \ref{proposition:TalagrandIQ}. 
    
    \subsection*{Notation}

    For $X\in C([0,T]; \bR^d):=\cC_T^d$, we denote $x[t] \in \cC_t^d$ for the continuous path up to $t$ and $x_t$ for the evaluation of the path at time $t\in [0,T]$. We denote the constant $0$ path on the interval $[0,t]$ by $0[t]$. 
    
    \subsubsection*{Measure theory}
    
    For $p \geq 1$, $(\cX, d)$ a metric space with Borel $\sigma$-algebra $\cB(\cX)$ and two probability measures $\mu, \nu \in \cP(\cX)$ we denote the $p$-\emph{Wasserstein distance} by
    \begin{equation*}
        \bW_d^{(p)} \big[ \mu, \nu \big] = \left( \inf_{\pi \in \Pi(\mu, \nu)} \int_{\cX \times \cX} d(x, y)^p d\pi(x, y) \right)^{\tfrac{1}{p}}
    \end{equation*}
    where $\Pi(\mu, \nu)$ is the set of all measures on $\Big( \cX \times \cX, \sigma \big( \cB(\cX) \times \cB(\cX) \big) \Big)$ with marginals $\mu$ and $\nu$. For a measure $\mu \in \cP(\cX)$ and $f:\cX \to \bR$ we denote
    \begin{equation*}
        \big\langle \mu, f \big\rangle = \int_{\cX} f(x) d\mu(x). 
    \end{equation*}

    For two measures $\mu, \nu \in \cP(\cX)$ we use $\nu << \mu$ to denote that $\nu$ is absolutely continuous with respect to $\mu$. We denote the \emph{relative entropy functional} by
    \begin{equation}
        \label{eq:RelativeEntropy}
        \bH\big[ \mu\big| \nu \big] = \left\{
        \begin{aligned}
            &\int_\cX \log\bigg( \frac{ d\mu}{d\nu}(x) \bigg) d\mu(x) 
            &\quad& \mbox{if}\quad
            \log\Big( \tfrac{ d\mu}{d\nu} \Big) \in L^1\big( \cX, \mu; \bR \big), 
            \\
            &\infty &\quad& \mbox{otherwise. }
        \end{aligned}
        \right.
    \end{equation}
    Sometimes, $\bH\big[ \mu \big| \nu \big]$ is referred to  as the \emph{Kullback–Leibler divergence}. 

	\section{Main Results}
    \label{section:MainResultsMim}

    The main contributions of this paper are three-fold. Firstly, we provide a Mimicking Theorem (Theorem \ref{theorem:fbm-Mimicking}) for stochastic differential equations driven by an additive fractional Brownian motion (see Definition \ref{definition:fBm} below). A similar result is also proved for a much richer class of stochastic differential equations with an additive Gaussian noise. Previous results rely heavily on semimartingale decompositions for the mimicked process, but our results are able to generalise these ideas to a much wider class of stochastic processes. Further generalisations of these results are discussed and proved in Section \ref{section:Mimicking}. 

    Secondly, as an application of this mimicking theorem we prove a sharp rate of convergence for \emph{quantitative propagation of chaos} for collections of interacting equations each driven by independent fractional Brownian motions (Theorem \ref{qPropChaosfBM}). These results are generalised and contextualised in Section \ref{section:QPoC}

    Thirdly, we derive and prove weak existence and uniqueness of a conditional McKean-Vlasov equation whose law is equal to the marginal dynamics of a collection of stochastic differential equations that interact locally with one another and are driven by independent additive fractional Brownian motions (see Theorem \ref{theorem:localEquation-fbm}). These results are expounded on in Section \ref{section:LocalEq}. 

    \subsection{Fractional Brownian motion}

    The focus of this work is for stochastic differential equations driven by an additive Gaussian noise, the most intuitive of these being \emph{fractional Brownian motion}:
    \begin{definition}
        \label{definition:fBm}
        A one dimensional \emph{fractional Brownian motion} $(Z_t)_{t\in [0,T]}$ with Hurst parameter $H\in (0, 1)$ is a centered Gaussian process with covariance defined for $t, s\in [0,T]$ by
        \begin{equation}
            \label{eq:covariance-fbm}
            \bE\Big[ Z_t Z_s \Big] = R(t, s) = \tfrac{1}{2} \Big( |t|^{2H} + |s|^{2H} - |t-s|^{2H} \Big). 
        \end{equation}
        A $d$-dimensional fractional Brownian motion is a centred Gaussian process with covariance
        \begin{equation*}
            \bE\Big[ \big\langle Z_t, Z_s \big\rangle \Big] = d \cdot R(t, s). 
        \end{equation*}
    \end{definition}
    A fractional Brownian motion is a continuous time stochastic process that satisfies $Z_0 = 0$ and is the unique centred Gaussian process that is self-similar and has stationary increments:
    \begin{equation*}
        Z_{\lambda t} \sim |\lambda|^H \cdot Z_t 
        \quad \mbox{and}\quad
        Z_t - Z_s \sim Z_{t-s}. 
    \end{equation*}

    The case $H=1/2$ corresponds to that of Brownian motion, but otherwise $Z$ is neither a Markov process nor a semimartingale. However, much like Brownian motion the process is $\bP$-almost surely $\alpha$-H\"older continuous for $\alpha < H$.

    A fractional Brownian motion can be written as a Gaussian Volterra process 
    \begin{equation*}
        Z_t = \int_0^t K(t, s) dW_s 
    \end{equation*}
    where $K(t, s)=0$ for $t<s$ and for $s<t$
    \begin{equation}
        \label{eq:fBmKernel}
        K(t, s) = \left\{
        \begin{aligned}
            &c_H \cdot s^{\tfrac{1}{2}-H} \int_s^t (r-s)^{H - \tfrac{3}{2}} r^{H-\tfrac{1}{2}} dr
            &\quad
            H>\tfrac{1}{2},
            \\
            &c_H\bigg( \Big( \frac{t(t-s)}{s} \Big)^{H-\tfrac{1}{2}} - \big( H- \tfrac{1}{2} \big) \cdot s^{\tfrac{1}{2}-H} \int_s^t r^{H-\tfrac{3}{2}} (r-s)^{H-\tfrac{1}{2}} dr \bigg)
            &\quad
            H<\tfrac{1}{2}
        \end{aligned}
        \right.
    \end{equation}
    where
    \begin{equation*}
        c_H = \left\{
        \begin{aligned}
            &\bigg( \frac{H(2H-1)}{\beta(2 - 2H, H- \frac{1}{2})} \bigg)^{\frac{1}{2}} 
            &\quad
            H>\tfrac{1}{2},
            \\
            &\bigg( \frac{2H}{(1 - 2H) \beta(1 - 2H, H + \frac{1}{2})}\bigg)^{\frac{1}{2}}
            &\quad
            H<\tfrac{1}{2} 
        \end{aligned}
        \right.
    \end{equation*}
    see for example \cite{nualart2006malliavin}*{Chapter 5}. Through direct computation, we can verify via the It\^o isometry that 
    \begin{align*}
        \bE\Big[ Z_t Z_s \Big] = \int_0^T K(t, u) K(s, u) du = R(t, s). 
    \end{align*}

    \subsubsection*{The fundamental Martingale}

    The fundamental martingale is concept first introduced in \cite{Norros1999Elementary} to understand the change of measure for stochastic differential equations driven by a fractional Brownian motion and was subsequently developed in \cites{Kleptsyna2000General, Kleptsyna2000Parameter} to study filtering problems for stochastic differential equations driven by fractional Brownian motion. We use a slightly different Volterra kernel (Equation \eqref{eq:fbm_Volterra} below) to the one stated in \cite{Norros1999Elementary} which allows us to map directly onto a Brownian motion. The Volterra kernel is somewhat more complicated, but the advantage is that we do not need to account for additional quadratic variation terms that arise in \cites{Kleptsyna2000General, Kleptsyna2000Parameter}. 

    It is commonly understood that a fractional Brownian motion is not a semimartingale; indeed fractional Brownian motion often chosen as the go-to stochastic process that is not a semimartingale. However, consider the following Volterra process driven by a fractional Brownian motion $Z$ (with Hurst parameter $H\in (0,1)$)
    \begin{equation*}
        W_t^* = \int_0^T L(t, s) dZ_s 
    \end{equation*}
    where $L(t, s)=0$ for $t<s$ and for $s<t$
    \begin{equation}
        \label{eq:fbm_Volterra}
        L(t, s) = \left\{
        \begin{aligned}
            &\Big( \frac{s(t-s)}{t} \Big)^{\tfrac{1}{2}-H} - (H-\tfrac{1}{2}) s^{\tfrac{1}{2} - H} \int_s^t (r-s)^{\tfrac{1}{2}-H} r^{H-\tfrac{3}{2}} dr
            &\quad
            H>\tfrac{1}{2},
            \\
            &s^{\tfrac{1}{2}-H} \int_s^t (r-s)^{-H-\tfrac{1}{2}} r^{H-\tfrac{1}{2}} dr
            &\quad
            H<\tfrac{1}{2}.
        \end{aligned}
        \right.
    \end{equation}
    The Volterra kernel $L:[0,T] \to \fWIC$ has the unique property that
    \begin{align*}
        \bE\Big[ W_t^* W_s^* \Big] = t\wedge s
    \end{align*}
    so that $W^*$ is a Brownian motion. We prove that the stochastic process $(W_t^*)_{t\in [0,T]}$ is a local-martingale with respect to the filtration $\bF^Z = (\cF_t^Z)_{t\in [0,T]}$ and further the two filtrations $\bF^Z$ and $\bF^{W^*}$ are equal. Therefore, while a fractional Brownian motion is not a semimartingale, it is a Volterra convolution away from a Brownian motion. We prove these results not just for fractional Brownian motion but also for a class of Gaussian processes in Section \ref{subsec:Filtering}. 

    The key insight of this paper is to exploit this proximity between a semimartingale and a stochastic differential equation with an additive fractional Brownian motion driving the dynamics to prove a \emph{Mimicking Theorem}. 
	
	\subsection{Projection and Mimicking Theorems}

    Let $(\Omega, \cF, \bP)$ be a complete probability space, let $Z$ be a fractional Brownian motion on $(\Omega, \cF, \bP)$ and let $\bF^Z = (\cF_t^Z)_{t\in [0,T]}$ be the filtration generated by $Z$ such that $\cF_0$ contains all $\bP$-null sets. Let $b:\Omega \times [0, T] \to \bR^d$ be some $\cF_t$-progressively measurable stochastic process and let $X_0$ be $\cF_0$-measurable. Consider a stochastic process of the form
    \begin{equation}
        \label{eq:brownianFilteringThm*}
		X_t = X_0 + \int_0^t b_s ds + Z_t. 
	\end{equation}
    The goal is to find stochastic differential equations whose law under the probability measure $\bP$ is equal to the law under $\bP$ of the stochastic process defined in Equation \eqref{eq:brownianFilteringThm*}. 

    \begin{theorem}
        \label{theorem:fbm-Mimicking}
        Let $(\Omega, \cF, \bP)$ be a probability space carrying a fractional Brownian motion $Z$ with Hurst parameter $H\in (0,1)$ and let $\bF = (\cF_t^Z)_{t\in [0,T]}$ be the filtration such that $\cF_t^Z = \sigma\big( Z_s: s\in [0,t] \big)$ such that $\cF_0^Z$ contains all $\bP$-null sets. Let $K:[0,T] \to L^2\big([0,T]; \bR\big)$ and $L:[0,T] \to \fWIC$ be the Volterra kernels defined in Equation \eqref{eq:fBmKernel} and Equation \eqref{eq:fbm_Volterra}. 
        
        Let $b:\Omega \times [0,T] \to \bR$ be progressively measurable and suppose that
        \begin{equation}
            \label{eq:theorem:fbm-Mimicking-b}
            \bE\bigg[ \Big\| \int_0^\cdot b_s ds \Big\|_{\RKHS_T} \bigg] < \infty.
        \end{equation}
        Let $X$ be a stochastic process of the form \ref{eq:brownianFilteringThm*} such that $X$ is the canonical process of the Wiener space. 
  
		Let $\bF^X = (\cF_t^X)_{t \in [0,T]}$ be the filtration generated by $X$. Then, there exists an extension 
		\begin{equation*}
			\big( \hat{\Omega}, \hat{\cF}, (\hat{\cF}_t)_{t\in [0,T]}, \hat{\bP} \big)
			\quad \mbox{of}\quad
			\big( \Omega, \cF, \bF^X, \bP \big)
		\end{equation*}
		carrying a fractional Brownian motion and a weak solution to the SDE 
		\begin{equation*}
			\hat{X}_t = \hat{X}_0 + \int_0^t K(t, s) \tilde{Q}\big( s, \hat{X}[s] \big) ds + \hat{Z}_t.
		\end{equation*}
		where
		\begin{align}
            \label{eq:theorem:fbm-Mimicking-Q}
			\tilde{Q}\big( t, X[t] \big) =& \bE\bigg[ \frac{d}{dt}\Big( \int_0^t L(t, s) b_s ds \Big) \bigg| X[t] \bigg]
            \\
            \nonumber
            =& (\tfrac{1}{2} -H) t^{H-\tfrac{1}{2} } \bE\bigg[ \int_0^t b_s \cdot (t-s)^{-(\tfrac{1}{2}+H)} s^{\tfrac{1}{2}-H} ds \bigg| X[t] \bigg]. 
		\end{align}
        Further, the law of $\hat{\bP} \circ (\hat{X})^{-1}$ is equal to $\bP \circ (X)^{-1}$. 
    \end{theorem}
    In fact, we prove a similar Theorem for a much wider class of Gaussian processes in Theorem \ref{brownianFilteringThm} (see Section \ref{subsection:FundamentWienerMim}). 

    \begin{remark}
        Looking ahead to the definition of $Q^b$ in Equation \eqref{Q} which we see present in \ref{eq:theorem:fbm-Mimicking-Q}. The existence of a measurable function of the form expressed in Equation \eqref{eq:theorem:fbm-Mimicking-Q} is by no means immediate but is derived from properties of the Volterra kernel. More details on this are given in Section \ref{subsec:Filtering} and also \cite{Hu2023Locally}. 
    \end{remark}

    \subsubsection*{Entropy estimates}

    Another favourable property of this representation for stochastic processes driven by fractional Brownian motion is that we can calculate the relative entropy between processes explicitly:    \begin{proposition}
        \label{fBmEntropy*}
        Let $(\Omega, \cF, \bP)$ be a Wiener space carrying a fractional Brownian motion with Hurst parameter $H \in (0,1)$. For $i=1,2$, let $P_0^i \in \cP( \bR^d )$ and let $b^i : [0,T] \times \Omega \to \bR^d$ be progressively measurable. Suppose that there exists a weak solution $P^i$ to the stochastic differential equations
        \begin{equation}
            dY^i_t = b^i\big( t, Y^i[t] \big) dt + dZ^i_t, \quad t \in [0, T] 
            \quad \mbox{with} \quad
            Y_0^i \sim P_0^i. 
        \end{equation}
        Additionally, suppose that
        \begin{align}
            \nonumber
            &P^i\bigg[ \Big\| \int_0^\cdot b_s^1 ds \Big\|_{\RKHS_T} < \infty, \quad  \Big\| \int_0^\cdot b_s^2 ds \Big\|_{\RKHS_T} < \infty \bigg] =1 \quad \mbox{and}\quad
            \bE^{P^i}\bigg[ \Big\| \int_0^\cdot b_s^1 - b_s^2 ds \Big\|_{\RKHS_T}^2 \bigg] < \infty. 
        \end{align}
        Then $P_0^1<< P_0^2$ implies that $P^1<<P^2$ and for $Y[t] \in \cC_t^d$,  
        \begin{align*}
            \frac{dP^1}{dP^2} \big( Y[t] \big)
            &= 
            \frac{dP_0^1}{dP_0^2}(Y_0) \cdot 
            \exp\Bigg( \delta\bigg( \int_0^{\cdot\wedge t} K(\cdot, s) Q^{b^1-b^2}\big(s, Y \big) ds \bigg) - \tfrac{1}{2} \bigg\| \int_0^{\cdot} (b^1 - b^2)\big( s, Y[s] \big) ds \bigg\|_{\RKHS_t}^2 \Bigg)
        \end{align*}
        where $\delta$ is the Malliavin divergence. 
        
        Further, suppose that $\bH\big[ P_0^1 \big| P_0^2 \big]< \infty$. Then for every $t\in [0,T]$
        \begin{equation*}
            \bH\Big[ P^1[t] \Big| P^2[t] \Big] = \bH\Big[ P_0^1 \Big| P_0^2 \Big] + \frac{1}{2} \bE\bigg[ \Big\| \int_0^{\cdot} b_s^1 - b_s^2 ds  \Big\|_{\RKHS_t}^2 \bigg]. 
        \end{equation*}
    \end{proposition}
    Proposition \ref{fBmEntropy*}, which we generalise in Proposition \ref{fBmEntropy}, has proved to be highly valuable for proving the relative entropy between two stochastic processes. We apply this in both Section \ref{section:QPoC} and Section \ref{section:LocalEq}. 

    \subsubsection*{Talagrand inequalities}
    
    Let $\big(\cX, \cB(\cX) \big)$ be a measurable space, let $d:\cX \times \cE \to \bR^+$ be a measurable metric and let $\mu \in \cP(\cX)$. We say that $\mu \in \cT_d^1(\cX)$ with constant $\gamma$ if for every $\nu \in \cP(\cX)$
    \begin{equation}
        \label{eq:talagrand}
        \bW_d^{(p)}\big[ \nu, \mu \big] \leq \sqrt{\gamma \cdot \bH\big[ \nu \big| \mu \big]}, \quad \mbox{for all $\nu \in \cP(\cX)$. }
    \end{equation}
    We denote the set of measures that satisfy Equation \eqref{eq:talagrand} by $\cT_d^p(\cX)$. 

    That Gaussian measures are $\cT_2(\cX)$ and the associated links to concentration of measure are now well documented \cite{Talagrand1996Transportation}. Let $(\cX, \cH, \mu)$ be an abstract Wiener space. Then for any measure $\nu \in \cP(\cX)$, 
    \begin{equation}
        \label{eq:Talagrand96}
        \inf_{\pi \in \Pi(\mu, \nu)} \int_{\cX \times \cX} d_{\cH}\big( x, x' \big)^2 \pi[dx, dx'] \leq 2\bH\big[ \nu\big| \mu \big]. 
    \end{equation}
    A reliable method for establishing that a measure $\mu \in \cT_d^1(\cX)$ with constant $\gamma$ which was first  proved in \cite{Djellout2004Transportation}*{Theorem 2.3} involves establishing that for every $\delta \in (0, \tfrac{1}{4\gamma})$, 
    \begin{equation}
        \label{eq:theorem:Djellout}
        \int_\cX \int_\cX \exp\Big( \delta \cdot d(x, y)^2 \Big) d\mu(x) d\mu(y)< \infty. 
    \end{equation}
    Typically, we will prove that there exists some $\delta>0$ such that
    \begin{equation*}
        \int_\cX \exp\Big( \delta \cdot d(x, 0)^2 \Big) d\mu(x), 
    \end{equation*}
    which implies that Equation \eqref{eq:theorem:Djellout} holds for some appropriately large choice of $\gamma$. 

    \begin{proposition}
        \label{proposition:TalagrandIQfbm}
        Let $(\Omega, \cF, \bP)$ be Wiener space carrying a fractional Brownian motion with Hurst parameter $H \in (0,1)$. 

        Let $P_0 \in \cT^1(\bR^d)$, let $b : [0,T] \times \cC_T^d \to \bR^d$ be progressively measurable and suppose that there exists $M \in L^2\big( [0,T]; \bR \big)$ such that 
        \begin{equation}
            \label{eq:proposition:TalagrandIQfbm}
            \Big| Q^b\big( t, X[t] \big) \Big| \leq M_t \cdot \Big( 1 + \| X \|_{\infty, t} \Big) \quad \bP\mbox{-almost surely. }
        \end{equation}
        Then there exists a unique weak solution to the SDE
        \begin{equation*}
            dX_t = b\big(t, X[t] \big) dt + dZ_t, \quad t\in [0,T] \quad X_0 \sim P_0, \quad Z \sim \gamma,
        \end{equation*}
        and for every $t\in [0,T]$ the weak solution satisfies that $P[t] \in \cT_{\infty, t}^1\big( \cC_t^d \big)$. 
    \end{proposition}
    Again, we prove Proposition \ref{proposition:TalagrandIQfbm} for a more general class of additive Gaussian processes in Proposition \ref{proposition:TalagrandIQ}. Working with measures that satisfy Equation \eqref{eq:talagrand} will be key to the techniques used in Sections \ref{section:QPoC} and \ref{section:LocalEq}. However, finding drift terms that satisfy \eqref{eq:proposition:TalagrandIQfbm} can be challenging:
    \begin{example}
        \label{example:Q}
        When $H<\tfrac{1}{2}$, suppose that there exists $M:[0,T] \to \bR$ such that
        \begin{align}
            \label{eq:example:Q-lin}
            &\Big| b\big(t, X[t] \big) \Big| \leq M_t \cdot \Big( 1 + \big\| X \big\|_{\infty, t} \Big)
            \quad \mbox{and}\quad
            \int_0^T \int_0^t (t-s)^{-(\tfrac{1}{2}+H)} |M_s| ds |M_t| dt< \infty. 
        \end{align}
        In particular, if $b$ satisfies a uniform in time linear growth condition then Equation \eqref{eq:example:Q-lin} follows. By adapting the techniques first detailed in \cite{Nualart2002Regularization}, we observe that
        \begin{align*}
            \Big| Q^b\big(s, X \big) \Big| =& \Big| s^{H-\tfrac{1}{2}} \cdot \int_0^s (s-r)^{-(\tfrac{1}{2}+H)} r^{\tfrac{1}{2}-H} b\big(r, X\big) dr \Big|
            \\
            \leq& \bigg( s^{H-\tfrac{1}{2}} \int_0^s (s-r)^{-\tfrac{1}{2}-H} r^{\tfrac{1}{2}-H} |M_r| dr \bigg) \cdot \Big( 1+ \big\| X \big\|_{\infty, s} \Big). 
        \end{align*}
        Further, by direct calculation we obtain that
        \begin{align*}
            \int_0^T \bigg| s^{H-\tfrac{1}{2}}& \int_0^s (s-r)^{-\tfrac{1}{2}-H} r^{\tfrac{1}{2}-H} |M_r| dr \bigg|^2 ds
            =
            \int_0^T s^{2H-1} \bigg( \int_0^s (s-r)^{-\tfrac{1}{2}-H} r^{\tfrac{1}{2}-H} \big| M_r \big| dr \bigg)^2 ds
            \\
            \leq& \int_0^T \int_0^r (r-u)^{-(\tfrac{1}{2}+H)} |M_u| |M_r| du dr < \infty. 
        \end{align*}
        As such, Equation \eqref{eq:proposition:TalagrandIQfbm} holds. 

        Now suppose that $H>\tfrac{1}{2}$. Consider a function of the form
        \begin{equation}
            \label{eq:example:Q-lin-}
            \left.
            \begin{aligned}
                &b\big(s, x[s] \big):=\overline{b}\big(s, x_s \big)
                \quad\mbox{where} \quad
                \overline{b}:[0,T] \times \bR^d \to \bR^d
                \quad \mbox{and}
                \\
                &\Big| \overline{b}\big(s, x\big) - \overline{b}\big(t, y\big) \Big| \leq C \Big( |t-s|^{\gamma} + |x - y|^{\alpha} \Big)
                \quad \mbox{for}\quad
                \alpha \in (1-\tfrac{1}{2H}, 1), \quad \gamma > H-\tfrac{1}{2}. 
            \end{aligned}
            \quad \right\}
        \end{equation}
        Again, by adapting the techniques first detailed in \cite{Nualart2002Regularization} we have that
        \begin{align*}
            \Big| Q^b\big(s, x[s] \big) \Big| 
            =& 
            \overline{b}(s, x_s) s^{\tfrac{1}{2}-H} 
            + (H-\tfrac{1}{2}) s^{H-\tfrac{1}{2}} \overline{b}(s, x_s) \int_0^s \Big( s^{\tfrac{1}{2}-H} - r^{\tfrac{1}{2}-H} \Big) (s-r)^{-(\tfrac{1}{2}+H)}dr 
            \\
            &+ (H-\tfrac{1}{2}) s^{H-\tfrac{1}{2}} \int_0^s \Big( \overline{b}(s, x_s) - \overline{b}(r, x_s) \Big) r^{\tfrac{1}{2}-H} (s-r)^{-(\tfrac{1}{2}+H)} dr
            \\
            &+ (H-\tfrac{1}{2}) s^{H-\tfrac{1}{2}} \int_0^s \Big( \overline{b}(r, x_s) - \overline{b}(r, x_r) \Big) (s-r)^{-(\tfrac{1}{2}+H)} r^{\tfrac{1}{2}-H} dr. 
        \end{align*}
        By applying Equation \eqref{eq:example:Q-lin-}, we conclude that for some $\varepsilon>0$ chosen small enough 
        \begin{align*}
            \Big| Q^b\big(s, x[s] \big) \Big| 
            \leq& 
            C_T s^{\tfrac{1}{2}-H} \Big( |x_0| + \big\| x\|_{\infty}^{\alpha} + |s|^{\gamma} + \big| \overline{b}(0, 0)\big| + \big\| x \big\|_{H-\varepsilon}^{\alpha} s^{\alpha(H-\varepsilon)} \Big). 
        \end{align*}
    \end{example}

	\subsection{McKean-Vlasov dynamics}

    Let $b_0: [0,T] \times \cC_T^d \rightarrow \bR^d$ and $b: [0,T] \times \cC_T^d \times \cC_T^d \rightarrow \bR^d$ be progressively measurable. First, consider the system of $n$-interacting (path-dependent) equations given by
	\begin{equation}
		\label{fBmParticleSystem}
		dX_t^{i,n} = \bigg( b_0 \big(t, X^{i,n}[t] \big) + \tfrac{1}{n-1}\sum_{\substack{j=1 \\ j \neq i}}^n b\big( t, X^{i,n}[t], X^{j,n}[t] \big) \bigg) dt + dZ_t^{i,n}, \quad i=1, ..., n.
	\end{equation}
	where $(Z_t^i)_{i=1, ..., n}$ is a family of exchangeable fractional Brownian motions. Let $P^{(n)} \in \cP\big( (\cC_T^d)^{\times n} \big)$ denote the law the entire system $\rx^n[T] = (X^i[T])_{i=1, ..., n}$ and $P^{(n, k)} \in \cP\big( (\cC_T^d)^{\times k} \big)$ be the marginal for the random variable $\rx^{n, k}[T] = \big( X^{i}[T] \big)_{i=1, ..., k}$ of the entire system $\rx^{n}[T]$. 
 
    The dynamics of systems of interacting equations such as Equation \eqref{fBmParticleSystem} are often modelled via the dynamics of the associated (path-dependent) McKean-Vlasov equation 
	\begin{equation}
		\label{fBmMV}
        dX_t = \bigg( b_0\big( t, X[t] \big) + \Big\langle \mu[t], b\big( t, X[t], \cdot \big) \Big\rangle \bigg) dt + dZ_t 
        \quad 
        \mu[t] = \bP \circ \big( X[t] \big)^{-1}. 
	\end{equation}
    Recent works studying such distributional dynamics driven additive noice include \cites{Han2022Solving,Galeati2022Distribution, Galeati2023Distribution,lacker2022Hierarchies}. 

    It is now well known that under various sets of assumptions on $b_0$ and $b$, as $n\to \infty$ the empirical distribution
	\begin{equation}
        \label{eq:ConvEmp-McV}
		\frac{1}{n} \sum_{i=1}^n \delta_{X^{i,n}[t]} \xrightarrow{\bP} \mu[t]. 
	\end{equation}
    Further, \eqref{eq:ConvEmp-McV} is equivalent (see \cite{Sznitman}) to the alternative limit
    \begin{equation}
        \label{eq:ConvEmp-McV:local}
        P^{(n,k)} \to \mu^{\otimes k}[t]. 
    \end{equation}
    Results of the form \eqref{eq:ConvEmp-McV} are referred to as \emph{global chaos} whereas results of the form \eqref{eq:ConvEmp-McV:local} are referred to as \emph{local chaos}. More focus has been paid to studying the \emph{global rate of convergence}, but the contribution of this paper is to prove a sharp upper bound of the \emph{local rate of convergence}. 

    \subsubsection*{Transform to the fundamental semimartingale}

    The dynamics of \eqref{fBmParticleSystem} and \eqref{fBmMV} can be transformed via a convolution with the Volterra kernel \eqref{eq:fbm_Volterra} into the stochastic process
    \begin{align*}
        X_t^{\dagger,i, n} =& X_0^{i, n} + \int_0^t L(t, s) dX_s^{i, n} 
        \\
        =& X_0^{i, n} + \int_0^t Q^{b_0}\big (s, X^{i, n}[s] \big)ds + \tfrac{1}{n-1} \sum_{\substack{j=1 \\ j\neq i}}^n \int_0^t Q^{b}\big( s, X^{i, n}[s], X^{j, n}[s] \big) ds + W_t^{i, n}
    \end{align*}
    and
    \begin{align*}
        X_t^{\dagger} =& X_0 + \int_0^t L(t, s) dX_s
        = X_0 + \int_0^t Q^{b_0}\big (s, X[s] \big)ds + \int_0^t \Big\langle \mu[s],  Q^{b}\big( s, X[s], \cdot \big) \Big\rangle ds + W_t
    \end{align*}
    where for any progressively measurable function $b: [0,T] \times \cC_T^d \to \bR^d$ we define $Q^b:[0,T] \times \cC_T^d \to \bR^d$ by
    \begin{equation}
    \label{eq:Q-transform_fbm}
        Q^b\big( t, x[t] \big) := (\tfrac{1}{2} - H) t^{H-\tfrac{1}{2}} \int_0^t (t-s)^{-(\tfrac{1}{2}+H)} b\big( s, x[s] \big)  s^{\tfrac{1}{2}-H} ds. 
    \end{equation}
    This choice of definition is motivated by the fact that
    \begin{equation*}
        \int_0^t Q^b\big( s, x[s] \big) ds = \int_0^t L(t, s) b\big( s, x[s] \big) ds
    \end{equation*}
    where the Volterra kernel $L:[0,T] \to \fWIC$ was defined in Equation \eqref{eq:fbm_Volterra}. 

    \subsubsection*{Quantitative Propagation of chaos}

    By transforming to the fundamental martingale and establishing key entropy estimates, we are able to adapt the techniques of \cite{lacker2022Hierarchies} to conclude the following:
    \begin{theorem}
		\label{qPropChaosfBM}
	    Suppose that there exists a  weak solution to the stochastic differential equation \eqref{fBmParticleSystem} whose law $P^n \in \cP\big( \cC_T^d \big)$ is exchangeable and suppose that there exists a weak solution to the McKean-Vlasov equation \eqref{fBmMV} with law $\mu \in \cP(\cC_T^d)$. 

        Further, suppose that 
        \begin{enumerate}  
			\item Square integrability:
			\begin{equation}
				\sup_{n\in \bN} \int_{(\cC_T^d)^{\times 2}} \bigg( \int_0^T \Big| Q^b\big(s, x[s], y[s] \big) - \Big\langle \mu, Q^b\big(s,x[s], \cdot \big) \Big\rangle \Big|^2 ds \bigg) P^{(n, 2)}(dx, dy) = M < \infty.
			\end{equation}
			\item Transport inequality. Let $\mu$ be the law of the solution of \eqref{fBmMV}. Then there exists a $\gamma >0$ such that for all $t \in [0,T]$, $x \in \cC_T^d$, and $\nu \in  \cP(\cC_T^d)$ such that $Q_t^b\big( x[t], \cdot \big) \in L^1(d\nu)$, we have
			\begin{equation}
				\bigg| \Big\langle \nu[t] - \mu[t], Q_t^b\Big( x[t], \cdot \Big) \Big\rangle \bigg|^2 
                \leq 
                \gamma\cdot \bH\big[ \nu[t] \big| \mu[t] \big].
			\end{equation}
			\item There exists $C_0>0$ such that
			\begin{equation}
				\bH\big[ P_0^{(n,k)} \big| \mu_0^{\otimes k} \big] \leq C_0 \frac{k^2}{n^2}.
			\end{equation}
		\end{enumerate}
        
		Then, there exists $C_H$ depending on the Hurst parameter $H$ such that for large enough $n$, $k \in \{1, \ldots n\}$, we have
		\begin{equation}
			\bH\big[ P^{(n, k)}[t] \big| \mu^{\otimes k}[t] \big] \leq 2C_H \frac{k^2}{n^2} + C \exp\bigg( - 2n(e^{-\gamma T} - \frac{k}{n})^2_+ \bigg).
		\end{equation}
	\end{theorem}

    \subsection{Marginal dynamics on $\kappa$-regular trees}

    Dynamics of the form described in Equation \eqref{fBmParticleSystem} can be rewritten for every $i =1, ..., n$ as
    \begin{equation*}
        dX_t^{i, n} = \bigg( b_0\big( t, X^{i, n}[t] \big) + \Big\langle \mu^{i, n}[t], b\big( t, X^{i, n}[t], \cdot \big) \Big\rangle \bigg) dt + dZ_t^{i, n}, 
        \quad \mu^{i, n}[t] = \tfrac{1}{n-1}\sum_{\substack{j=1 \\ j\neq i}}^n \delta_{X^{j, n}[t]}
    \end{equation*}
    where the (random) empirical measure $\mu^{i, n}[t] \in \cP_2(\cC_t^d)$. It is only through the dependency on the dynamics on the empirical measure $\mu^{i, n}[t]$ that the random variables $X^{i, n}$ and $X^{j, n}$ exhibit any correlation. Of course, we expect that as the number of interacting equations is taken to be very large, the empirical measure $\mu^{i, n}$ will converge $\bP$-almost surely to the (deterministic) mean-field limit $\mu[t]$ from Equation \eqref{fBmMV} and each equation will decouple from all other equations. However, for fixed $n\in \bN$ each SDE remains dependent on the dynamics of all other equations, a state that we visualise with the complete graph (where each vertex represents an SDE and each edge represents dynamic dependency). Complete graphs contain all possible edges between any two vertices. 
    
    In contrast to this, consider a collection of interacting SDEs represented by some sparse graph $(V, E)$. Informally, sparse graphs contain only a small proportion of the total number of possible edges. For each $u \in V$, we denote $N_u = \{v \in V: (u, v) \in E\}$ the neighbourhood of the vertex $u$. 

    As before, let $b_0:[0,T] \times \cC_T^d \to \bR^d$ and $b: [0,T] \times \cC_T^d \times \cC_T^d$ be progressively measurable functions. Now consider the system of (path-dependent) stochastic differential equations given by
    \begin{equation}
    	\label{eq:local-interaction}
        \left.
    	\begin{aligned}
        &dX_t^{u} = \bigg( b_0\big( t, X^u[t]\big) + \Big\langle \mu^u\big( X[t] \big), b\big(t, X^u[t], \cdot \big) \Big\rangle \bigg) dt + dZ_t^u
        \\
        &\mu^u\big( X[t] \big) = \tfrac{1}{|N_u|} \sum_{v\in N_u} \delta_{X^u[t]}. 
        \end{aligned}
        \right\}
        \quad u \in V
    \end{equation}
    For simplicity, let us turn our attention here to the case where the graph $(V, E)$ is the $\kappa$-regular graph. When $\kappa=2$, this corresponds to the graph with vertices indexed by the integers on the real line with edges corresponding to each integers nearest neighbours. Hence, for each integer $u \in \bZ$, the dynamics of the SDE $X^u$ are dependent on the dynamics of both $X^{u+1}$ and $X^{u-1}$. Such locally interacting dynamics are of increased interest in recent years due to their applications in opinion dynamics within social media, contagion modelling and the simulation of SPDEs from mathematical physics. For an incomplete list of references, we refer the reader to \cites{lacker2020marginal, Chatterjee2009Contact, Delattre2016Note, Bhamidi2021Survival}. 

    We could approximate the solution of this countably infinite collection of interacting SDEs via a hydrodynamic limit by considering the dynamics corresponding to a finite local neighbourhood of the $\kappa$-regular graph, but to get an accurate approximation we would need to take the finite local neighbourhood to be large. 

    Instead, our focus is on providing a characterisation for the marginal dynamics of some small local neighbourhood of the $\kappa$-regular graph. In the case where $\kappa=2$, we might be interested in the dynamics of the triple $(X^{-1}, X^{0}, X^{1})$ taking values in $(\cC_T^d)^{\times 3}$. 
    \begin{theorem}
    	\label{theorem:localEquation-fbm}
        Suppose that there exists a weak solution to the (countably infinite) collection of stochastic differential equations \eqref{eq:local-interaction}. 

        Let
        \begin{equation*}
			\gamma\Big( t, X^0[t], X^1[t] \Big) = \bE\bigg[ Q^b\Big( t, X^0[t],  \mu^0\big(X[t] \big) \Big) \bigg| X^0[t], X^1[t] \bigg]
			\end{equation*}
        Then there exists a unique weak solution to the stochastic differential equation
        \begin{equation}
            \label{eq:localEq-fbm}
            \left.
            \begin{aligned}
                X_t^{-1} =& X_0^{-1} + \int_0^t b_0\big(s, X^{-1}[s] \big) ds + \int_0^t K(t, s) \gamma\big(s, X^{-1}[s], X^{0}[s] \big) ds + Z_t^{-1}
                \\
                X_t^{0} =& X_0^{0} + \int_0^t b_0\big(s, X^{0}[s] \big) ds + \int_0^t b\Big(s, X^{0}[s], \mu^{0}\big( X[s] \big) \Big) ds + Z_t^{0}
                \\
                X_t^{1} =& X_0^{1} + \int_0^t b_0\big(s, X^{1}[s] \big) ds + \int_0^t K(t, s) \gamma\big(s, X^{1}[s], X^{0}[s] \big) ds + Z_t^{1}
            \end{aligned}
            \right\}
        \end{equation}
        Further, the weak solution to Equation \eqref{eq:localEq-fbm} agrees with the marginal $P^{\{-1, 0, 1\}}$ of the weak solution to Equation \eqref{eq:local-interaction}. 
    \end{theorem}
    
	\section{Gaussian processes and the fundamental Martingale}
	\label{subsec:Filtering}
	
	For a useful introduction to some of the ideas of \emph{calculus of variations for Gaussian processes} (also called Malliavin calculus), we refer the reader to \cite{nualart2006malliavin}. 

    In this section, we explore how a Gaussian process can be ``diagonalised'' or ``inverted'' to a martingale. The existence of such a transformation is not necessary for any Gaussian process, but we will only consider Gaussian processes for which such a transformation exists. 

    We want to work with a class of Gaussian Volterra processes that can be transformed in a measurable bijecting way into a Brownian motion. 
	
	\subsection{Gaussian Volterra processes}
 
	\label{subsection:GVP}

	Let us start by considering a class of Gaussian processes that can be represented as a measurable transformation of a Brownian motion:
    \begin{definition}
        Let $d\in \bN$ and let $(\Omega, \cF, \bF, \bP)$ be a complete filtered probability space carrying a $d$-dimensional Brownian motion. 

        Let $K:[0,T] \to L^2\big( [0,T]; \lin(\bR^d, \bR^d) \big)$ be a Volterra kernel (that is $K(t, s) = 0$ for $s>t$). We say that a stochastic process is a \emph{Gaussian Volterra process} if it can be expressed as
        \begin{equation}
    		\label{eq:GaussVolterra}
    		Z_t = \int_0^t K(t, s) dW_s
	    \end{equation}
        where $W$ is a Brownian motion and $K$ is a Volterra kernel. 
    \end{definition} 
    In particular, this means that the covariance function for a Gaussian processes of the form Equation \eqref{eq:GaussVolterra} satisfy
	\begin{equation}
        \label{eq:R-covariance}
		R(s, t) = \bE\Big[ \big\langle Z_t , Z_s \big\rangle \Big] = \int_0^{t \wedge s} \Big\langle K(t, u), K(s, u)\Big\rangle_{\lin(\bR^d, \bR^d)} du. 
	\end{equation}

    In order to illustrate the key properties of the Gaussian processes that we consider for driving signals, we need to consider three distinct Hilbert spaces:
    \begin{enumerate}[label=(\roman*)]
        \item The \emph{reproducing kernel Hilbert space} is the closure of the collection of covariance functions 
        \begin{align*}
            &\RKHS_T:=\spn\bigg\{ \bE\Big[ \langle Z_t, u\rangle_{\bR^d} Z_{\cdot} \Big]: t\in [0,T], u \in \bR^d \bigg\}
            \\
            &\mbox{with inner product} \quad 
            \Big\langle \bE\big[ \langle Z_t, u\rangle Z_{\cdot} \big], \bE\big[ \langle Z_s, v\rangle Z_{\cdot} \big] \Big\rangle_{\RKHS_T} = \bE\Big[ \langle Z_t, u\rangle  \cdot \langle Z_s, v \rangle \Big]. 
        \end{align*}
        \item The \emph{first Wiener-Ito chaos} is the closure of the collection of step functions
        \begin{align*}
            &\fWIC_T:=\spn\bigg\{ \1_{[0,t]}(\cdot) e_{i,i}: t\in [0,T], i\in \{1, ..., d\} \bigg\}
            \\
            &\mbox{with inner product}\quad
            \Big\langle \1_{[0,t]} e_{i,i}, \1_{[0,s]} e_{j,j} \Big\rangle_{\fWIC_T} = \bE\Big[ \langle Z_t, e_i\rangle \cdot \langle Z_s, e_j\rangle \Big]. 
        \end{align*}
        \item The \emph{Volterra space} is the closure of the collection of Volterra kernels
        \begin{equation*}
            \cV_T:=\spn\Big\{ K(t, \cdot): t\in [0,T] \Big\}
            \quad \mbox{with inner product}\quad
            \Big\langle K(t, \cdot), K(s, \cdot) \Big\rangle_{\cV_T} = R(t, s). 
        \end{equation*}
    \end{enumerate}
    We have the following Hilbert space isometric isomorphisms
    \begin{equation*}
        \scI: \fWIC_T \to \RKHS_T, 
        \quad
        \scJ: \cV_T \to \fWIC_T, 
        \quad 
        \scK: \cV_T \to \RKHS_T, 
    \end{equation*}
    defined by
    \begin{equation*}
        \scI \big[ \1_{[0,t]}(\cdot) I_d \big](s) = R(t, s), 
        \quad
        \scJ\big[ K(t, \cdot) \big](s) = \1_{[0,t]}(s) I_d, 
        \quad
        \scK\big[ K(t, \cdot) \big](s) = R(t, s), 
    \end{equation*}
    and $\scI \circ \scJ = \scK$. 

	\subsection{The fundamental martingale}
	
	The following Assumption on the Volterra kernel will be central to this work:
	\begin{assumption}
		\label{assumption:VolterraK}
		Let $K:[0,T] \to L^2\big([0,T]; \lin(\bR^d, \bR^d) \big)$ be a Volterra kernel. Suppose that there exists another Volterra kernel $L:[0,T] \to \fWIC$ such that for any $s, t\in [0,T]$, 
		\begin{equation}
			\label{eq:assumption:VolterraK}
			\scJ^*\big[ L(t, \cdot) \big](s) =  \1_{[0,t]}(s) I_d. 
		\end{equation}
	\end{assumption}
    As the closure of the Hilbert space
    \begin{equation*}
        \spn\Big\{ \1_{[0,t]} : t\in [0,T] \Big\}
        \quad \mbox{with inner product} \quad
        \Big\langle \1_{[0,t]}, \1_{[0,s]} \Big\rangle = t\wedge s
    \end{equation*}
    is equal to $L^2\big( [0,T]; \lin(\bR^d, \bR^d) \big)$, we observe that Assumption \ref{assumption:VolterraK} is equivalent to the Volterra space $\cV_T$ being isomorphic to $L^2\big( [0,T]; \lin(\bR^d, \bR^d) \big)$. 

    \begin{proposition}
        \label{proposition:RKHS=I^*}
        Let $(\Omega, \cF, \bF, \bP)$ be a complete filtered probability space carrying a $d$-dimensional Brownian motion, let $K:[0,T] \to L^2\big([0,T]; \lin(\bR^d, \bR^d) \big)$ be a Volterra kernel that satisfies Assumption \ref{assumption:VolterraK} and let $Z$ be a Gaussian Volterra process of the form \eqref{eq:GaussVolterra}. 

        Then for any $t\in [0,T]$ the reproducing kernel Hilbert space is isometrically isomorphic to
        \begin{equation}
            \label{eq:prop:LND-filtration-H}
            \RKHS_t = \bigg\{ \int_0^\cdot K(\cdot, s) h_s ds: h\in L^2\big( [0,t]; \bR^d \big) \bigg\} \subseteq \cC_{0, T}^d. 
        \end{equation}
        We define
		\begin{equation}
            \label{eq:proposition:Fundamental-Wiener}
			W_t^* = \int_0^t L(t, s) dZ_s. 
		\end{equation}
		Then $W_t^*$ is an $\bF$ Brownian motion. We refer to $W^*$ as the \emph{fundamental Brownian motion}. 
    \end{proposition}
    For a proof of \ref{proposition:RKHS=I^*}, see \cite{Hu2023Locally}. 
    
    \begin{example}
        Recalling the Volterra kernel defined in Equation \eqref{eq:fBmKernel}, we recall the result originally proved in \cite{Decreusfond1999Stochastic} that the reproducing kernel Hilbert space for Fractional Brownian motion can be written as
        \begin{equation*}
            \RKHS_t = 
            \bigg\{ \int_0^\cdot K(\cdot, s) f_s ds : \quad f\in L^2\big( [0,t]; \bR^d \big) \bigg\} \subseteq \cC_T^d
        \end{equation*}
        where $K:[0,T] \to L^2\big([0,T]; \bR^d \big)$ is the Volterra kernel as defined in Equation \eqref{eq:fBmKernel}. 
    \end{example}

    \subsubsection*{Secure local non-determinism}

    \begin{definition}
        \label{definition:Filtration-Hilbert}
        Let $(I, \leq)$ be a totally ordered set and suppose that for every $i \in I$ we have that $\big(\cH_i, +, \langle \cdot, \cdot\rangle_{i} \big)$ is a Hilbert space over a common field $\bF$. We say that $(\cH_i)_{i\in I}$ is a \emph{filtration of Hilbert spaces} if
        \begin{equation*}
            \mbox{for every $i\leq j$}\quad \cH_i \subseteq \cH_j \quad \mbox{and}\quad \forall f,g \in \cH_i \quad \langle f, g \rangle_j = \langle f, g \rangle_i. 
        \end{equation*}

        Two Hilbert space filtrations $(\cH_i)_{i\in I}$ and $(\cG_i)_{i \in I}$ are said to be isomorphic if for every $j \in I$ there exists a Hilbert space isomorphism
        \begin{equation*}
            \Psi^j: \big( \cH_j, \langle \cdot, \cdot\rangle_j \big)
            \to
            \big(\cG_j, \langle, \cdot, \cdot \rangle_j \big)
            \quad\mbox{such that}
            \quad \forall i \leq j \quad
            \Psi^j\big|_{\cH_i} = \Psi^i. 
        \end{equation*}
    \end{definition}

    We include the following result from \cite{Hu2023Locally} which 
    \begin{theorem}
        \label{thm:LND-filtration}
        Let $(Z_t)_{t\in [0,T]}$ be a $d$-dimensional Gaussian process with covariance $R:[0,T]^{\times 2} \to \bR$. Let
        \begin{align*}
            &\RKHS_t:= \spn \Big\{ \bE\big[ \langle Z_s, u\rangle_{\bR^d} \cdot Z_{\cdot} \big]: s\in [0,t], u\in \bR^d \Big\}
            \\
            &\mbox{with inner product}\quad
            \Big\langle \bE\big[ \langle Z_s, u\rangle_{\bR^d} \cdot Z_{\cdot} \big], \bE\big[ \langle Z_r, v\rangle_{\bR^d} \cdot Z_{\cdot} \big] \Big\rangle_{\RKHS_t} = \bE\Big[ \langle Z_s, u\rangle_{\bR^d} \cdot \langle Z_r, v\rangle_{\bR^d} \Big]. 
        \end{align*}
        Then the following are equivalent:
        \begin{enumerate}
            \item 
            \label{enum:thm:LND-filtration-1}
            The Hilbert space filtration
            \begin{equation}
                \label{EQ:thm:LND-filtration}
                \Big( \RKHS_t, \big\langle \cdot, \cdot \big\rangle_{\RKHS_t} \Big)_{t\in [0,T]} 
                \quad\mbox{is isomorphic to} \quad
                \bigg( \Big\{ f \cdot \1_{[0,t]}: f \in L^2\big( [0,T]; \bR\big) \Big\}, \big\langle \cdot, \cdot \big\rangle_t \bigg)_{t\in [0,T]}; 
            \end{equation}
            \item 
            \label{enum:thm:LND-filtration-2}
            There exists a Volterra kernel $K:[0,T] \to L^2\big( [0,T]; \lin(\bR^d, \bR^d) \big)$ such that
            \begin{equation*}
                Z_t = \int_0^t K(t, s) dW_s
            \end{equation*}
            and $K$ satisfies Assumption \ref{assumption:VolterraK};
        \end{enumerate}
        We say that Gaussian processes with a Hilbert space filtration of reproducing kernels that satisfies Equation \eqref{EQ:thm:LND-filtration} is said to be \emph{securely locally non-deterministic}. 
    \end{theorem}

    \subsubsection*{The Filtration of Gaussian processes}

    We should interpret Proposition \ref{proposition:RKHS=I^*} as proof of a bijection between the reproducing kernel Hilbert space of a Gaussian Volterra process (that satisfies Assumption \ref{assumption:VolterraK}) and the reproducing kernel Hilbert space of Brownian motion. The upshot of the existence of this bijection is that we can define another Gaussian Volterra process that has the same covariance as Brownian motion (the fundamental Brownian motion). Our next step is to show that this new process is in fact \emph{the underlying Brownian motion} with which the Gaussian process is defined. To establish this, we need to consider the $\sigma$-algebra filtration generated by these two processes:
    \begin{definition}
        \label{definition:canonical}
        Let $\big( \Omega, \cF, \bF, \bP\big)$ be a complete filtered probability space carrying a $d$-dimensional Brownian motion. Let $K:[0,T] \to L^2\big( [0,T]; \lin(\bR^d, \bR^d) \big)$ be a Volterra kernel and let $Z$ be a Gaussian Volterra process of the form \eqref{eq:GaussVolterra}. Let $(\cF_t^Z)_{t\in [0,T]}$ where $\cF_t^Z = \sigma\big( Z_s: s\in [0,t]\big)$. 
        
        Then the Volterra kernel is said to be \emph{canonical} if for every $t\in [0,T]$ the $\sigma$ algebra $\cF_t^Z = \cF_t$.
    \end{definition}
    A Gaussian Volterra process of the form \eqref{eq:GaussVolterra} necessarily exists on a probability space that carries a Brownian motion. Under Assumption \ref{assumption:VolterraK}, there is a measurable transformation of the process $Z$ to $W^*$ which is also a Brownian motion. Here we address whether the \emph{fundamental Brownian motion} is the same as the \emph{underlying Brownian motion}:	
	\begin{proposition}
        \label{proposition:filtration1}
		Let $\big( \Omega, \cF, \bF, \bP \big)$ be a filtered probability space carrying a $d$-dimensional Brownian motion. Let $K:[0,T] \to L^2\big([0,T]; \lin(\bR^d, \bR^d) \big)$ be a Volterra kernel that satisfies Assumption \ref{assumption:VolterraK} and let $Z$ be a Gaussian Volterra process of the form \eqref{eq:GaussVolterra}.  
		
		Then the filtration generated by the Brownian motion $\bF=(\cF_t^W)_{t\in [0,T]}$ is the same as the filtration $(\cF_t^Z)_{t\in [0,T]}$ generated by $Z$ up to sets of measure 0. 

        Hence, every Gaussian Volterra process that satisfies Assumption \ref{assumption:VolterraK} is canonical. 
	\end{proposition}
	
	\begin{proof}
		It should be clear that for each choice of $t\in [0,T]$ the random variable
		\begin{equation*}
			Z_t = \int_0^t K(t, s) dW_s 
			\quad \mbox{is measurable with respect to} \quad
			\cF_t^W. 
		\end{equation*}
		Therefore, we focus on proving the reverse implication. 
		
		First of all, we have that
        \begin{equation*}
    		\int_0^t h(s) dZ_s = \int_0^t \scJ^*\big[ h \big](s) dW_s. 
    	\end{equation*}
        Therefore, we have that for any two  elements $f, g\in \fWIC$
		\begin{equation*}
			\bE\Big[ \int_0^T f(s) dZ_s \cdot \int_0^T g(s) dZ_s \Big]
			=
			\int_0^T \scJ^*\big[f\big](s) \cdot \scJ^*\big[ g\big](s) ds
		\end{equation*}
		Thanks to Assumption \ref{assumption:VolterraK}, we have that a Volterra kernel $L:[0,T] \to \fWIC$ that satisfies Equation \eqref{eq:assumption:VolterraK} exists. Thus for each choice of $t\in [0,T]$ the random variable
		\begin{equation*}
			W_t^* = \int_0^t L(t, s) dZ_s 
			\quad \mbox{is measurable with respect to} \quad
			\cF^Z_t. 
		\end{equation*}
		Therefore, we have that $\cF_t^{W^*} \subseteq \cF_t^{Z} \subseteq \cF_t^W$. 
		
		Now we show that the two Gaussian processes $W$ and $W^*$ are $\bP$-almost surely equal. For any $t\in [0,T]$, 
		\begin{align*}
			&\bE\big[ W_t \cdot W_t^* \big] = \bE\bigg[ \int_0^T \1_{[0,t]}(u) dW_u \cdot \int_0^T L(t, u) dZ_u \bigg]
			\\
			&= \bE\bigg[ \int_0^T \1_{[0,t]}(u) dW_u \cdot \int_0^T \scJ^*\big[ L(t, \cdot) \big](u) dW_u \bigg]
			\\
			&= \bE\bigg[ \int_0^T \1_{[0,t]} (u) dW_u \cdot \int_0^T \1_{[0,t]} (u) dW_u \bigg]
			= \int_0^T \1_{[0,t]} (u) du = t
		\end{align*}
		so that
		\begin{align*}
			\bE\Big[ \big| W_t^* - W_t \big|^2 \Big] &= \bE\Big[ \big| W_t^*\big|^2 \Big] + \bE\Big[ \big| W_t \big|^2 \Big] - 2\bE\Big[ W_t^* \cdot W_t \Big]
			\\
			&= 2t - 2t = 0. 
		\end{align*}
		Therefore $\cF_t^{W^*} = \cF_t^W$ and we conclude that $\cF_t^Z = \cF_t^W$. 
	\end{proof}

    Having considered Gaussian Volterra processes which have the property (Equation \eqref{eq:assumption:VolterraK}), we want to consider a class of stochastic process analogous to semimartingales for which we can embed the solution of our stochastic differential equations:
	\begin{definition}
        \label{definition:Q}
		Let $\big( \Omega, \cF, \bF, \bP \big)$ be a complete filtered probability space supporting a $d$-dimensional $\bF$-Brownian motion $W$. Let $K:[0,T] \to L^2\big([0,T]; \lin(\bR^d, \bR^d) \big)$ be a Volterra kernel that satisfies Assumption \ref{assumption:VolterraK}. 
		
		Let $b:\Omega \times [0,T] \to \bR^d$ be progressively measurable, suppose that
        \begin{equation*}
            t \mapsto \int_0^t L(t, s) b_s ds
        \end{equation*}
        is absolutely continuous and the integral
		\begin{equation}
			\label{eq:definition:Fund-SemiMart}
			\bE\bigg[ \int_0^T \Big| \frac{d}{dt} \int_0^t L(t, s) b_s ds \Big|^2 dt \bigg] < \infty. 
		\end{equation}
		
		We define $Q^b: \Omega \times [0,T] \to \bR^d$ by
		\begin{equation}
			\label{Q}
			Q^b_t := \frac{d}{dt} \int_0^t L(t, s) b_s ds. 
		\end{equation}
	\end{definition}
    
    As first proved in \cite{Hu2023Locally}, Equation \eqref{eq:definition:Fund-SemiMart} is equivalent to
    \begin{equation}
        \label{eq:lemma:Existence_Q}
        \Big\| \int_0^\cdot b_s ds \Big\|_{\cH_t}< \infty
    \end{equation}
    and
    \begin{equation}
        \label{eq':lemma:Existence_Q}
        \int_0^t b_s ds = \int_0^t K(t, s) Q_s^b ds
        \quad \mbox{and}\quad 
        \Big\| \int_0^\cdot b_s ds \Big\|_{\RKHS_t} = \bigg( \int_0^t \big| Q_s^b \big|^2 ds \bigg)^{\tfrac{1}{2}}. 
    \end{equation}
    
    \begin{proposition}
        \label{proposition:ProgressiveMeasurability}
        Let $\big( \Omega, \cF, \bP \big)$ be a complete probability carrying a $d$-dimensional Brownian motion and let $K:[0,T] \to L^2\big([0,T]; \lin(\bR^d, \bR^d) \big)$ be a Volterra kernel that satisfies Assumption \ref{assumption:VolterraK}. Let $Z$ be the Gaussian Volterra process of the form Equation \eqref{eq:GaussVolterra} and let $\bF = (\cF_t^Z)_{t\in [0,T]}$ be the filtration generated by $Z$ such that $\cF_0$ contains all $\bP$-null sets. 

        Let $b:\Omega \times [0,T] \to \bR^d$ and suppose that Equation \eqref{eq:lemma:Existence_Q} holds. Then $b$ is progressively measurable if and only if $Q^b$ is progressively measurable. 
    \end{proposition}

    \begin{proof}
        Thanks to Proposition \ref{proposition:filtration1}, we know that the filtration $\bF^Z = (\cF_t^Z)_{t\in [0,T]}$ is the same as the filtration $\bF^{W} = (\cF_t^W)_{t\in [0,T]}$, so by choosing $\cF_0^Z = \cF_0$ we obtain $\bF^Z = \bF$. 

        First suppose that $b$ is progressively measurable. Then each $t\in [0,T]$, we have that the mapping $[0,T] \times \Omega \ni (t, \omega) \mapsto b_t(\omega)$ is $\mbox{Borel}([0,t]) \otimes \cF_t$ measurable so that
        \begin{equation*}
            [0,T] \times \Omega \ni (t, \omega) \mapsto \int_0^t L(t, s) b_s(\omega) ds 
            \quad \mbox{is} \quad 
            \mbox{Borel}([0,t]) \otimes \cF_t\quad \mbox{-measurable.}
        \end{equation*}
        Hence by Equation \eqref{Q}
        \begin{equation*}
            [0,T] \times \Omega \ni (t, \omega) \mapsto \int_0^t Q_s^b(\omega) ds 
            \quad \mbox{is} \quad 
            \mbox{Borel}([0,t]) \otimes \cF_t\quad \mbox{-measurable.}
        \end{equation*}
        and we conclude that $Q_t^b$ is $\mbox{Borel}([0,t]) \otimes \cF_t$-measurable. 

        On the other hand suppose that $Q_t^b$ is progressively measurable and define
        \begin{equation*}
            \tilde{b}:\Omega \times [0,T] \to \bR^d 
            \quad \mbox{by}\quad
            \tilde{b}(\omega, t) := \frac{d}{dt} \int_0^t K(t, s) Q_s^b ds. 
        \end{equation*}
        A similar argument yields that $\tilde{b}$ is also progressively measurable, and we conclude from Equation \eqref{eq':lemma:Existence_Q} that $\bP \times \mbox{Lebesgue}$ almost everywhere that $\tilde{b}(\omega, s) = b(\omega, s)$. 
    \end{proof}

	\section{Mimicking Theorems for fractional SDEs}
    \label{section:Mimicking}
	
	Following the ideas of \cite{Brunick2013Mimicking}, we prove a mimicking theorem for stochastic processes with additive securely locally non-deterministic Gaussian noises where the drift term is a stochastic process on the associated reproducing kernel Hilbert space.  
	
	In section \ref{subsection:Shreve}, we recall the results of \cite{Brunick2013Mimicking}. In Section \ref{subsection:FundamentWienerMim}, we introduce the fundamendal semimartingale associated to SDEs with additive Gaussian noise and exploit this process to prove a Mimicking theorem (see Theorem \ref{brownianFilteringThm}). In Section \ref{subsec:EntropyEstimates}, we provide an entropy estimates for stochastic differential equations driven by Gaussian processes and prove that the laws of such processes satisfy a Talagrand inequality. 
	
	\subsection{Pathwise Mimicking Theorem for Martingales}
    \label{subsection:Shreve}

	\begin{proposition}
		\label{kavitaFilteringLemma}
		Let $\big( \cC_T^d, \cB(\cC_T^d), \bF, \bP \big)$ be a probability space  supporting an $\bF$-Brownian motion $W$ taking values in $\bR^d$. Let $b:[0,T] \times \cC_T^d \to \bR^d$ be progressively measurable, let $X_0:\Omega \to \bR^d$ be $\cF_0$-measurable and suppose that 
		\begin{equation}
			\label{eq:kavitaFilteringLemma}
			\bE\Big[ \int_0^T \big| b_s \big|^2 ds \Big] < \infty, 
			\quad 
			\bE\Big[ \big| X_0 \big|^2 \Big] < \infty. 
		\end{equation}
		
		We define the stochastic process 
		\begin{equation}
			\label{eq:kavitaFilteringLemma-X}
			X_t = X_0 + \int_0^t b_s ds + W_t. 
		\end{equation}
		
		Let $\check{b}: [0,T] \times \cC_T^d \rightarrow \bR^d$ be the progressively measurable functions satisfying
		\begin{equation*}
			\check{b}(t, X[t]) = \bE\Big[ b_t \Big| X[t] \Big], \quad \mbox{ $\bP$-a.s., for a.e. } t \in [0,T].
		\end{equation*}
		Let $\cF^{X}=\big( \cF_{t}^{X} \big)_{t \geq 0}$ denote the filtration generated by $X$. 
		
		Then there exists an extension $\big( \check{\Omega}, \check{\cF}, (\check{\cF}_t)_{t\in [0,T]}, \check{\bP} \big)$ of the probability space $\big( \Omega, \cF, (\cF_t^{X})_{t\in [0,T]}, \bP \big)$ supporting a $\check{\cF}$-Martingale $\check{W}$ such that the stochastic differential equations
		\begin{equation}
			\check{X}_t = \int_0^t \check{b}\big( t, \check{X}[t] \big) ds + \check{W}_t
            \label{eq:filteredEqn}
		\end{equation}
		admits the weak solution $\big( (\check{\Omega}, \check{\cF}, (\check{\cF}_t)_{t\in [0,T]}, \check{\bP}), \check{W}, X \big)$ whose law is equal to that of $X_t$ defined in Equation \eqref{eq:kavitaFilteringLemma-X}. 
	\end{proposition}
	
	\begin{proof}
		Let $C_{c}^{\infty}(\bR^d)$ denote the set of smooth functions on $\bR$ with compact support. Write $\nabla$ and $\nabla^{2}$ for the gradient and Hessian operators, respectively. By Itô's formula and Equation \eqref{eq:kavitaFilteringLemma}, for each $\varphi \in C_{c}^{\infty}(\bR)$ the process
		\begin{equation*}
			\varphi(X_t)-\int_{0}^{t} \Big( b_r \cdot \nabla \varphi( X_r ) + \tfrac{1}{2} \nabla^{2} \varphi(X_r) \Big) dr
		\end{equation*}
		is a $\cF$-martingale. In particular, if $t>s$, and if $Z$ is any bounded $\mathcal{F}_{s}$-measurable random variable then
		\begin{equation*}
			0 = \bE\left[ Z\left(\varphi(X_t)-\varphi(X_s)-\int_{s}^{t} \Big(b_r \cdot \nabla \varphi(X_r)+\tfrac{1}{2} \nabla^{2} \varphi(X_r) \Big) dr \right) \right].
		\end{equation*}
		Suppose that $Z$ is measurable with respect to $\cF_{s}^{X} \subset \cF_{s}$. Then we may use Fubini's theorem and the tower property of conditional expectations to obtain
		\begin{equation*}
			0 = \bE\left[ Z \left(\varphi(X_t) - \varphi(X_s) - \int_{s}^{t}\left( \widetilde{b} \big( r, X[r] \big) \cdot \nabla \varphi(X_r)+\tfrac{1}{2}  \nabla^{2} \varphi(X_r) \right) dr \right) \right].
		\end{equation*}
		This shows that the process
		\begin{equation*}
			\varphi(X_t) - \int_{0}^{t}\left(\widetilde{b}\big( r, X[r] \big) \cdot \nabla \varphi( X_r )+\tfrac{1}{2} \nabla^{2} \varphi(X_r) \right) d r
		\end{equation*}
		is a $\mathbb{F}^{X}$-martingale, for every $\varphi \in C_{c}^{\infty}\left(\mathbb{R}^{d}\right)$. We now follow the usual construction of weak solutions from solutions to martingale problems (using the arguments in \cite{karatzasShreve}*{Proposition 5.4.6 and Theorem 3.4.2}). By \cite{karatzasShreve}*{Proposition 5.4.6}, there exists an extended probability space $\big( \check{\Omega}, \check{\cF}, (\check{\cF}_t)_{t\in [0,T]}, \check{\bP} \big)$ such that the canonical process satisfies \eqref{eq:filteredEqn}.
	\end{proof}
	
	\subsection{The fundamental semimartingale process}
    \label{subsection:FundamentWienerMim}

    Firstly, we introduce the setting within which we will be working:
    \begin{assumption}
        \label{assumption:Mimicking}
        Let $\big( \Omega, \cF, \bF, \bP \big)$ be a complete probability carrying a $d$-dimensional Brownian motion and let $K:[0,T] \to L^2\big([0,T]; \lin(\bR^d, \bR^d) \big)$ be a Volterra kernel that satisfies Assumption \ref{assumption:VolterraK}. 
		
		Let $X_0:\Omega \to \bR^d$ be an $\cF_0$-measurable random variable and let $Z$ be the Gaussian Volterra process of the form Equation \eqref{eq:GaussVolterra}. Let $b:\Omega \times [0,T] \to \bR^d$ be progressively measurable and suppose that for every $t\in [0,T]$
        \begin{equation*}
            \bE\bigg[ \Big\| \int_0^\cdot b_s ds \Big\|_{\RKHS_t}^2 \bigg] < \infty. 
        \end{equation*}
    \end{assumption}
    
    Our goal is to transform a stochastic process driven by a Gaussian Volterra process of the form \eqref{eq:GaussVolterra} into a stochastic process driven by a Brownion motion is such a way that the transformed process becomes a semimartingale. 
    \begin{definition}
        Suppose that Assumption \ref{assumption:Mimicking} holds. For any stochastic process of the form
		\begin{equation}
            \label{eq:X}
			X_t = X_0 + \int_0^t b_s ds + Z_t,  
		\end{equation}
		we define
		\begin{align}
            \nonumber
			\newX_t :=& X_0 + \int_0^t L(t, s) dX_s
			= X_0 + \int_0^t L(t, s) b_s ds + \int_0^t L(t, s) dZ_s
            \\
            \label{eq:newX}
            =& X_0 + \int_0^t Q_s^b ds + W_t^*. 
		\end{align}
    \end{definition}
    The key insight of this paper is that stochastic processes of the form \eqref{eq:X} can be identified bijectively with processes of the form \eqref{eq:newX} and that the key properties for proving a \emph{Mimicking Theorem} (see Theorem \ref{brownianFilteringThm} below) also translate. 
    
	\begin{proposition}
        \label{proposition:Filtrations-X,newX}
		Suppose that Assumption \ref{assumption:Mimicking} holds. Let $X$ be a stochastic process of the form \eqref{eq:X} and let $\newX$ be defined as in Equation \eqref{eq:newX}. Then filtration $(\cF_t^X)_{t\in [0, T]}$ is equal to the filtration $(\cF_t^{\newX})_{t\in [0,T]}$ and $\newX$ is an $\bF$-adapted semimartingale. 
	\end{proposition}

	\begin{proof}
        We have already proved in Proposition \ref{proposition:filtration1} that the filtrations $\bF = \bF^Z$, and
        \begin{equation*}
            \int_0^t Q^b_s ds = \int_0^t L(t, s) b_s ds
            \quad \mbox{and}\quad
            W_t^* = \int_0^t L(t, s) dZ_s. 
        \end{equation*}
        Hence the function $t\mapsto \int_0^t L(t, s) b_s(\omega) ds$ is $\bP$-almost surely absolutely continuous. Further, by Proposition \ref{proposition:ProgressiveMeasurability} and the assumption that $b$ is progressively measurable, we conclude that the map
        \begin{equation*}
            (t, \omega) \mapsto \int_0^t L(t, s) b_s(\omega) ds
        \end{equation*}
        is c\`adl\`ag adapted process of locally bounded variation. Further, by Proposition \ref{proposition:RKHS=I^*} the process $W_t^*$ is a local martingale with respect to the filtration $\bF$ and we conclude that $\newX$ is an $\bF$-semimartingale. 

        Arguing as before, firstly we have that
        \begin{equation*}
            \newX_t = X_0 + \int_0^t L(t, s) dX_s 
            \quad \mbox{is measurable with respect to}\quad 
            \cF_t^X. 
        \end{equation*}
        Next, we define the random variable
        \begin{equation*}
            X_t^* := X_0 + \int_0^t K(t, s) d\newX_s 
            \quad \mbox{is measurable with respect to} \quad \cF_t^{\newX}. 
        \end{equation*}
        Therefore we have that $\cF_t^{X^*} \subseteq \cF_t^{\newX} \subseteq \cF_t^X$. 

        Proposition \ref{proposition:RKHS=I^*} implies that $\bP$-almost surely for every $t\in [0,T]$
        \begin{align*}
            \int_0^t b_s(\omega) ds = \int_0^t K(t, s) Q_s^b(\omega) ds 
            \iff
            &\int_0^t K(t, s) Q_s^b(\omega) ds = \int_0^t b_s(\omega) ds
        \end{align*}
        so that the stochastic process
        \begin{align*}
            X_t^* =& X_0 + \int_0^t K(t, s) Q_s^b ds + \int_0^t K(t, s) dW_s
            = X_0 + \int_0^t b_s ds + Z_t
        \end{align*}
        and for every $t\in [0,T]$
        \begin{equation*}
            \bE\Big[ \big| X_t - X_t^* \big|^2 \Big] = 0. 
        \end{equation*}
        Hence for every $t\in [0,T]$ we have that $\cF_t^X = \cF_t^{X^*}$ and we conclude. 
	\end{proof}

    The following Theorem is the key result of this work and implies Theorem \ref{theorem:fbm-Mimicking}:
	\begin{theorem}
		\label{brownianFilteringThm} 
        Suppose that Assumption \ref{assumption:Mimicking} holds. Let $X$ be a stochastic process of the form \eqref{eq:X} and let $\bF^X = (\cF_t^X)_{t \in [0,T]}$ be the filtration generated by $X$. 
		
		Then, there exists an extension 
		\begin{equation*}
			\big( \hat{\Omega}, \hat{\cF}, \hat{\bF}, \hat{\bP} \big)
			\quad \mbox{of}\quad
			\big( \Omega, \cF, (\cF_t^X)_{t\in [0,T]}, \bP \big)
		\end{equation*}
		carrying a Gaussian process with the same covariance as $Z$ and a weak solution to the SDE
		\begin{equation}
			\label{eq:brownianFilteringThm}
			\hat{X}_t = \hat{X}_0 + \int_0^t K(t, s) \tilde{Q}\big( s, \hat{X}[s] \big) ds + \hat{Z}_t.
		\end{equation}
        where
		\begin{equation}
			\label{eq:brownianFilteringThm-} 
			\tilde{Q}\big( t, X[t] \big) = \bE\Big[ Q^b_t \Big| X[t] \Big], 
		\end{equation}
		and $Q^b:[0,T] \times \Omega \to \bR^d$ is defined in \eqref{Q}. 
        
		Further, the law $\hat{\bP} \circ (\hat{X})^{-1}$ is equal to $\bP \circ (X)^{-1}$. 
	\end{theorem}

	\begin{proof} 
		Under Assumption \ref{assumption:Mimicking}, we define $\newX$ as in Equation \eqref{eq:newX} and by Proposition \ref{proposition:Filtrations-X,newX} $\newX$ is a $\bF^X$-semimartingale. 
		
		By Proposition \ref{kavitaFilteringLemma}, we can find a filtered probability space $\big( \hat{\Omega}, \hat{\cF}, (\hat{\cF}_t)_{t\in [0,T]}, \hat{\bP} \big)$ carrying a Brownian motion such that a weak solution to the SDE
		\begin{equation}
			\label{eq:brownianFilteringThm-pf} 
			\newX_t = X_0 + \int_0^t \bE\Big[ Q^b_s \Big| \newX[s] \Big] ds + W_t
		\end{equation}
		exists and has the same distribution as the stochastic process defined in Equation \eqref{eq:brownianFilteringThm}. 
		
		To conclude, we again apply Proposition \ref{proposition:Filtrations-X,newX} to observe that the stochastic process
		\begin{equation*}
			X_t = X_0 + \int_0^t K(t, s) d\newX_s = \int_0^t K(t, s) \bE\Big[ Q^b_s \Big| \newX[s] \Big] ds + Z_t. 
		\end{equation*}
        under $\hat{\bP}$ has the same distribution as $X$ under $\bP$. 
        
	    Thanks to Proposition \ref{proposition:Filtrations-X,newX}, we have that any measurable function of a path $\newX[t]$ can be written as another measurable function of $X[t]$. Thus for clear exposition, we denote the progressively measurable function $\check{b}:[0,T] \times \cC_T^d \to \bR^d$ defined by
		\begin{equation*}
			\check{Q}\big( t, \newX[t] \big) = \bE\Big[ Q_t^b \Big| \newX[t] \Big]
		\end{equation*}
		and (in order to define \eqref{eq:brownianFilteringThm-}) we define the progressively measurable function $\tilde{Q}: [0,T] \times \cC_T^d \to \bR^d$ by
		\begin{equation*}
			\tilde{Q}\big( t, X[t] \big) = \bE\Big[ Q_t^b\Big| X[t] \Big] = \check{Q}\Big( t, \newX[t] \Big). 
		\end{equation*}
		Putting this all together, we get the stochastic differential equation \eqref{eq:brownianFilteringThm}. As a solution to this is equivalent to a solution to Equation \eqref{eq:brownianFilteringThm-pf}, we conclude that a weak solution exists and has the same distribution as \eqref{eq:X}. 
	\end{proof}

    \subsection{Entropy and transport estimates}
    \label{subsec:EntropyEstimates}

    The additive Gaussian noises in the system of interacting equation \eqref{fBmParticleSystem} and the associated McKean Vlasov equation \eqref{fBmMV} ensure that the associated laws on pathspace have many favourable properties. In this section, we illustrate the properties that will be key to sharpening the rate of convergence for propagation of chaos. 

    In particular, a neat obersveration is that the ideas and techniques of this section are martingale in their nature, but the results apply to stochastic processes that are not semimartingales. 
    
    \subsubsection*{Entropy estimate}

    First, we establish some results for the distributions of SDEs with additive Gaussian noise:
    \begin{proposition}[Entropy Estimates]
        \label{fBmEntropy}
        Let $d \in \bN$ and let $i=1,2$. Let $(\cC_T^d, \RKHS_T, \gamma)$ be an abstract Wiener space carrying a Gaussian Volterra process with Volterra kernel $K:[0,T] \to L^2\big( [0,T]; \lin(\bR^d, \bR^d) \big)$ that satisfies Assumption \ref{assumption:VolterraK}.

        Let $b^i : [0,T] \times \cC_T^d \to \bR^d$ be progressively measurable and suppose that a weak solution to the SDE
        \begin{equation}
            dY^i_t = b^i\big( t, Y^i[t] \big) dt + dZ^i_t, \quad t \in [0, T] 
            \quad \mbox{with} \quad
            Y_0^i \sim P_0^i, \quad Z^i \sim \gamma
        \end{equation}
        exists and denote them by $P^i[t] \in \cP(\cC_t^d)$. Next, suppose that
        \begin{align}
            &P^i\bigg[ \Big\| \int_0^\cdot b_s^1 ds\Big\|_{\RKHS_T} < \infty, \quad \Big\| \int_0^\cdot b_s^2 ds \Big\|_{\RKHS_T} < \infty \bigg] =1, 
            \quad
            \label{eq:fBmEntropy-1}
            \bE^{P^i}\bigg[ \Big\| \int_0^\cdot b_s^1 - b_s^2 ds \Big\|_{\RKHS_T}^2 \bigg] < \infty. 
        \end{align}
        If $P_0^1<< P_0^2$, then $P^1<<P^2$ and for $Y[t] \in \cC_t^d$,  
        \begin{align}
            \nonumber
        	&\frac{dP^1}{dP^2} \big( Y[t] \big) 
            \\
            \label{ap:def:girsanov}
            &= 
            \frac{dP_0^1}{dP_0^2}(Y_0) \cdot 
            \exp\Bigg( \delta\bigg( \int_0^{\cdot\wedge t} K(\cdot, s) Q^{b^1-b^2}\big(s, Y \big) ds \bigg) - \tfrac{1}{2} \bigg\| \int_0^{\cdot} (b^1 - b^2)\big( s, Y[s] \big) ds \bigg\|_{\RKHS_t}^2 \Bigg)
        \end{align}
        where $\delta$ is the Malliavin divergence. 
        
        Further, suppose that $\bH\big[ P_0^1 \big| P_0^2 \big]< \infty$. Then
        \begin{equation*}
            \bH\Big[ P^1[t] \Big| P^2[t] \Big] = \bH\Big[ P_0^1 \Big| P_0^2 \Big] + \frac{1}{2} \bE^{P^1}\bigg[ \Big\| \int_0^{\cdot} b_s^1 - b_s^2 ds  \Big\|_{\RKHS_t}^2 \bigg]. 
        \end{equation*}
    \end{proposition}
    
    In particular, Proposition \ref{fBmEntropy} implies Proposition \ref{fBmEntropy*}:
    \begin{proof}
        We denote by $P^{\dagger,1}, P^{\dagger,2} \in \cP(\cC_T^d)$ the law of the stochastic process
        \begin{align*}
            Y_t^{\dagger,i} = Y_0^i + \int_0^t L(t, s) dY_s^i = Y_0^i + \int_0^t Q_s^{b^i}\big( Y^{i}[s] \big) ds + W_t^i, \quad t\in [0,T]. 
        \end{align*}
        Clearly, $W$ is a local martingale and the process $\int_0^t Q^b_s(Y^i) ds$ is c\`adl\`ag, locally bounded variation and adapted to the filtration $\cF_t^Y$. Thanks to Proposition \ref{proposition:Filtrations-X,newX}, the filtrations $\bF^Y$ and $\bF^{Y^{\dagger}}$ are equal and both are contained in the filtration $\bF$. 
        
        Thus the dynamics $Y^{\dagger, i}$ are a semimartingale and can be written as a measurable function of the dynamics $Y^{i}$. Similarly, 
        \begin{align*}
            Y_t^{i} = Y_0^i + \int_0^t K(t, s) dY_s^{\dagger, i} = Y_0^i + \int_0^t b^i\big(s, Y^{i} \big) ds + Z_t^i, \quad t\in [0,T],  
        \end{align*}
        so that the dynamics of $Y^i$ can be written as a measurable function of the dynamics of $Y^{\dagger, i}$. By the data processing inequality and the bijective nature of the measurable maps that transform $P^{i}$ to $P^{\dagger, i}$ and back, we conclude that
        \begin{align*}
            \bH\big[ P^{\dagger,1}[t] \big| P^{\dagger,2}[t] \big] \leq \bH\big[ P^1[t] \big| P^2[t]\big] 
            \quad \mbox{and}\quad
            \bH\big[ P^1[t] \big| P^2[t]\big] \leq  \bH\big[ P^{\dagger,1}[t] \big| P^{\dagger,2}[t] \big]
        \end{align*}
        so that 
        \begin{equation}
            \label{eq:fBmEntropy-Proof1}
            \bH\big[ P^{\dagger,1}[t] \big| P^{\dagger,2}[t] \big] = \bH\big[ P^1[t] \big| P^2[t]\big]. 
        \end{equation}
        
        As $P^{\dagger, i}$ are the laws of semimartingales and Equation \eqref{eq:fBmEntropy-1} holds, we apply \cite{Leonard2012Girsanov}*{Theorem 2.3} to get that
        \begin{equation}
            \label{eq:fBmEntropy-Proof2}
            \bH\big[ P^{\dagger,1}[t] \big| P^{\dagger,2}[t] \big] = \bH\big[ P_0^1 \big| P_0^2 \big] + \frac{1}{2} \bE\bigg[ \int_0^t \Big| Q^{b^1}\big(s, Y^{1} \big) - Q^{b^2}\big(s, Y^{1} \big) \Big|^2 ds \bigg]. 
        \end{equation}
        To conclude, we just combine Equation \eqref{eq:fBmEntropy-Proof1} with Equation \eqref{eq:fBmEntropy-Proof2} and apply Equation \eqref{eq':lemma:Existence_Q}. 
    \end{proof}
    
    \subsubsection*{Talagrand Inqualities}

    Our interest in Talagrand inequalities is derived from the fact that in Theorem \ref{qPropChaosfBM} we estimate the relative entropy between the interactive particle system and the associated McKean-Vlasov equation. A Talagrand inequality allows us to link this directly to the rate of convergence of the Wasserstein distance between the corresponding laws. 

    \begin{proposition}
        \label{proposition:TalagrandIQ}
        Let $(\cC_{0,T}^d, \RKHS_T, \gamma)$ be an abstract Wiener space carrying a Gaussian Volterra process with Volterra kernel $K:[0,T] \to L^2\big( [0,T]; \lin(\bR^d, \bR^d) \big)$ that satisfies Assumption \ref{assumption:VolterraK}.

        Suppose that \emph{either}:
        \begin{enumerate}[label=(A.\roman*)]
            \item 
            \label{enum:proposition:TalagrandIQ-1}
            The measure $P_0 \in \cT^1(\bR^d)$ and $b : [0,T] \times \cC_T^d \to \bR^d$ is progressively measurable and satisfies that there exists $M \in L^1\big( [0,T]; \bR \big)$ such that
            \begin{equation*}
                \big| b(t, X) - b(t, Y) \big| \leq M_t \cdot \big\| X - Y \big\|_{\infty, t}
                \quad \mbox{and}\quad
                \int_0^t \big| b(s, 0\big) \big| ds < \infty; 
            \end{equation*}
            \item 
            \label{enum:proposition:TalagrandIQ-2}
            The measure $P_0 \in \cP(\bR^d)$ and there exists $\epsilon>0$ such that
            \begin{equation}
                \label{eq:proposition:TalagrandIQ-init}
                \int_{\bR^d} \exp\Big( \epsilon \big| x \big|^2 \Big) P_0(dx) < \infty. 
            \end{equation}
            Further, let $b : [0,T] \times \cC_T^d \to \bR^d$ is progressively measurable and satisfies that there exists $M \in L^2\big( [0,T]; \bR \big)$ such that 
            \begin{equation}
                \label{eq:proposition:TalagrandIQ}
                \Big| Q^b\big( t, X[t] \big) \Big| \leq M_t \Big( 1 + \big\| X \big\|_{\infty, t} \Big);
            \end{equation}
        \end{enumerate}
        Then there exists a unique weak solution to the SDE
        \begin{equation}
            \label{eq:proposition:TalagrandIQ-exuq}
            dX_t = b\big(t, X[t] \big) dt + dZ_t, \quad t\in [0,T] \quad X_0 \sim P_0, \quad Z \sim \gamma,
        \end{equation}
        and for every $t\in [0,T]$ the weak solution satisfies that $P[t] \in \cT_{\infty, t}^1\big( \cC_t^d \big)$. 
    \end{proposition}

    \begin{proof}
        The proof of weak existence and uniqueness is established in \cite{Hu2023Locally} and we focus our attention on proving that a Talagrand inequality is satisfied. 
        
        Firstly, under \ref{enum:proposition:TalagrandIQ-1} we remark that for the product probability space $\Omega \times \tilde{\Omega}$ carrying two independent solutions to Equation \eqref{eq:proposition:TalagrandIQ-exuq}, we have that
        \begin{align*}
            \big\| X - \tilde{X} \big\|_{\infty, t} \leq \big| X_0 - \tilde{X}_0 \big| + \big\| Z - \tilde{Z} \big\|_{\infty, t} + \int_0^t M_s \cdot \big\| X - \tilde{X} \big\|_{\infty, s} ds
        \end{align*}
        Applying Gr\"onwall's inequality, we conclude that
        \begin{equation*}
            \big\| X - \tilde{X} \big\|_{\infty, t} \leq \bigg( \big| X_0 - \tilde{X}_0 \big| + \big\| Z - \tilde{Z} \big\|_{\infty, t} \bigg) \cdot \exp\Big( \int_0^t M_s ds \Big) 
        \end{equation*}
        so that
        \begin{equation*}
            \big\| X - \tilde{X} \big\|_{\infty, t}^2 \leq 2\bigg( \big| X_0 - \tilde{X}_0 \big|^2 + \big\| Z - \tilde{Z} \big\|_{\infty, t}^2 \bigg) \cdot \exp\Big( 2 \int_0^t |M_s| ds \Big) 
        \end{equation*}
        Therefore, 
        \begin{align*}
            \bE\times \tilde{\bE}\bigg[ \exp\bigg( \delta \big\| X - \tilde{X} \big\|_{\infty, t}^2 \bigg) \bigg] 
            \leq& 
            \tfrac{1}{2} \cdot \bE\times \tilde{\bE}\bigg[ \exp\bigg( \delta \cdot 4 \exp\Big( 2\int_0^t |M_s| ds \Big) \cdot \big| X_0 - \tilde{X}_0 \big|^2 \bigg) \bigg]
            \\
            &+ \tfrac{1}{2} \cdot \bE\times \tilde{\bE}\bigg[ \exp\bigg( \delta \cdot 4 \exp\Big( 2\int_0^t |M_s| ds\Big) \cdot \big\| Z - \tilde{Z} \big\|_{\infty, t}^2 \bigg) \bigg]
        \end{align*}
        and $P_0 \in \cT^1(\bR^d)$ and $Z$ is Gaussian allows us to conclude. 
        
        Secondly, under \ref{enum:proposition:TalagrandIQ-2} we have that
        \begin{align*}
            \big\| X - \tilde{X} \big\|_{\infty, T} \leq \big| X_0 - \tilde{X}_0 \big| + \big\| Z - \tilde{Z} \big\|_{\infty, T} + \bigg\| \int_0^{\cdot} b\big( s, X\big) - b\big(s, \tilde{X} \big) ds \bigg\|_{\RKHS_T}
        \end{align*}
        so that
        \begin{align*}
            \big\| X - \tilde{X} \big\|_{\infty, T}^2 \leq 3\bigg( \big| X_0 - \tilde{X}_0 \big|^2 + \big\| Z - \tilde{Z} \big\|_{\infty, T}^2 + \bigg\| \int_0^{\cdot} b\big( s, X\big) - b\big(s, \tilde{X} \big) ds \bigg\|_{\RKHS_T}^2 \bigg)
        \end{align*}
        and
        \begin{align*}
            \exp\Big( \delta \cdot \big\| X - \tilde{X} \big\|_{\infty, T}^2 \Big) 
            \leq& 
            \frac{\exp\Big( 9 \delta \cdot \big| X_0 - \tilde{X}_0 \big|^2 \Big)}{3}
            +
            \frac{\exp\Big( 9 \delta \cdot \big\| Z - \tilde{Z} \big\|_{\infty, T}^2 \Big)}{3}
            \\
            &+
            \frac{\exp\bigg( 9 \delta \cdot \bigg\| \int_0^\cdot b\big(s, X) - b\big(s, \tilde{X} \big) ds \bigg\|_{\RKHS_T}^2 \bigg)}{3}. 
        \end{align*}
        Equation \eqref{eq:proposition:TalagrandIQ-init} implies that $P_0 \in \cT^1(\bR^d)$ and $\gamma$ is Gaussian so we conclude via Fernique's Theroem that there exists a constant $C>0$ such that for every $\delta <C$, 
        \begin{equation*}
            \bE\times \tilde{\bE}\bigg[ \exp\Big( \delta \cdot \big| X_0 - \tilde{X}_0 \big|^2 \Big) \bigg]< \infty 
            \quad \mbox{and}\quad
            \bE\times \tilde{\bE}\bigg[ \exp\Big( \delta \cdot \big\| Z - \tilde{Z} \big\|_{\infty, T}^2 \Big) \bigg]< \infty. 
        \end{equation*}
        Thanks to Equation \eqref{eq:proposition:TalagrandIQ} and Equation \eqref{eq':lemma:Existence_Q}, we have that
        \begin{align*}
            \bigg\| \int_0^{\cdot} b\big(s, X\big) - b\big( s, \tilde{X} \big) ds \bigg\|_{\RKHS_T}^2 =& \int_0^T \Big| Q^b\big(s, X \big) - Q^b\big(s, \tilde{X} \big) \Big|^2 ds
            \\
            \leq& 2\int_0^T |M_t|^2 dt \cdot \Big( 2+ \big\| X \big\|_{\infty, T}^2 + \big\| \tilde{X} \big\|_{\infty, T}^2 \Big). 
        \end{align*}
        To conclude, note that
        \begin{align*}
            \big\| X \big\|_{\infty, t}^2 \leq& 3 \Big( \big| X_0 \big|^2 + \big\| Z \big\|_{\infty, t}^2 + \int_0^t \big| Q^b(s, X) \big|^2 ds \Big)
            \\
            \leq& 3 \bigg( \big| X_0 \big|^2 + \big\| Z \big\|_{\infty, t}^2 + 2 \int_0^t |M_s|^2 \cdot \Big( 1+ \big\| X \big\|_{\infty, s}^2 \Big) ds \bigg)
        \end{align*}
        so that by Gr\"onwall's inequality
        \begin{equation*}
            \big\| X \big\|_{\infty, t}^2 \leq 3\Big( \big| X_0 \big|^2 + \big\| Z \big\|_{\infty, t}^2 + 2\int_0^t |M_s|^2 ds \Big) \exp\Big( 2\int_0^t |M_s |^2 ds \Big). 
        \end{equation*}
        Therefore, for an appropriately small choice of $\delta$ we conclude that
        \begin{equation*}
            \bE\Big[ \exp\Big( \delta \cdot \big\| X \big\|_{\infty, t}^2 \Big) \Big]< \infty
        \end{equation*}
        and we conclude that for $\delta$ small enough
        \begin{equation*}
            \bE\Bigg[ \exp\bigg( \delta \cdot \bigg\| \int_0^{\cdot} b\big(s, X \big) - b\big(s, \tilde{X} \big) ds \bigg\|_{\RKHS_T}^2 \bigg) \Bigg] < \infty. 
        \end{equation*}
        Hence, for $\delta$ small enough
        \begin{equation*}
            \bE\times \tilde{\bE}\bigg[ \exp\Big( \delta \cdot \big\| X - \tilde{X} \big\|_{\infty, T}^2 \Big) \bigg] < \infty
        \end{equation*}
        and we conclude that $P \in \cT_{\infty}^1(\cC_T^d)$. 
    \end{proof}
	
	\section{Quantitative propagation of chaos}
    \label{section:QPoC}

    We now prove \emph{quantitative propagation of chaos} results for collections of interacting stochastic processes with additive Gaussian noises. In Section \ref{subsec:OriginalRes}, we highlight the previous techniques used to prove quantitative propagation of chaos for systems of interacting stochastic processes in \cite{lacker2022Hierarchies}. In Section \ref{subsec:QPoC}, we establish some necessary results and prove Theorem \ref{theorem:qPropChaos}. Finally, in Section \ref{subsection:QPoCExample} we provide an example to show that the rate of convergence we establish is optimal. 
    
	\subsection{Strategy for interacting equations}
    \label{subsec:OriginalRes}

    Let us now turn our attention to a collection of $n$ (path dependent) SDEs with addititive independent Brownian motions that are coupled together via a dependency on the empirical measure
	\begin{equation}
		\label{danParticleSystem}
		dX_t^{i, n} = \bigg( b_0\big( t, X^{i,n}[t] \big) + \tfrac{1}{n-1}\sum_{\substack{j=1 \\j \neq i}}^n b\big( t, X^{i,n}[t], X^{j,n}[t] \big) \bigg) dt + dW_t^{i,n}, \quad i = 1, ..., n.
	\end{equation}
	On the other hand, the associated dynamics of the McKean-Vlasov SDE that describes the limiting marginal dynamics as $n\to \infty$ is given by
	\begin{equation}
		\label{danMV}
		dX_t = \bigg( b_0\big( t, X[t] \big) + \Big\langle \mu[t], b\big( t, X[t], \cdot \big) \Big\rangle \bigg) dt + dW_t \quad \mu[t] = \bP \circ \big( X[t] \big)^{-1}
	\end{equation}

    \begin{definition}
        \label{definition:well-posed}
        Let $n\in \bN$ and let $\overline{b}:[0,T] \times (\cC_T^d)^{\times n} \to (\bR^d)^{\times n}$ be progressively measurable. We say that the stochastic differential equation
        \begin{equation*}
            dX_t = \overline{b}\big( t, X[t] \big) dt + dW_t
        \end{equation*}
        is \emph{well-posed} if for every $(t_0, z) \in [0, T) \times (\cC_T^d)^{\times n}$, there exists a unique in law solution to the stochastic differential equation
        \begin{equation*}
            dX_t = \overline{b}\big( t, X[t] \big) dt + dW_t, \quad t\in (t_0, T] 
            \quad \mbox{and}\quad
            X_t = z_t \quad t\in [0,t_0]. 
        \end{equation*}
    \end{definition}
    
	The main result of \cite{lacker2022Hierarchies} establishes an optimal rate of convergence for Equation \eqref{eq:ConvEmp-McV}:
	\begin{theorem}[\cite{lacker2022Hierarchies}*{Theorem 2.2}]
        Suppose that there exists a weak solution to \eqref{danParticleSystem} which we denote by $P^{(n)} \in \cP\big( (\cC_T^d)^{\times n} \big)$ and additionally suppose that $P^{(n)}$ is exchangeable. Suppose that there exists a weak solution to the McKean-Vlasov equation \eqref{danMV} with law $\mu \in \cP\big( C_T^d \big)$. 
  
        Further, suppose that
		\begin{enumerate}
			\item Suppose that for every $k\in \bN$ and $(t_0, z) \in [0,T] \times (\cC_T^{d})^{\times k}$, there exists a unique in law weak solution to the stochastic differential equation
            \begin{equation*}
                d\overline{\rx}_t^{k} = \overline{b}^{\otimes k}\big( t, \overline{\rx}^{k}[t] \big) dt + dW_t^{k} \quad t\in (t_0, T] 
                \quad \mbox{and} \quad 
                \overline{X}^k_t = z_t \quad \mbox{for} \quad t\in [0, t_0]
            \end{equation*}
            where $\overline{\rx}^k[t] = \big( \overline{X}^{i,k}[t] \big)_{i=1, ..., k}$
			\begin{align*}
				\overline{b}^{\otimes k}\big( t, \overline{\rx}^k[t] \big) = \bigotimes_{i=1}^k \bigg( b_0\big( t, \overline{X}^{i,k}[t] \big) + \Big\langle \mu[t], b\big( t, \overline{X}^{i, k}[t], \cdot \big) \Big\rangle \bigg).
			\end{align*}
			\item Square Integrability: 
			\begin{equation}
                \label{danSqInt}
				M := \sup_{t \in [0, T]} \int_{\cC_T^d \times \cC_T^d} \Big| b(t, x, y) - \big\langle \mu, b(t, x, \cdot) \big\rangle \Big|^2 P^{(2)}(dx,dy) < \infty
			\end{equation}
			\item There exists $\gamma>0$ such that for any $t\in (0, T)$, $x \in \cC_T^d$, $\nu \in \cP( \cC_T^d)$ such that $b(t, x[t], \cdot) \in L^1(\nu)$, 
			\begin{equation}
				\label{danTransport}
				\Big| \Big\langle \mu - \nu, b\big( t, x[t], \cdot \big) \Big\rangle \Big|^2 \leq \gamma \cdot \bH\big[ \nu | \mu],
			\end{equation}
			\item There exists $C_0>0$ such that
			\begin{equation}
				\bH\big[ P_0^{(n,k)} \big| \mu_0^{\otimes k} \big] \leq C_0 \frac{k^2}{n^2}.
			\end{equation}
		\end{enumerate}
        Then for any $n \geq 6\exp(\gamma T)$ and $k \in \{1, ..., n\}$, we have that
        \begin{equation}
            \bH \big[ P^{(n,k)}[T] \big| \mu^{\otimes k}[T] \big] \leq 2C\frac{k^2}{n^2} + C \exp \bigg( -2n(e^{-\gamma T - \frac{k}{n}})^2_+\bigg)
        \end{equation}
        where
        \begin{equation*}
            C:= 8\Big( C_0 + (1+\gamma)MT \Big) \exp\big( 6\gamma T \big)
        \end{equation*}
	\end{theorem}

    \subsection{Quantitative propagation of chaos}
    \label{subsec:QPoC}

    We now return our attention to the main result of this paper, a quantitative propagation of chaos result for collections of interacting stochastic differential equations of the form
    \begin{equation}
		\label{eq:danParticleSystem_Z}
		dX_t^{i, n} = \bigg( b_0\big( t, X^{i,n}[t] \big) + \tfrac{1}{n-1}\sum_{\substack{j=1 \\j \neq i}}^n b\big( t, X^{i,n}[t], X^{j,n}[t] \big) \bigg) dt + dZ_t^{i,n}, \quad i = 1, ..., n.
	\end{equation}
	where $(Z^{i, n})$ are exchangeable Gaussian processes each of the form \eqref{eq:GaussVolterra}. As in the Brownian setting, we expect the limiting marginal dynamics as $n\to \infty $ to converge to the McKean-Vlasov SDE 
	\begin{equation}
		\label{eq:danMV_Z}
		dX_t = \bigg( b_0\big( t, X[t] \big) + \Big\langle \mu[t], b\big( t, X[t], \cdot \big) \Big\rangle \bigg) dt + dZ_t \quad \mu[t] = \bP \circ \big( X[t] \big)^{-1}. 
	\end{equation}
    We are proving a slightly more general result that Theorem \ref{qPropChaosfBM} where the driving Gaussian signal was taken to be a fractional Brownian motion. 

    Following on from Definition \ref{definition:well-posed}, we need a specific notion for well-posedness of our solution in order to prove uniform propagation of chaos:
    \begin{definition}
        \label{definition:well-posed_Z}
        Let $K:[0, T] \to L^2\big( [0,T]; \lin(\bR^d, \bR^d) \big)$ be a Volterra kernel that satisfies Assumption \ref{assumption:VolterraK} and let $Z$ be Gaussian process of the form \eqref{eq:GaussVolterra}. 
        
        Let $n\in \bN$ and let $\overline{b}:[0,T] \times \cC_T^n \to \bR^n$ be progressively measurable. We say that the stochastic differential equation
        \begin{equation*}
            dX_t = \overline{b}\big( t, X[t] \big) dt + dW_t
        \end{equation*}
        is \emph{well-posed} if for every $(t_0, z) \in [0, T) \times \cC_T^n$, there exists a unique in law solution to the stochastic differential equation
        \begin{equation*}
            dX_t = \overline{b}\big( t, X[t] \big) dt + dW_t, \quad t\in (t_0, T] 
            \quad \mbox{and}\quad
            X_t = z_t \quad t\in [0,t_0]. 
        \end{equation*}
    \end{definition}

    Our goal is to establish a rate of convergence following the techniques pioneered in \cite{lacker2022Hierarchies}. 
    \begin{assumption}
        \label{assumption:QPoC}
        Suppose that there exists a unique weak solution to Equation \eqref{eq:danParticleSystem_Z} whose law $P^{(n)} \in \cP\big( (\cC_T^d)^{\times n} \big)$ is exchangeable, and suppose that there exists a weak solution to the McKean-Vlasov equation \eqref{eq:danMV_Z} with law $\mu \in \cP(\cC_T^d)$. 

        Further, suppose
        \begin{enumerate}
            \item The Volterra kernel $K:[0,T] \to L^2\big( [0,T]; \lin(\bR^d, \bR^d) \big)$ satisfies Assumption \ref{assumption:VolterraK} and the collection of Gaussian Volterra processes $(Z^i)_{i=1, ..., n}$ defined as in Equation \eqref{eq:GaussVolterra} are exchangeable. 
            \item For every $k\in \bN$, there exists a well defined solution (in the sense of Definition \ref{definition:well-posed_Z}) to the stochastic differential equation
            \begin{equation*}
                d\overline{\rx}_t^{k} = \overline{b}^{\otimes k}\big( t, \overline{\rx}^k[t] \big) dt + d\rZ_t^{k} 
            \end{equation*}
            where $\overline{\rx}_t^k = (\overline{X}_t^{i, k})_{i=1, .., k}$, $\rZ_t^{k} = (Z_t^{i,k})_{i=1, ..., k}$ and
			\begin{align*}
				\overline{b}^{\otimes k} \big( t, \overline{\rx}^k[t] \big) = \bigotimes_{i=1}^k \bigg( b_0\big( t, \overline{X}^{i,k}[t] \big) + \Big\langle \mu[t], b\big( t, \overline{X}^{i, k}[t], \cdot \big) \Big\rangle \bigg).
			\end{align*}    
			\item Square integrability:
			\begin{equation}
				\label{eq:fBmSqInt}
				\sup_{n\in \bN} \int_{(\cC_T^d)^{\times 2}} \bigg( \int_0^T \Big| Q^b\big(s, x, y \big) - \Big\langle \mu, Q^b\big(s, x, \cdot \big) \Big\rangle \Big|^2 ds \bigg) P^{(n, 2)}(dx, dy) = M < \infty.
			\end{equation}
			\item Transport inequality. Let $\mu$ be the law of the solution of \eqref{fBmMV}. Then there exists a $\gamma >0$ such that for all $t \in [0,T]$, $x \in \cC_T^d$, and $\nu \in  \cP(\cC_T^d)$ such that $Q_t^b\big( x, \cdot \big) \in L^1(d\nu)$, we have
			\begin{equation}
				\label{eq:fBmTransportInequality}
				\bigg| \Big\langle \nu[t] - \mu[t], Q^b\big( t, x, \cdot \big) \Big\rangle \bigg|^2 
                \leq 
                \gamma\cdot \bH\big[ \nu[t] \big| \mu[t] \big].
			\end{equation}
			\item There exists $C_0>0$ such that
			\begin{equation}
				\bH\big[ P_0^{(n,k)} \big| \mu_0^{\otimes k} \big] \leq C_0 \frac{k^2}{n^2}.
			\end{equation}
		\end{enumerate}
    \end{assumption}

    We implement our Mimicking Theorem (see Theorem \ref{brownianFilteringThm}) and Entropy estimates (see Proposition \ref{fBmEntropy}) to prove a quantitative propagation of chaos result:
    \begin{theorem}
		\label{theorem:qPropChaos}
	    Let $K:[0, T] \to L^2\big( [0,T]; \lin(\bR^d, \bR^d) \big)$ be a Volterra kernel, and let $b_0: [0,T] \times \cC_T^d \to \bR^d$ and $b:[0,T] \times \cC_T^d \times \cC_T^d \to \bR^d$ be progressively measurable. Suppose that Assumption \ref{assumption:QPoC} is satisfied. 
        
		Then, there exists $\bar{M}$ depending on $\gamma$, $M$, and $C_0$ such that for large enough $n \in \bN$ and any $k \in \{1, \ldots n\}$, we have
		\begin{equation}
			\bH\big[ P^{(n, k)}[t] \big| \mu^{\otimes k}[t] \big] \leq \bar{M}\bigg(\frac{k^2}{n^2} + \exp\bigg( - 2n(e^{-\gamma T} - \frac{k}{n})^2_+ \bigg) \bigg).
		\end{equation}
	\end{theorem}

    By convolving the dynamics of Equation \eqref{eq:danParticleSystem_Z} and Equation \eqref{eq:danMV_Z} with the Volterra kernel $L:[0,T] \to \fWIC$ (derived from Assumption \ref{assumption:VolterraK}), we obtain for $i=1, ..., n$ the dynamics
    \begin{align}
        \nonumber
        &(\newX_t)^{i,n} = X_0^i + \int_0^t L(t, s) \cdot \bigg( b_0\big( s, X^{i,n} \big) + \tfrac{1}{n-1} \sum_{\substack{j=1 \\ j\neq i}}^n b\big( s, X^{i,n}, X^{j,n} \big) \bigg) ds + \int_0^t L(t, s) dZ_s^{i,n}
        \\
        \label{eq:qPropChaosfBM-1}
        &= X_0^i + \int_0^t Q^{b_0}\Big(s,  X^{i,n} \Big) ds + \int_0^t \tfrac{1}{n-1} \sum_{\substack{j=1\\j\neq i}}^n Q^{b}\Big(s, X^{i,n},  X^{j,n} \Big) ds + W_t^{i,n}
    \end{align}
    where $Q^{b_0}$ and $Q^b$ are defined as in Equation \eqref{Q}. 
    
    Further, 
    \begin{align}
        \nonumber
        \newX_t =& X_0 + \int_0^t L(t, s) \cdot \bigg( b_0\big( s, X \big) + \Big\langle \mu[s], b\big( s, X, \cdot \big) \Big\rangle \bigg) ds + W_t
        \\
        \label{eq:qPropChaosfBM-2}
        =& X_0 + \int_0^t Q^{b_0}\big( s, X \big) ds + \int_0^t \Big\langle \mu[s], Q^{b}\big( s, X, \cdot \big) \Big\rangle ds + W_t. 
    \end{align}

    We want to establish a global estimate for the Kullback-Lieber diverence between the measures $P^{(n, k)}$ and $\mu^{\otimes k}$. In order to compress notation, we denote
    \begin{equation}
        \label{eq:Ent-Definition}
        \bH_t^{(n)}:= \bH\big[ P^{(n,n)}[t] \big| \mu^{\otimes n}[t] \big]
        \quad \mbox{and}\quad 
        \bH_t^{(n, k)}:= \bH\big[ P^{(n,k)}[t] \big| \mu^{\otimes k}[t] \big]. 
    \end{equation}

    \begin{proposition}
        \label{proposition:GlobelEnt}
        Under Assumption \ref{assumption:QPoC} the relative entropy between the law of the $n$-interacting system of equations \eqref{eq:danParticleSystem_Z} and the $n$-tensor of the law of the McKean-Vlasov equation \eqref{eq:danMV_Z} satisfies the global estimate
        \begin{equation}
            \label{eq:proposition:GlobelEnt}
			\bH_T^{(n)} \leq \bH_0^{(n)} + \frac{n M}{2}
		\end{equation}
        where $M$ is the constant from \eqref{eq:fBmSqInt}. 
    \end{proposition}

    \begin{proof}
        By applying Proposition \ref{fBmEntropy} followed by using the exchangeability of the system, we get
		\begin{align*}
			\bH_T^{(n)} &= \bH_0^{(n)} + \frac{1}{2} \bE \bigg[ \sum_{i = 1}^n \int_0^T \Big| \tfrac{1}{n-1} \sum_{\substack{j=1 \\ j\neq i}}^n Q^{b} \big(s, X^{i, n}, X^{j, n} \big) - \big\langle \mu, Q^b(s, X^{i, n}, \cdot) \big\rangle  \Big|^2 ds \bigg]
            \\
			&= \bH_0^{(n)} + \frac{n}{2} \bE\bigg[ \int_0^T \Big| \tfrac{1}{n-1} \sum_{\substack{j=1 \\ j\neq i}}^n Q^{b} \big(s, X^{1, n}, X^{j, n} \big) - \big\langle \mu, Q^b(s, X^{1, n}, \cdot) \big\rangle \Big|^2 ds \bigg] 
            \\
			&\leq \bH_0^{(n)} + \frac{n}{2} \int_0^T \bE\bigg[ \Big| Q^{b} \big(s, X^{1, n}, X^{n, n} \big) - \big\langle \mu, Q^b(s,X^{1, n}, \cdot) \big\rangle \Big|^2 \bigg] ds. 
		\end{align*}
    Finally, apply \eqref{eq:fBmSqInt} to conclude with Equation \eqref{eq:proposition:GlobelEnt}. 
    \end{proof}
    
    \subsubsection*{Volterra transformations of measure dependent functions}

    This first Lemma allows us to change the order of integration between between convolution with the Volterra kernel and integration by some measure on pathspace:
    \begin{lemma}
    	\label{lem:qprop:fubini}
        Let $\mu \in \cP_2( \cC_T^d)$. Let $K:[0,T] \to L^2\big( [0,T]; \lin(\bR^d, \bR^d) \big)$ be a Volterra kernel that satisfies \ref{assumption:VolterraK}, let $\RKHS_T$ be the associated reproducing kernel Hilbert space and let $b: [0, T] \times \cC_T^d  \to \bR^d$ be a progressively measurable process such that
        \begin{equation*}
        \bigg\| \int_0^\cdot b_s ds \bigg\|_{\RKHS_T} < \infty \quad \mu\mbox{-almost surely. }
        \end{equation*}
        Then for every $t \in (0, T]$, we have
        \begin{equation*}
            Q^{\langle \mu, b\rangle}(t, X) = \frac{d}{dt} \int_0^t L(t, s) \langle \mu, b(s,\cdot) \rangle ds = \Big\langle \mu, \frac{d}{dt} \int_0^t L(t, s) b(s, \cdot) ds \Big\rangle.
        \end{equation*}
    \end{lemma}

    \begin{proof}
    By Assumption \ref{assumption:VolterraK}, we have $L(t, \cdot) \in \fWIC$ and is therefore continuous Therefore, we have
        \begin{equation*}
            \int_0^t L(t, s) b_s ds \in L^1(\Omega \times 
            [0, T]; d\mu \times dt),
        \end{equation*}
        and
        \begin{equation*}
            L(t, s) b_s \in L^1(\Omega \times 
            [0, T]; d\mu \times ds).
        \end{equation*}        
        Let $\phi \in C_0^\infty((0, T); \mathbb{R}^d)$ be a smooth compactly supported function vanishing at $0$ and $t$. Then, by \eqref{eq:definition:Fund-SemiMart} and the Fubini-Tonelli theorem, we have \begin{equation*}
            \int_0^T 
            \big\langle \mu, Q_t^b \big\rangle \cdot \phi_t dt = \bigg \langle \mu, \int_0^T Q^b_t \cdot \phi_t dt \bigg \rangle = \bigg \langle \mu, \int_0^T \frac{d}{dt} \bigg(\int_0^t L(t, s) b_s ds\bigg) \cdot \phi_t dt \bigg \rangle.
        \end{equation*}
        Integrating by parts, we have
        \begin{equation*}
            \bigg \langle \mu, \int_0^T \frac{d}{dt} \bigg(\int_0^t L(t, s) b_s ds\bigg) \cdot \phi_t dt \bigg \rangle =  - \bigg \langle \mu, \int_0^T  \bigg(\int_0^t L(t, s) b_s ds\bigg) \cdot \phi'_t dt \bigg \rangle.
        \end{equation*}
        By the assumptions we have
        \begin{equation*}
            - \bigg \langle \mu, \int_0^T  \bigg(\int_0^t L(t, s) b_s ds\bigg) \cdot \phi'_t dt \bigg \rangle = - \int_0^T \bigg( \int_0^t L(t, s) \langle \mu, b_s \rangle ds \bigg) \cdot \phi'_t dt.
        \end{equation*}
        Integrating by parts again we have \begin{equation*}
            - \int_0^T \bigg( \int_0^t L(t, s) \langle \mu, b_s \rangle ds \bigg) \cdot \phi'_t dt = \int_0^T Q^{\langle\mu, b \rangle }_t \phi_t dt,
        \end{equation*}
        and the claim follows.
    \end{proof}
    
	\begin{proof}[Proof of Theorem \ref{theorem:qPropChaos}]
        Thanks to Theorem \ref{brownianFilteringThm} applied to Equation \eqref{eq:qPropChaosfBM-1}, we can find an extended probability space
        \begin{equation*}
            \big( \hat{\Omega}, \hat{\cF}, (\hat{\cF_t})_{t\in [0,T]}, \hat{\bP} \big)
            \quad \mbox{of}\quad
            \big( \Omega, \cF, \bF, \bP \big)
        \end{equation*}
        carrying a family of Gaussian processes $(W^{i,n})_{i=1, ..., k}$ such that $\rx_t^{\dagger,(n,k)} = (X_t^{\dagger,i,n})_{i=1, ..., k}$
        \begin{align*}
			(\newX_t)^{i, n} =& X^{i, n}_0 + \int_0^t Q^{b_0}\Big(s, X^{i,n} \Big) ds + \int_0^t \tfrac{1}{n-1} \sum_{\substack{j=1 \\ j\neq i}}^k Q^b\Big(s, X^{i,n}, X^{j,n} \Big) ds 
            \\
            &+ \int_0^t \tfrac{1}{n-1} \sum_{j=k+1}^n \bE\bigg[ Q^b\Big(s, X^{i,n}, X^{j,n} \Big) \bigg| \rx^{(n,k)} \bigg] ds 
            + W^{i, n}_t, \quad i = 1, ..., k. 
		\end{align*}
        Notice that for $j=k+1, ..., n$
        \begin{align*}
            \bE\bigg[ Q^b\Big(s, X^{i,n}, X^{j,n} \Big) \bigg| \rx^{(n,k)}[s] \bigg]
            =
            \Big\langle P^{(n,k+1|k)}_{\rx^{(n,k)}[s]}, Q^b\Big( s, X^{i}, \cdot \Big) \Big\rangle
        \end{align*}
        so that for $i \leq k$, 
		\begin{align*}
			X_t^{\dagger,i,n} =& X_0^i + \int_0^t Q^{b_0}\Big( s, X^{i,n} \Big) ds
            + \int_0^t \tfrac{1}{n-1} \sum_{\substack{j=1 \\ j\neq i }}^k Q^b_s\Big( X^{i,n}, X^{j,n} \Big) ds 
			\\
			&+ \tfrac{n - k}{n-1} \int_0^t \Big\langle P^{(n,k+1|k)}_{\rx^{(n,k)}[s]}, Q^b\Big(s, X^{i,n}, \cdot \big) \Big\rangle ds + W_t^{i,n}.
		\end{align*}
          
		By applying Proposition \ref{fBmEntropy} we get that,
		\begin{align}
            \nonumber
			\bH_t^{(n,k)} =& \bH\Big[ P_0^{(n,k)} \Big| \mu_0^{(n, k)} \Big] 
            \\
            \label{eq:qPropChaosfBM-4}
			&+ \frac{1}{2} \sum_{i = 1}^k \int_0^t
			\bE\Bigg[ \bigg| 
            \tfrac{1}{n-1} \sum_{\substack{j=1 \\ j\neq i}}^k Q^{b}\Big(s, X^{i,n}, X^{j,n} \Big) 
            - \tfrac{k-1}{n-1} \cdot 
            \Big\langle \mu[s], Q^{b}\Big(s, X^{i,n}, \cdot \Big) \Big\rangle \bigg|^2 \Bigg] ds
            \\
            \label{eq:qPropChaosfBM-5}
            &+ \frac{1}{2} \sum_{i=1}^k \int_0^t \Big( \tfrac{n-k}{n-1}\Big)^2\cdot \bE\Bigg[ \bigg| \Big\langle P^{(n,k+1|k)}_{\rx^{(n,k)}[s]} - \mu[s], Q^b\big(s, X^{i,n}, \cdot \big) \Big\rangle
            \bigg|^2 \Bigg] ds. 
		\end{align}
  
        Firstly, by the exchangeability of the collection of random variables $(X^{j,n}[t])_{j=1, ..., k}$ and the square integrability condition \eqref{eq:fBmSqInt},
		\begin{align}
            \nonumber
            \eqref{eq:qPropChaosfBM-4} 
            =& \tfrac{k(k-1)^2}{2(n-1)^2} \cdot \int_0^t \bE \Bigg[ \bigg| Q^b\Big(s, X^{1,n}, X^{n,n} \Big) 
            - 
            \Big\langle \mu[s], Q^b\big(s, X^{1,n}, \cdot \big) \Big\rangle \bigg|^2 \Bigg] ds
            \\
            \label{I1Estimate}
            \leq& \tfrac{k(k-1)^2}{(n-1)^2} \cdot TM. 
		\end{align}
		
        On the other hand,
		\begin{align*}
			\eqref{eq:qPropChaosfBM-5} = \tfrac{k(n-k)^2}{2(n-1)^2} \cdot \int_0^t \bE \Bigg[ \bigg| \Big\langle P^{(n,k + 1| k)}_{\rx^{(n,k)}[s]} - \mu[s], Q^b\big( s, X^{1,k}, \cdot \big) \Big\rangle \bigg|^2 \Bigg] ds.
		\end{align*}
		Apply the transport inequality \eqref{eq:fBmTransportInequality}, we get
		\begin{align*}
			\bE \Bigg[ \bigg| \Big\langle P^{(n,k + 1| k)}_{\rx^{(k,n)}[s]} - \mu[s], Q^b\big( s, X^{1,k}, \cdot \big) \Big\rangle \bigg|^2 \Bigg] 
            \leq 
            \gamma\cdot \bE \Big[ \bH\big[ P^{(n,k + 1| k)}_{\rx^{(k,n)}[s]} \big| \mu[s] \big] \Big]
		\end{align*}
		Applying the chain rule for relative entropy 
        \begin{equation*}
            \bE\Big[ \bH\big[ P^{(n,k + 1| k)}_{\rx^{(n,k)}[s]} \big| \mu[s] \big] \Big] 
            = 
            \bH\Big[ P_s^{(n,k + 1)} \Big| \mu[s] \Big] - \bH\Big[ P_s^{(n,k)} \Big| \mu[s] \Big]
        \end{equation*}
        yields that
		\begin{equation}
            \label{diffInequality}
			\bH_t^{(n,k)} \leq \bH\Big[ P_0^{(n,k)}\Big| \mu_0^{(k)} \Big] + \tfrac{k(k-1)^2}{(n-1)^2}\cdot TM + k \gamma \int_0^t \Big( \bH_s^{(n,k+1)} - \bH_s^{(n,k)} \Big) ds.
		\end{equation}
        We apply Gronwall's inequality to  \eqref{diffInequality} to get
        \begin{equation}
            \label{intInequality}
            \bH_t^{(n,k)} \leq e^{ - \gamma kt} \cdot \bH_0^{(n,k)} + \int_0^t e^{- \gamma k\cdot (t - s)} \cdot \bigg( \tfrac{k(k-1)^2}{(n - 1)^2}TM + (\gamma k) \cdot \bH_s^{(n,k+1)} \bigg) ds.
        \end{equation}
        Iterate \eqref{intInequality} $(n - k)$ times to obtain
        \begin{equation}
            \bH_T^{(n,k)} \leq \sum_{l = k}^{n-1} \bigg[ B_k^l(T) \bH_0^{(n,l)} + \tfrac{(l - 1)^2 TM}{\gamma(n-1)^2} \cdot A_k^l(T) \bigg] + A_k^{n - 1}(T) \bH_T^{n,n},
        \end{equation}
        where (following \cite{lacker2022Hierarchies}*{Eq (4.19)}) we define
        \begin{equation}
            \label{eq:Lacker-functions}
            \begin{aligned}
                A_k^l(t_k) :=& \Big( \prod_{j=k}^l (\gamma j) \Big) \cdot \int_0^{t_k} \int_0^{t_{k-1}} ... \int_0^{t_l} \exp\bigg( -\sum_{j=k}^l \gamma_j (t_j - t_{j+1}) \bigg) dt_{l+1} ... dt_{k+2} dt_{k+1}
                \\
                B_k^l(t_k) :=& \Big( \prod_{j=k}^{l-1} (\gamma j) \Big) \cdot \int_0^{t_k} \int_{0}^{t_{k+1}} ... \int_0^{t_{l-1}} \exp\bigg( -\gamma l t_l - \sum_{j=k}^{l-1} \gamma j\big( t_j - t_{j+1} \big) \bigg) dt_l ... dt_{k+2} dt_{k+1}
            \end{aligned}
        \end{equation}
        for $t_k>0$ and $n\geq l \geq k \geq 1$, with $B_k^k(t):=e^{-\gamma k t}$. 
        By Lemma \ref{Lemma:Lacker-ODE}, we can estimate $A_k^l$ and $B_k^l$ to get the inequality
        \begin{equation}
            \bH_T^{(n,k)} \leq \frac{e^{3\gamma T}}{n^2} \cdot \Big( 2C_0k^2 + 5M T k^3 \Big) + \Big( C_0 + MnT \Big) \cdot \exp\bigg( -2n \Big( e^{-\gamma T} - \tfrac{k}{n} \Big)_+^2 \bigg). 
        \end{equation}
        This is $O(k^3/n^2)$, so we improved the estimate by exploiting exchangeability. Notice that we can find a constant $\hat{C}$ such that
        \begin{equation}
            \label{H3Estimate}
            \bH^{(n,3)}_T \leq \frac{\hat{C}}{n^2}.
        \end{equation}
        
        Next, we can improve the estimate \eqref{I1Estimate} by expanding the square. By exchangeability, 
        \begin{align*}
            \eqref{eq:qPropChaosfBM-4} =& \tfrac{k(k-1)}{(n - 1)^2} \cdot \int_0^t \bE\Bigg[ \bigg| \sum_{j = 2}^k Q^b\big(s, X^{1,n}, X^{j,n} \big) - \Big\langle \mu, Q^b\big(s, X^{1,n}, \cdot \big) \Big\rangle \bigg|^2 \Bigg] ds
            \\
            =& \tfrac{k(k-1)(k-2)}{(n - 1)^2} \bE\Bigg[ \int_0^t \bigg( Q^b \big(s, X^{1,n}, X^{2,n} \big) - \Big\langle \mu, Q^b\big(s, X^{1,n}, \cdot \big) \Big\rangle \bigg)
            \\
            &\hspace{60pt} \cdot \bigg( Q^b\big(s, X^{1,n}, X^{3,n} \big) - \Big\langle \mu[s], Q^b\big(s, X^{1,n}, \cdot \big) \Big\rangle \bigg) ds \Bigg] 
            + MT \frac{k(k-1)}{(n-1)^2}.
        \end{align*}
        
        Using the transport inequality from Equation \eqref{eq:fBmTransportInequality}, we get that for any $t\in [0,T]$
        \begin{align*}
            \bE\Bigg[& \bigg( Q^b\big(s, X^{1,n}, X^{2,n} \big) - \Big\langle \mu[s], Q^b\big(s, X^{1,n}, \cdot \big) \Big\rangle \bigg) 
            \cdot\bigg( Q^b\big(s, X^{1,n}, X^{3,n} \big) - \Big\langle \mu[s], Q^b\big(s, X^{1,n}, \cdot \big) \Big\rangle \bigg) \Bigg],
            \\
            =& \bE\Bigg[ \bigg( Q^b\big(s, X^{1,n}, X^{2,n} \big) - \Big\langle \mu[s], Q^b\big(s, X^{1,n}, \cdot \big) \Big\rangle \bigg)
            \cdot \bigg( \Big\langle P^{(n,3|2)}_{(X^{1,n}, X^{2,n})[s]} - \mu[s], Q^b\big(s, X^{1,n}[s], \cdot \big) \Big\rangle \bigg) \Bigg], 
            \\
            \leq& M^{\frac{1}{2}} \cdot \bE\Bigg[ \bigg| \Big\langle P^{(n,3|2)}_{(X^{1,n}, X^{2,n})[s]} - \mu[s], Q^b\big(s, X^{1,n}, \cdot \big) \Big\rangle \bigg|^2 \Bigg]^{\frac{1}{2}}
            \\
            \leq& M^{\frac{1}{2}} \cdot \gamma^{\frac{1}{2}} \cdot \bE\bigg[ \bH\Big[ P^{(n,3|2)}_{(X^{1,n}, X^{2,n})[s]} \Big| \mu[s] \Big] \bigg]^{\frac{1}{2}} 
            \leq 
            M^{\frac{1}{2}} \cdot \gamma^{\frac{1}{2}} \cdot \Big( \bH_s^{(n,3)} - \bH_s^{(n,2)} \Big)^{\frac{1}{2}} 
            \leq \frac{\sqrt{M\gamma \hat{C}}}{n},
        \end{align*}
        where the last inequality is due to \eqref{H3Estimate}. Therefore, for some constant $M^* \in (0, \infty)$, 
        \begin{align*}
            \eqref{eq:qPropChaosfBM-4} \leq MT\frac{k(k-1)}{(n-1)^2} + \sqrt{M\gamma \hat{C}}\frac{k(k-1)(k-2)}{n(n-1)^2} \leq M_* \cdot \frac{k^2}{n^2}.
        \end{align*}
        Thus we improve the original differential inequality \eqref{diffInequality} to 
        \begin{equation}
            \label{diffInequality2}
            \bH_t^{(n,k)} \leq \bH_0^{(n, k)} + \frac{k^2}{n^2}M_* + k \int_0^t \Big( \bH_s^{(n, k+1)} - \bH_s^{(n,k)} \Big) ds.
        \end{equation}
        Just as before, we apply Gr\"onwall's inequality to get
        \begin{equation}
            \bH_t^{(n, k)} \leq e^{-\gamma k t} \bH_0^{(n, k)} + \int_0^t e^{- \gamma k (t - s)}\bigg( \frac{k^2}{n^2}M_* + \gamma k \bH_s^{(n, k+1)}\bigg) ds.
            \label{eq:betterGronwall}
        \end{equation}
        Iterating \eqref{eq:betterGronwall} $(n - k)$-times yields
        \begin{equation}
            \bH_T^{(n, k)} \leq \sum_{l = k}^{n - 1} \bigg[ B_k^l(T) \frac{C_0 l^2}{n^2} + M_* \frac{l}{\gamma n^2}A_k^l(T)\bigg] + A_k^{(n - 1)}(T) \bH^{(n, n)}_T
        \end{equation}
        Applying our global entropy estimate \eqref{eq:proposition:GlobelEnt} and Lemma \ref{Lemma:Lacker-ODE}, we find a constant $\bar{M} \in (0, \infty)$ such that
        \begin{align*}
            \bH_T^{(n,k)} \leq \bar{M} \bigg( \frac{k^2}{n^2} + \exp\bigg( - 2n\bigg( e^{- \gamma T} - \frac{k}{n}\bigg)_+^2 \bigg).
        \end{align*}
	\end{proof}

    \begin{remark}
        We emphasise that Theorem \ref{qPropChaosfBM} holds as a Corollary of Theorem \ref{theorem:qPropChaos} since it was proved in \cite{Hu2023Locally} that a fractional Brownian motion is a Gaussian Volterra process that satisfies Assumption \ref{assumption:VolterraK}. 
    \end{remark}
	
    \subsection{Examples}
    \label{subsection:QPoCExample}
    In this final section, we include some examples to demonstrate the power of Theorem \ref{theorem:qPropChaos}. 
    
    \subsubsection*{Lower bounds on rate of convergence}
    Next, we compute the explicit rate of convergence for a couple of examples to illustrate these results:
    \begin{example}[Fractional Ornstein-Uhlenbeck process]
        \label{example:fOU-prcess}
        Let $H > \frac{1}{2}$. For each $n\in \bN$, let $(Z^i)_{i=1, ..., n}$ be a collection of $n$-independent fractional Brownian motions, each with Hurst parameter $H$. Fix $a, b \in \mathbb{R}$ such that $-a \neq b$. Consider the system of $n$ interacting equations 
        \begin{equation}
            \label{eq:example:fOU-prcess}
            dX_t^{i,n} = -\bigg(aX_t^{i,n} + \tfrac{b}{n} \sum_{j=1 }^n X_t^{j,n} \bigg) dt + dZ_t^{i,n}, 
            \quad
            X^{i,n}_0 = 0
            \quad i=1, ...,n.
        \end{equation}
        Notice that we have modified the equation slightly from \eqref{eq:danParticleSystem_Z}. The difference is negligible and clarifies the exposition. Equation \eqref{eq:example:fOU-prcess} describes the dynamics of a fractional Ornstein-Uhlenbeck process whose solution can be expressed in vector form $\rx^n = (X^{i,n})_{i=1, ..., n}$ and $\rZ^n = (Z^{i, n})_{i=1, ..., n}$
        \begin{equation*}
            d\rx_t^n = - A_n \cdot \rx_t^n dt + d\rZ_t^n.
        \end{equation*}
        where
        \begin{equation*}
            A_n = a I_n + \tfrac{b}{n}J_n.
        \end{equation*} 
        Let $P^{(n, k)} \in \cP \big( (\cC_T^d)^{\times k} \big)$ denote the law of the collection of stochastic processes $\rx^{(n, k)} = (X^{i, n})_{i=1, ..., k}$ where each $X^{i, n}$ is determined by Equation \eqref{eq:example:fOU-prcess} and $\mu$ denote the law of the McKean-Vlasov limit. The goal of this example is to show that for any $t \in (0,T]$
        \begin{align*}
            \liminf_{n \rightarrow \infty} \frac{n^2}{k^2}\bW^{(2)}\big[ P^{(n, k)}_t, \mu^{\otimes k}_t \big]^2 > 0.
        \end{align*}
        The SDE \eqref{eq:example:fOU-prcess} has explicit solution
        \begin{align*}
            \rx^n_t = \int_0^t e^{-(t-s)A_n} d\rZ^n_s.
        \end{align*}
        It is well known (see for example \cite{cheridito2003fou}), that the covariance of $P^{(n, k)}_t$ is given by  \begin{equation*}
            \Sigma_t^n := \mathbb{E}[\rx^n_t \otimes \rx^n_t] = C_H\int_0^t \int_0^t e^{-(t-u)A_n}e^{-(t-v)A_n}|u - v|^{2H-2} dudv.
        \end{equation*}
        By the change of variable $u \mapsto t - u$ and $v \mapsto t - v$, we have 
        \begin{equation*}
            \Sigma_t^n = C_H \int_0^t \int_0^t e^{-A_n(u + v)} |u - v|^{2H-2} dudv.
        \end{equation*}
        The matrix exponential of the Toeplitz matrix $A_n$ is
        \begin{equation*}
                e^{rA_n} = e^{ra} \bigg( I_n + \tfrac{1}{n}\Big( e^{rb} - 1 \Big) J_n \bigg).
        \end{equation*}
        Define the functions \begin{align*}            \xi(t) &= C_H \int_0^t \int_0^t e^{-a(u + v)}|u - v|^{2H-2} du dv, \\
        \eta(t) &= C_H \int_0^t \int_0^t e^{-(a+b)(u+v)}|u - v|^{2H-2}.
        \end{align*}
        We can then express the covariance $\Sigma_t^n$ as \begin{equation*}
            \Sigma_t^n = \xi(t) I_n + \frac{1}{n}(\eta(t) - \xi(t)) J_n.
        \end{equation*}
        Using classical propagation of chaos, we see that the covariance of $\mu_t^{\otimes k}$ is given by $\bar{\Sigma}_t = \xi(t) I_k. $
        To compute the Wasserstein distance, we use (3.2) of \cite{lacker2022Hierarchies} to see
        \begin{align*}
            \frac{n^2}{k^2}\bW^{(2)}\big[ P^{(n, k)}_t, \mu^{\otimes k}_t \big]^2 &= \frac{n^2}{k^2}\Bigg[\bigg(\xi(t) + \frac{k}{n} \big(\eta(t) - \xi(t)\bigg)^{\frac{1}{2}} - \xi(t)^{\frac{1}{2}} \Bigg] ^2 \\
            &=  \frac{n^2}{k^2}\xi(t) \Bigg[\bigg(1 + \frac{k}{n} \frac{\eta(t) - \xi(t))}{\xi(t)}\bigg)^{\frac{1}{2}} - 1 \Big] ^2.
        \end{align*}
        For any $t > 0$, we have $c_{a, b}(t) := \tfrac{\eta(t) - \xi(t)}{\xi(t)} \neq 0 $ since $-a \neq b$. Therefore, 
    \begin{align*}
            \frac{n^2}{k^2}\bW^{(2)}\big[ P^{(n, k)}_t, \mu^{\otimes k}_t \big]^2 
            =  \frac{n^2}{k^2}\xi(t) \Bigg[\bigg(1 + \frac{k}{n} c_{a, b}(t)\bigg)^{\frac{1}{2}} - 1 \Bigg] ^2 = \xi(t) c_{a, b}(t)^2 \frac{\big[\big(1 + \tfrac{k}{n}c_{a, b}(t)\big)^{\frac{1}{2}} - 1\big]^2}{\big(\tfrac{k}{n}c_{a, b}(t)\big)^2}.
        \end{align*}
    As noticed in \cite{lacker2022Hierarchies}, $\frac{[(1 + x)^{\frac{1}{2}} - 1]^2}{x^2} \rightarrow\frac{1}{4}$ as $x \rightarrow 0$. Since $\xi(t) c_{a, b}(t)^2 > 0$, we conclude \begin{equation*}
        \lim_{n \rightarrow \infty}  \frac{n^2}{k^2}\bW^{(2)}\big[ P^{(n, k)}_t, \mu^{\otimes k}_t \big]^2 = \frac{\xi(t) c_{a, b}(t)^2}{4} > 0.
    \end{equation*}
    \end{example}

    \section{Dynamics of marginals of locally interaction equations}
    \label{section:LocalEq}
    
    Next, we turn our attention away from mean-field dynamics and instead focus on collections of locally interacting equations. 
    
    Let $V$ be a set and let $E \subset \big\{ \{u, v\}: u, v \in V\big\}$. Then we say that $(V, E)$ is a graph. For any $u \in V$, let $N_u:=\big\{v \in V: \{u, v\} \in E \big\}$. A graph is referred to as a \emph{tree} if any two vertices are connected by exactly one path. A tree $(V, E)$ paired with a vertex $\oSlash \in V$ is said to be a rooted tree. A tree $(V, E)$ is said to be $\kappa$-regular if every vertex has exactly $\kappa$-neighbours. Given the $\kappa$-regular rooted forest $(V, E, \oSlash)$ and vertex $u \in V$, we denote $\pi[u] \in V$ to be the parent vertex of $u$ while $C_u$ is the children of $u$. Hence, whenever $u \in V \backslash\{\oSlash\}$, we have that $N_u = C_u \cup \{ \pi[u] \}$.  
    
    Let $(V, E)$ be a tree and let $(\cM, d)$ be a metric space. For some $L:V \to \cM$, we say that $(V, E, L)$ is a marked graph with marks taking values in the metric space $(\cM, d)$. Let $\cG_{\ast}(\cM)$ be the collection of all rooted trees with marks taking values in a metric space $(\cM, d)$. We endow $\cG_{\ast}(\cM)$ with the topology of local convergence. When $(\cM, d)$ is a Polish space, we conclude that $\cG_{\ast}(\cM)$ is also a Polish space. 
    
    For any rooted tree $(V, E, \oSlash)$, we denote $\cG_{\ast}^{(V, E)}(\cM)$ to be the collection of all marked rooted trees $(\tilde{V}, \tilde{E}, \tilde{\oSlash}, \tilde{L})$ such that the rooted tree $(\tilde{V}, \tilde{E}, \tilde{\oSlash})$ is isomorphic to $(V, E, \oSlash)$. 
    
    Let $(V, E, \oSlash)$ be a $\kappa$-regular rooted tree and consider the system of locally interacting SDEs of the form
    \begin{equation}
    	\label{eq:Interacting_Equation-E} 
    	\left.
    	\begin{aligned}
	    	&dX_t^u = b\Big( t, X^u[t], \mu^u\big( X[t] \big) \Big) dt + dZ_t^u
    		\quad
    		X_0^u \sim \lambda_0
    		\\
    		&\mu^u(X[t]) = \tfrac{1}{\kappa} \sum_{v\in N_u} \delta_{X^u[t]} 
    	\end{aligned}
    	\right\}
    	\quad
    	u \in V
    \end{equation}
    where the collection $(Z^u)_{u\in V}$ are independent and identically distributed Gaussian Volterra processes. 
    
    Our goal (which is achieved in Theorem \ref{theorem:LocalEquation} below) is to find a (distribution dependent) stochastic differential equation which has the same probability distribution as the collection of stochastic processes $(X^u[t])_{u\in B_1(\oSlash)}$ where $B_1(\oSlash) = \{\oSlash\} \cup N_{\oSlash}$ is ball of radius 1 centred at the root $\oSlash$. We refer to this equation as the \emph{local equation}. The local equation is itself a conditional McKean-Vlasov equation. Further, under stronger assumptions we are also able to show that there is a unique weak solution to the local equation with Theorem \ref{theorem:Local-Uniq}. 

    First, we need to extend the concept of progressive measurability to more general filtered spaces:
    \begin{definition}
        \label{def:ProgressiveMeas}
        Let $f:[0,T] \times \cC_T^d \times \cP_2( \cC_T^d) \to \bR^d$ be a function. We say that $f$ is \emph{progressively measurable} if for every $t\in [0,T]$, the map
        \begin{equation*}
            [0, t] \times \cC_T^d \times \cP_2(\cC_T^d) \ni (s, X, \mu) \mapsto f\big(s, X, \mu \big)
        \end{equation*}
        is $\cB([0,t]) \otimes \cB(\cC_t^d) \otimes \cB\big( \cP_2(\cC_t^d) \big)$-measurable. 
    \end{definition}

    Next, we introduce some assumptions under which we start proving weak existence and uniqueness:
    \begin{assumption}
    	\label{assumption:local}
        Let $(V, E, \oSlash)$ be a rooted tree and let $M \in L^2\big([0,T]; \bR \big)$. Suppose that:
	    \begin{enumerate}
            \item $K: [0,T] \to L^2\big( [0,T]; \lin(\bR^d, \bR^d) \big)$ is a Volterra kernel that satisfies Assumption \ref{assumption:VolterraK};
		    \item $b:[0,T] \times \cC_T^d \times \cP_2( \cC_T^d) \to \bR^d$ is progressively measurable (in the sense of Definition \ref{def:ProgressiveMeas}) and suppose for every $(X, \mu) \in \cC_T^d \times \cP_1(\cC_T^d)$ that
            \begin{align*}
                &\Big| Q^{b}\big( t, X[t], \mu[t] \big) \Big| 
                \leq 
                M_t \cdot \Big( 1 + \| X^u \|_{\infty, t} + \bW_{\infty, t}^{(1)}\big[ \mu[t], \delta_{0[t]} \big] \Big) < \infty;
            \end{align*}
    	\end{enumerate}
    \end{assumption}
    
    \begin{theorem}
    	\label{theorem:Exist+Uniq-local}
    	Let $\kappa\geq 2$ and let $(V, E, \oSlash)$ be the $\kappa$-regular rooted tree. Let $\lambda_0\in \cP_2(\bR^d)$, let $K:[0,T] \to L^2\big( [0,T]; \lin(\bR^d, \bR^d) \big)$ and let $b:[0,T] \times \cC_T^d \times \cP_2( \cC_T^d) \to \bR^d$. 
     
        Suppose that Assumption \ref{assumption:local} is satisfied. Then there exists a unique
        \begin{equation*}
            \Big( (\Omega, \cF, \bP), \big(V, E, \oSlash, (X_0^u, Z^u, X^u)_{u\in V} \big) \Big)
        \end{equation*}
        such that
        \begin{align*}
            &(\Omega, \cF, \bP) 
            \quad \mbox{is a probability space and}
            \\
            &\omega \mapsto \big(V, E, (X_0^u, Z^u, X^u)_{u\in V} \big) \in \cG_{\ast}^{(V, E)}\big( \bR^d \times \cC_{0,T}^d \times \cC_T^d \big)
        \end{align*}
        and
        \begin{enumerate}
            \item For every finite $U \subseteq V$, the collection of random variables $(X_0^u)_{u \in U}$ are independent and identically distributed with distribution $\lambda_0$. Similarly, for every finite $U \subseteq V$, the collection of Gaussian Volterra processes $(Z^u)_{u\in U}$ as defined as in Equation \eqref{eq:GaussVolterra} are independent;
            \item The expectation
            \begin{equation*}
                \sup_{u \in V} \bE\Big[ \big\| X^u \big\|_{\infty}^2 \Big] < \infty;
            \end{equation*}	
            \item The random variable
            \begin{equation*}
                \omega \mapsto \bigg( V, E, \oSlash, \Big( \int_0^T \Big| b\big( t, X^{u}[t], \mu^u\big( X[t]\big) \big) \Big| dt \Big)_{u \in V} \bigg) \in \cG_{\ast}^{(V, E)}(\bR^d)
                \quad
                \bP\mbox{-almost surely; }
            \end{equation*}	
            \item The random variables
            \begin{align*}
                &\big( V, E, \oSlash, (X_{\cdot}^u)_{u\in V} \big) \quad \mbox{and}
                \\
                &\bigg( V, E, \oSlash, \Big( X_0^u + \int_0^\cdot b\big( t, X^u[t], \mu^u(X[t]) \big) dt + Z_{\cdot} \Big)_{u \in V} \bigg)\quad
                \mbox{are $\bP$-almost surely equal.}
            \end{align*}
        \end{enumerate}
        We say that
        \begin{equation*}
            \Big( (\Omega, \cF, \bP), \big( V, E, (X_0^u, Z^u, X^u)_{u\in V} \big) \Big) 
            \quad \mbox{is a weak solution to Equation \eqref{eq:Interacting_Equation-E}; }
        \end{equation*}
    \end{theorem}
    For a proof of Theorem \ref{theorem:Exist+Uniq-local}, we refer the reader to \cite{Hu2023Locally}. 

    \subsection{Symmetry condition}

    The first step to describing the marginal distributions for some local subsets of collection of locally interacting SDEs marking a $\kappa$-regular tree is understanding how the random marks associated to different vertices are correlated to one another and what their probability distributions are. 
    
    \subsubsection*{Locally interacting processes as Markov Random Fields}

    \begin{theorem}
		\label{theorem:MRF}
		Let $\kappa\geq 2$ and let $(V, E, \oSlash)$ be the $\kappa$-regular rooted tree. 

        Let $\lambda_0 \in \cP_2(\bR^d)$, let $K:[0,T] \to L^2\big( [0,T]; \lin(\bR^d, \bR^d) \big)$ and let $b:[0,T] \times \cC_T^d \times \cP_2(\cC_T^d) \to \bR^d$. Suppose that Assumption \ref{assumption:local} is satisfied. 

        Let $\big( V, E, \oSlash, (X^u)_{u \in V} \big)$ be the unique in law solution to the stochastic differential equation \eqref{eq:Interacting_Equation-E}. Then for any $t\in [0,T]$, the collection of random variables $\big( V, E, \oSlash, (X^v[t])_{v\in V} \big)$ which take their value of $\cG_{\ast}^{(V, E)}(\cC_t^d)$ is a 2-MRF. 
	\end{theorem}
    A proof of the 2-Markov Random Field property for a general class of locally interacting processes can be found in \cite{Hu2023Locally}
    
    \subsubsection*{Symmetry of the conditional law}
    
    The first thing to note is that the law of any weak solution to Equation \eqref{eq:Interacting_Equation-E} satisfies a certain symmetry property derived from the inherent symmetry of the $\kappa$-regular graph. 
    \begin{proposition}
        \label{prop:SymmetryLaw}
        Let $\kappa\geq 2$ and let $(V,E,\oSlash)$ be the $\kappa$-regular rooted tree and let $(\overline{V}, \overline{E}, \overline{\oSlash})$ be the $B_1(\oSlash)$-subgraph of $(V, E, \oSlash)$. 
        
        Let $\lambda_0 \in \cP_2(\bR^d)$, let $K:[0,T] \to L^2\big( [0,T]; \lin(\bR^d, \bR^d) \big)$ and let $b:[0,T] \times \cC_T^d \times \cP_2(\cC_T^d) \to \bR^d$. Suppose that Assumption \ref{assumption:local} is satisfied and let 
        \begin{equation*}
            \Big( (\Omega, \cF, \bP), (V, E, (X_0^u, Z^u, X^u)_{u\in V} \Big)
        \end{equation*}
        be the unique weak solution to Equation \eqref{eq:Interacting_Equation-E}. 
        
        Let $t \in [0,T]$. Then for every $u \in V$ the conditional law of the collection of random variables $\big( X^v[t] \big)_{v \in B_1(u)}$ conditioned on $\big( X^u[t], X^{\pi[u]}[t] \big)$ is independent of the choice of $u$. 
	
    	More specifically, for every $t\in [0, T]$ there exists a measurable map
        \begin{equation*}
            \Lambda_t: \cC_t^d \times \cC_t^d \to \cP_2\Big( \cG_\ast^{(\overline{V}, \overline{E})}(\cC_t^d) \Big)
        \end{equation*}
        such that for every Borel set $\cA \in \cB\big( \cG_\ast^{(\overline{V}, \overline{E})}(\cC_t^d) \big)$ we have that
    	\begin{equation*}
    		\Lambda_t\big( X^u[t], X^{\pi[u]}[t] \big)\big[ \cA \big] = \bP\Big[ \big( X^{v}[t] \big)_{v\in B_1(u)} \in \cA \Big| X^u[t], X^{\pi[u]}[t] \Big]
    		\quad\bP\mbox{-almost surely.}
    	\end{equation*}
    \end{proposition}
    Proposition \ref{prop:SymmetryLaw} is a generalisation of \cite{lacker2020marginal}*{Proposition 3.17} which proves a similar result for Galton-Watson random graphs marked with SDEs driven by Brownian motion. We prove this result for completeness:
    \begin{proof}
    	For $(V, E, \oSlash)$, let $\phi:V \to V$ be a graph isomorphism ($(u, v) \in E$ if and only if $(\phi[u], \phi[v]) \in E$) that satisfies $\phi[\oSlash] = u$. Then
        \begin{equation*}
            (\tilde{V}, \tilde{E}, \tilde{\oSlash}):= \Big( \big\{ \phi[u]: u \in V \big\}, \big\{ \{\phi[u], \phi[v]\}: \{u, v\} \in E \big\}, u \Big) \in \cG_{\ast}
        \end{equation*}
        is equivalent to the $\kappa$-regular tree and the system of locally interacting equations
        \begin{equation*}
            \left.
            \begin{aligned}
                &d\tilde{X}_t^{\tilde{u}} = b\Big( t, \tilde{X}^{\tilde{u}}[t], \mu^{\tilde{u}}\big( \tilde{X}[t] \big) \Big) dt + d\tilde{Z}^{\tilde{u}} \quad \tilde{X}_0^{\tilde{u}} \sim \lambda_0
                \\
                &\mu^{\tilde{u}}\big( \tilde{X}[t] \big) = \tfrac{1}{\kappa} \sum_{\tilde{v} \in N_{\tilde{u}}} \delta_{X^{\tilde{u}}[t]} 
            \end{aligned}
            \right\}
            \quad \tilde{u} \in \tilde{V}
        \end{equation*}
        satisfies Assumption \ref{assumption:local} and has a weak solution 
        \begin{align*}
            &\Big( (\tilde{\Omega}, \tilde{\cF}, \tilde{\bP}), \big( \tilde{V}, \tilde{E}, \tilde{\oSlash}, (\tilde{X}_0^{\tilde{u}}, \tilde{Z}^{\tilde{u}}, \tilde{X}^{\tilde{u}} )_{\tilde{u} \in \tilde{V}} \big) \Big)
            \quad\mbox{defined by}
            \\
            &(\tilde{\Omega}, \tilde{\cF}, \tilde{\bP}) = (\Omega, \cF, \bP)
            \qquad\mbox{and}\qquad
            \forall u \in V, \quad \tilde{X}_0^{\phi[u]} = X_0^u, \quad \tilde{Z}^{\phi[u]} = Z^u, \quad \tilde{X}^{\tilde{u}} = X^u. 
        \end{align*}
        By Theorem \ref{theorem:Exist+Uniq-local}, this weak solution is unique and we conclude that for every $\cA \in \cB\big( \cG_{\ast}^{B_1(\oSlash)}(\cC_T^d) \big)$
        \begin{equation}
            \label{eq:prop:SymmetryLaw:pf1}
            \bP\Big[ (X^v)_{v \in B_1(\oSlash)} \in \cA \Big] = \tilde{\bP}\Big[ (\tilde{X}^{\tilde{v}})_{\tilde{v} \in B_1(\tilde{\oSlash})} \in \cA \Big] = \bP\Big[ (X^v)_{v \in B_1(u)} \in \cA \Big]. 
        \end{equation}

        Next, let $u \in V$ and $t\in [0,T]$. We define the measurable function $\Lambda_t^u: \cC_t^d \times \cC_t^d \to \cP_2\big( \cG_{\ast}^{B_1(u)}(\cC_t^d) \big)$ by
        \begin{equation*}
            \Lambda_t^u\big( X^u[t], X^{\pi[u]}[t] \big)\big[ \cA \big] = \bP\Big[ \big( X^v[t] \big)_{v \in B_1(u)} \in \cA\Big| X^u[t], X^{\pi[u]}[t] \Big]
        \end{equation*}
        Under the isomorphism $\phi$, there exists $v' \in B_1(\oSlash)$ such that $\phi[v'] = \pi[u]$. Let $g: \cC_T^d \times \cC_T^d \to \bR$ and $h: \cG_{\ast}^{B_1(\oSlash)}(\cC_T^d) \to \bR$ be bounded and measurable functions. Then by Equation \eqref{eq:prop:SymmetryLaw:pf1}
        \begin{align*}
            \bE\bigg[ g\big(X^u[t]&, X^{\pi[u]}[t] \big) \cdot \Big\langle \Lambda_t^u\big( X^u, X^{\pi[u]}\big), h \Big\rangle \bigg]
            \\
            =&
            \bE\bigg[ g\big(X^u[t], X^{\pi[u]}[t] \big) \cdot \bE\Big[ h\big( (X^v[t])_{v\in v \in B_1(u)} \big) \Big| X^u[t], X^{\oSlash}[t] \Big] \bigg] 
            \\
            =& \bE\Big[ g\big(X^u, X^{\pi[u]} \big) \cdot h\big( (X^v)_{v \in B_1(u)} \big) \Big]
            =\bE\Big[ g\big(X^\oSlash, X^{v'} \big) \cdot h\big( (X^v)_{v \in B_1(\oSlash)} \big) \Big]
            \\
            =&
            \bE\bigg[ g\big(X^\oSlash[t], X^{v'}[t] \big) \cdot \bE\Big[ h\big( (X^v[t])_{v\in B_1(\oSlash)} \big) \Big| X^\oSlash[t], X^{v}[t] \Big] \bigg] 
            \\
            =& \bE\bigg[ g\big(X^\oSlash[t], X^{v'}[t] \big) \cdot \Big\langle \Lambda_t^\oSlash\big( X^\oSlash, X^{v'}\big), h \Big\rangle \bigg]. 
        \end{align*}
        Further, by a similar argument that concludes Equation \eqref{eq:prop:SymmetryLaw:pf1}, the measures
        \begin{equation*}
            \bP\circ \Big( X^u, X^{\pi[u]} \Big)^{-1} = \bP \circ \Big( X^{\oSlash}, X^{v'} \Big)^{-1} \in \cP_2\big( \cC_T^d \times \cC_T^d \big)
        \end{equation*}
        so we conclude that outside of a $\bP\circ (X^u[t], X^{\pi[u]}[t])^{-1}$-null set the two measures
        \begin{equation*}
            \cC_t^d \times \cC_t^d \ni \big(x[t], y[t] \big) \mapsto \Lambda_t^{\oSlash}(x, y) = \Lambda_t^{u}(x, y) \in \cP_2\big( \cG_{\ast}^{B_1(\oSlash)}(\cC_t^d) \big)
        \end{equation*}
        for every $t\in [0,T]$. 
    \end{proof}
    
    \subsection{Derivation of the marginal}
    
    Before we introduce and derive the local equation for the (finite) neighbourhood $B_1(\oSlash)$, consider the following functional:
    \begin{definition}
        Let $(V, E, \oSlash)$ be the $\kappa$-regular rooted graph, let $(\overline{V}, \overline{E}, \overline{\oSlash})$ be the $B_1(\oSlash)$-subgraph of $(V, E, \oSlash)$ and let $u \in N_\oSlash$. Let $\cL \in \cG_{\ast}^{(V, E)}(\cC_T^d)$ and let $x, y \in \cC_T^d$. 

        Given some function $f: \cG_\ast^{(\overline{V}, \overline{E})}(\cC_T^d) \to \bR^d$ such that
        \begin{equation*}
            \Big| f\Big( \overline{V}, \overline{E}, \big( X^v \big)_{v\in \overline{V}} \Big) \Big| \leq C \sum_{v \in \overline{V}} \big\| X^v \big\|_{\infty, T}^2, 
        \end{equation*}
        we denote
        \begin{align*}
            &\Gamma[f]: \cP_2\Big( \cG_\ast^{(\overline{V}, \overline{E})}(\cC_T^d) \Big) \to \bR^d
            \\
            &\Gamma[f](\cL) = \int_{\cG_{\ast}^{(V, E)}(\cC_T^d)} f\Big( V, E, (X^{v})_{v\in V} \Big) d\cL\Big( (X^v)_{v\in V} \Big)
        \end{align*}
    \end{definition}

    \begin{lemma}
        \label{lemma:Local-progmeas}
        Let $(V, E, \oSlash)$ be the $\kappa$-regular rooted graph, let $(\overline{V}, \overline{E}, \overline{\oSlash})$ be the $B_1(\oSlash)$-subgraph of $(V, E, \oSlash)$ and let $u \in N_\oSlash$. 
        
        Let $B:[0,T] \times \cC_T^d \times \cP_2(\cC_T^d) \to \bR^d$ be progressively measurable and let $\cL \in \cG_{\ast}^{(\overline{V}, \overline{E})}(\cC_T^d)$. Suppose that for every $t\in [0,T]$
        \begin{equation*}
            \int_{\cG_{\ast}^{(\overline{V}, \overline{E})}(\cC_T^d)} \Big| B\Big( t, X^\oSlash[t], \tfrac{1}{\kappa}\sum_{v\in N_{\oSlash} } \delta_{X^v[t]} \Big) \Big| d\cL\Big( \big( X^v \big)_{v\in \overline{V} }\Big)< \infty. 
        \end{equation*}
        We define $f: [0,T] \times \cG_\ast^{(\overline{V}, \overline{E})}(\cC_T^d) \to \bR^d$ by
        \begin{equation}
            \label{eq:lemma:Local-progmeas-f}
            f\Big( t, \big(\overline{V}, \overline{E}, \overline{\oSlash}, (X^v)_{v\in \overline{V}} \big) \Big) = B\Big( t, X^{\oSlash}[t], \tfrac{1}{\kappa}\sum_{v\in N_{\oSlash}} \delta_{X^{v}[t]} \Big) 
        \end{equation}
        Next, given $u\in N_\oSlash$ and $(x, y) \in (\cC_T^d)^{\times 2}$, we define
        \begin{align*}
            &\gamma^u: [0, T] \times \cC_T^d \times \cC_T^d \to \bR^d
            \quad\mbox{by}\quad
            \gamma^u\big( t, x, y \big)
            := \Gamma\Big[ f(t, \cdot) \Big]\Big[ \cL\big[ \cdot\big| (X^\oSlash,X^{u}) = (x, y) \big] \Big]. 
        \end{align*}
        Then the function $\gamma^u: [0,T] \times \cC_T^d \times \cC_T^d \to \bR^d$ is progressively measurable. 
    \end{lemma}

    \begin{proof}
        Courtesy of the \emph{disintegration theorem}, see \cite{Dellacherie1988Probabilities}, we know that there is a measurable mapping
        \begin{equation*}
            (\cC_T^d)^{\times 2} \ni\big( x, y \big) \mapsto \cL\big[ \cdot \big| (X^{\oSlash}, X^{u})=(x,y) \big] \in \cP_2\Big( \cG_{\ast}^{(\overline{V}, \overline{E})}(\cC_T^d) \Big)
        \end{equation*}
        Next, for every choice of $g$, the operation $\Gamma[g]$ is measurable mapping. Finally, note that the function $f$ as defined in Equation \eqref{eq:lemma:Local-progmeas-f} is progressively measurable and the composition of a measurable function and a progressively measurable function is measurable. 
    \end{proof}

    Next, let $b:[0, T] \times \cC_T^d \times \cP_2(\cC_T^d) \to \bR^d$ be progressively measurable and consider the conditional McKean-Vlasov dynamics
    \begin{equation}
        \label{eq:localEquation}
        \left.
        \begin{aligned}
            &X_t^\oSlash = X_0^{\oSlash} + \int_0^t b\Big( s, X^{\oSlash}[s], \tfrac{1}{\kappa} \sum_{v \in N_\oSlash} \delta_{X^v[s]} \Big) ds + Z_t^{\oSlash} 
            \\
            u\in N_{\oSlash},\quad &X_t^u = X_0^u + \int_0^t K(t, s) \Gamma\big[ Q^b(s, \cdot ) \big]\big[ \mu_{X^{u}, X^{\oSlash}} \big] ds + Z_t^{u}
            \\
            &\mu_{x,y} = \bP\Big[ \cdot \Big| (X^\oSlash, X^u) = (x, y) \Big] \circ \Big( \overline{V}, \overline{E}, \overline{\oSlash}, (X^u)_{u \in \overline{V}} \Big)^{-1}
        \end{aligned}
        \right\}
    \end{equation}
    
    This first result (which relies on the Mimicking Theorem) implies the existence result for Theorem \ref{theorem:localEquation-fbm}:
    \begin{theorem}
    	\label{theorem:LocalEquation}
        Let $(V, E, \oSlash)$ be the $\kappa$-regular rooted graph and let $(\overline{V}, \overline{E}, \overline{\oSlash})$ be the $B_1(\oSlash)$-subgraph of $(V, E, \oSlash)$. 
        
        Let $\lambda_0 \in \cP_2(\bR^d)$, let $K: [0,T] \to L^2\big([0,T]; \lin(\bR^d, \bR^d) \big)$ and let $b:[0,T] \times \cC_T^d \times \cP_2\big( \cC_T^d \big) \to \bR^d$. Suppose that Assumption \ref{assumption:local} is satisfied. Then there exists
        \begin{equation*}
            \Big( (\Omega, \cF, \bP), \big( \overline{V}, \overline{E}, \overline{\oSlash}, (X_0^u, Z^u, X^u)_{u\in \overline{V}} \big) \Big)
        \end{equation*}
        such that
        \begin{align*}
    		&(\Omega, \cF, \bP) \quad \mbox{is a probability space and}
    		\\
    		&\omega \mapsto \big( \overline{V}, \overline{E}, \overline{\oSlash}, (X_0^u, Z^u, X^u)_{u\in \overline{V}} \big) \in \cG_{\ast}^{(\overline{V}, \overline{E})}( \bR^d \times \cC_{0,T}^d \times \cC_T^d),
        \end{align*}
        and
        \begin{enumerate}
		\item The collection of random variables $(X_0^v)_{v\in \overline{V}}$ are independent and identically distributed each with distribution $\lambda_0$. Similarly, the collection of Gaussian Volterra processes $(Z^u)_{u\in \overline{V}}$ as defined in Equation \eqref{eq:GaussVolterra} are independent;
		\item The expectation
		\begin{equation*}
			\sup_{u\in \overline{V}} \bE\Big[ \big\| X^u \big\|_{\infty, T}^2 \Big] < \infty;
		\end{equation*}
		\item The random variable
		\begin{equation*}
			\omega \mapsto \int_0^T \Big| Q^b\big( t, X^{\oSlash}[t], \tfrac{1}{\kappa} \sum_{v\in N_{\oSlash}} \delta_{X^v[t]} \Big) \Big|^2 dt < \infty
			\quad\bP\mbox{-almost surely;}
		\end{equation*}
            \item For every $u \in N_{\oSlash}$, we define
            \begin{align*}
                &\mu_{x, y}^u:= \bP\Big[ \cdot \Big| (X^{\oSlash}, X^u) = (x, y) \Big] \circ \Big( V, E, \oSlash, (X^v)_{v\in V} \Big)^{-1} \in \cP\big( \cG_{\ast}^{(\overline{V}, \overline{E})}(\cC_T^d)\big), 
                \\
                &\gamma^u\big(t, x, y\big):= \Gamma\big[Q^b(t, \cdot) \big]\big[ \mu_{x, y}^u \big]. 
            \end{align*}
            For every $u \in N_{\oSlash}$, the random variable
            \begin{equation*}
                \omega \mapsto \int_0^T \Big| \gamma^u\big( t, X^u[t], X^{\oSlash}[t] \big) \Big|^2 dt < \infty \quad \bP\mbox{-almost surely;}
            \end{equation*}
            \item The random variable
            \begin{equation*}
                X_\cdot^\oSlash = X_0^\oSlash + \int_0^\cdot b\Big( t, X^{\oSlash}[t], \tfrac{1}{\kappa}\sum_{v\in N_\oSlash} \delta_{X^v[t]} \Big) dt + Z^\oSlash_{\cdot}
            \end{equation*}
            and for every $u \in N_{\oSlash}$ the random variable
            \begin{equation*}
                X_\cdot^u = X_0^u + \int_0^\cdot K(\cdot, t) \gamma^u\Big( t, X^{u}[t], X^\oSlash[t] \Big) dt + Z^u_{\cdot};
            \end{equation*}
        \end{enumerate}
        We say that 
        \begin{equation*}
            \Big( (\Omega, \cF, \bP), \big( \overline{V}, \overline{E}, \overline{\oSlash}, (X_u^u, Z^u, X^u)_{u\in \overline{V}} \big) \Big)
            \quad
            \mbox{is a weak solution to Equation \eqref{eq:localEquation}}
        \end{equation*}
        Secondly, let
        \begin{equation*}
            \Big( (\Omega, \cF, \bP), \big( V, E, \oSlash, (X_0^u, Z^u, X^u)_{u\in V} \big) \Big)
        \end{equation*}
        be a weak solution to Equation \eqref{eq:Interacting_Equation-U}. Then
        \begin{equation*}
            \Big( (\Omega, \cF, \bP), \big( \overline{V}, \overline{E}, \overline{\oSlash}, (X_0^u, Z^u, X^u)_{u\in \overline{V}} \big) \Big)
        \end{equation*}
        is a weak solution to Equation \eqref{eq:localEquation}. 
    \end{theorem}
    
    \begin{remark}
    	We emphasise here that we could generalise the assumptions on the rooted tree $(V, E, \oSlash)$ to those of \cite{lacker2020marginal} where the graph structure can be random. We choose not to do this as it simplifies the prerequisite material for this paper. Including this would have allowed for collections of interacting processes indexed by a Galton-Watson tree which corresponds to the case when the spacial environment within which the processes interact is random. These arise naturally in applications since the Galton-Watson tree is the appropriate scaling limit of many configuration models which are commonly used models in random graph theory. 
    \end{remark}
    
    \begin{proof}[Proof of Theorem \ref{theorem:LocalEquation}]
    	Firstly, let $\big( (\Omega, \cF, \bP), \big(X_0^u, Z^u, X^u)_{u\in V} \big)$ be a weak solution to Equation \eqref{eq:Interacting_Equation-E} and consider the distribution of the sub-collection of random variables $(X^u)_{u \in \overline{V}}$
    	\begin{equation}
    		\label{eq:theorem:LocalEquation-p1}
    		\left.
    		\begin{aligned}
    			X_t^{\oSlash} =& X_0^{\oSlash} + \int_0^t b\Big( s, X^{\oSlash}[s], \mu^{\oSlash}\big( X[t]\big) \Big) ds + Z_t^{\oSlash}
    			\\
    			u \in N_{\oSlash},\quad X_t^u =& X_0^{u} + \int_0^t b\Big( s, X^{u}[s],\mu^u\big( X[t] \big) \Big) ds + Z_t^u
    		\end{aligned}
    		\quad
    		\right\}
    	\end{equation}
    	Due to the fact that for each $u \in N_{\oSlash}$, the dynamics of $X^u$ are dependent on $N_u$ (which is not contained in the set $B_1(\oSlash)$), this is not a stochastic differential equation but only a stochastic process. For each $u \in N_{\oSlash}$, let us denote $\tilde{Q}_u: [0,T] \times \cG_{\ast}^{(\overline{V}, \overline{E})}(\cC_T^d) \to \bR^d$ to be the progressively measurable functions
    	\begin{equation*}
    		\tilde{Q}_u\Big( t, \big( X^{u}[t] \big)_{u\in \overline{V}} \Big) = \bE\bigg[ Q^b\Big( t, X^{u}[t], \mu^u\big( X[t] \big) \Big) \bigg| \big( X^u[t] \big)_{u\in \overline{V}} \bigg]. 
    	\end{equation*}
    	We apply Theorem \ref{brownianFilteringThm} to conclude that there exists an extension
    	\begin{equation*}
    		\big( \hat{\Omega}, \hat{\cF}, \hat{\bF}, \hat{\bP} \big)
    		\quad \mbox{of}\quad
    		\big( \Omega, \cF, \bF^{(\overline{V}, \overline{E}, \overline{\oSlash})}, \bP \big)
    	\end{equation*}
    	(where 
    	\begin{equation*}
    		\bF^{(\overline{V}, \overline{E}, \overline{\oSlash})} = \Big( \cF_t^{(X^{u})_{u\in \overline{V}}} \Big)_{t\in [0,T]}
    	\end{equation*}
    	is the filtration generated by $\big(\overline{V}, \overline{E}, \overline{\oSlash}, (X^u)_{u \in \overline{V}} \big)$) carrying a Gaussian process with the same covariance as the collection of Gaussian processes $(Z^u)_{u \in \overline{V}}$ such that a weak solution to the stochastic differential equation
    	\begin{equation}
    		\left.
    		\begin{aligned}
    			X_t^{\oSlash} =& X_0^{\oSlash} + \int_0^t b\Big( s, X^{\oSlash}[s], \mu^{\oSlash}\big( X[t]\big) \Big) ds + Z_t^{\oSlash}
    			\\
    			u \in N_{\oSlash},\quad X_t^u =& X_0^{u} + \int_0^t K(t, s) \tilde{Q}_u\Big( s, \big( X^u[t] \big)_{u\in \overline{V}} \Big) ds + Z_t^u
    		\end{aligned}
    		\quad
    		\right\}
    	\end{equation}
    	exists and is equal in distribution to the law of the stochastic process described in Equation \eqref{eq:theorem:LocalEquation-p1}. We remark that the drift term for $X^{\oSlash}$ has not been changed because
    	\begin{equation*}
    		\bE\bigg[ Q^b\Big(t, X^{\oSlash}[t], \mu^{\oSlash}\big( X[t] \big) \Big) \bigg| \big( X^u[t] \big)_{u\in \overline{V}} \bigg] = Q^b\Big(t, X^{\oSlash}[t], \mu^{\oSlash}\big( X[t] \big) \Big). 
    	\end{equation*}
    	Next, we remark that thanks to Theorem \ref{theorem:MRF} the conditional law
    	\begin{align*}
    		\cL\Big( \big( X^v[t] \big)_{v \in B_2(\oSlash)} \Big| \big( X^u[t] \big)_{u \in \overline{V}} \Big) 
    		= \prod_{u\in \overline{V}} \delta_{X^u[t]} \times \prod_{u \in N_{\oSlash}} \cL\Big( \big( X^v[t] \big)_{v \in C_u} \Big| X^u[t], X^{\oSlash}[t] \Big). 
    	\end{align*}
    	Therefore, for every $u \in N_{\oSlash}$ we can write
    	\begin{align}
    		\nonumber
    		&\tilde{Q}_u\Big( t, \big( X^{u}[t] \big)_{u\in \overline{V}} \Big) 
    		=
    		\bE\bigg[ Q^b\Big( t, X^{u}[t], \mu^u\big( X[t] \big) \Big) \bigg| X^{u}[t], X^{\oSlash}[t] \bigg]
    		\\
    		\label{eq:theorem:LocalEquation-p2}
    		&= \int_{\cG_{\ast}^{B_1(u)}(\cC_T^d)} Q^b\Big( t, X^u[t], \tfrac{1}{\kappa} \delta_{X^{\oSlash}[t]} + \tfrac{1}{\kappa}\sum_{v\in C_u} \delta_{x[t]} \Big) d\cL\Big( ( X^v)_{v\in B_1(u)}\big| X^u, X^{\oSlash} \Big)\Big( ( x^v )_{v\in B_1(u)} \Big). 
    	\end{align}
    	Next, thanks to Proposition \ref{prop:SymmetryLaw}, we can conclude that for $(x, y) \in \cC_T^d \times \cC_T^d$, the condition law
    	\begin{equation*}
    		\cL\Big( ( X^v )_{v \in B_1(u)} \Big| (X^u, X^{\oSlash}) = (x, y) \Big) = \cL\Big( \big( X^v\big)_{v \in \overline{V}} \Big| (X^{\oSlash}, X^u) = (x, y) \Big). 
    	\end{equation*}
    	Therefore, we can rewrite Equation \eqref{eq:theorem:LocalEquation-p2} as
    	\begin{align*}
    		\tilde{Q}_u\Big( t, \big( X^v[t]\big)_{v\in \overline{V}} \Big) =& \tilde{\Gamma}_u\Big( t, X^{u}[t], X^{\oSlash}[t] \Big)
    	\end{align*}
    	where
    	\begin{align*}
    		\tilde{\Gamma}_u\Big( t, x, y \Big) =& \hat{\bE}\bigg[ Q^b\Big( t, \hat{X}^{\oSlash}[t], \tfrac{1}{\kappa}\sum_{v\in N_{\oSlash}} \delta_{\hat{X}^v[t]} \Big) \bigg| \big( \hat{X}^{\oSlash}[t], \hat{X}^u[t] \big) = \big( x[t], y[t] \big) \bigg]
    		\\
    		=& \int_{\cG_{\ast}^{(\overline{V}, \overline{E})}(\cC_t^d)} Q^b\Big( t, z^{\oSlash}[t], \tfrac{1}{\kappa} \sum_{v\in N_{\oSlash}} \delta_{z^v[t]} \Big) d\cL\Big( ( X^v )_{v\in \overline{V}} \Big| (X^{\oSlash},X^u) = (x, y) \Big) \Big( (z^v)_{v\in \overline{V}} \Big)
    		\\
    		=& \Gamma\big[ Q^b(t, \cdot) \big]\big[ \mu_{x, y} \big]. 
    	\end{align*}
    	Therefore, we conclude from Theorem \ref{theorem:Exist+Uniq-local} that we can take the subgraph associated to the unit ball $B_1(\oSlash)$ of any weak solution to Equation \ref{eq:Interacting_Equation-E} and obtain that for Lebesgue almost every $t\in [0,T]$, the random variables
    	\begin{align*}
    		\int_0^t& K(t, s) \Gamma\Big[ Q^b(s, \cdot) \Big] \Big[ \hat{\bP}\big[ \cdot \big| (\hat{X}^{\oSlash}, \hat{X}^u) = (\hat{X}^u, \hat{X}^{\oSlash}) \big] \circ \big( (\hat{X}^v)_{v \in \overline{V}} \big) \Big] ds
    		\\
    		&=\int_0^t K(t, s) \tilde{Q}_u\Big( t, \big( \hat{X}^v[t]\big)_{v \in \overline{V}} \Big) ds
    		\quad \mbox{are $\hat{\bP}$-almost surely equal. }
    	\end{align*}
    	As such we conclude that
    	\begin{equation*}
    		\Big( (\Omega, \cF, \bP), \big( \overline{V}, \overline{E}, \overline{\oSlash}, (X_0^v, Z^v, X^v)_{v \in \overline{V}} \big) \Big)
    	\end{equation*}
    	is a weak solution to \eqref{eq:localEquation}. As, by Theorem \ref{theorem:Exist+Uniq-local}, a weak solution to Equation \ref{eq:Interacting_Equation-E} necessarily exists, we conclude that a weak solution to Equation \eqref{eq:localEquation} must also necessarily exist. 
    \end{proof}

	\subsection{Uniqueness in law of the local equation}
	
    We have shown that the local equation (a singular McKean-Vlasov equation) has a weak solution with the same law as the law of some local marginal of a system of locally interacting equations. 
    
    Further, as the system of locally interacting equations has a unique weak solution, the question arises: ``does the associated local equation have a weak unique solution''. In the original work \cite{lacker2020marginal} this was only proved when the drift term was uniformly bounded (which limits the applicability). Here, we develop the ideas of this proof to prove uniqueness in law of the local equation. To do this, we need to verify a transport inequality (which serves a similar purpose to the transport inequality in \cite{lacker2022Hierarchies}) which follows from establishing a Talagrand inequality. 
	 
    Let $b_0:[0,T] \times \cC_T^d \to \bR^d$ and $b:[0,T] \times \cC_T^d \times \cP_2(\cC_T^d) \to \bR^d$ be progressively measurable and let $(V, E, \oSlash)$ be the $\kappa$-regular rooted tree. Consider the system of locally interacting equations indexed by $(V, E, \oSlash)$
	\begin{equation}
	 	\label{eq:Interacting_Equation-U} 
	 	\left.
	 	\begin{aligned}
			&dX_t^u = b_0\Big(t, X^u[t] \Big) + b\Big( t, X^u[t], \mu^u\big( X[t] \big) \Big) dt + dZ_t^u
			\\
			&\mu^u\big( X[t] \big) = \frac{1}{|N_u|} \sum_{v\in N_u} \delta_{X^v[t]}
		\end{aligned}
		\right\} \quad u\in V. 
	\end{equation}
	Notice that the dependency on the local interactions now takes the very specific form. As we shall see below, this will be necessary for the entropy techniques that we use. 

	Following on from Theorem \ref{theorem:LocalEquation}, the associated local equation for the unit ball of the dynamics \eqref{eq:Interacting_Equation-U} should take the form
	\begin{equation}
		\label{eq:Local-Equation-U}
		\left.
        \begin{aligned}
            &X_t^\oSlash = X_0^{\oSlash} + \int_0^t b_0\big(s, X^{\oSlash}[s] \big) + b\Big( s, X^\oSlash[s], \sum_{v\in N_{\oSlash}} \delta_{X^v[s]} \Big) ds + Z_t^\oSlash
            \\
            u \in N_{\oSlash},\quad &X_t^{u} = X_0^u + \int_0^t b_0\big(s, X^u[s] \big) ds + \int_0^t K(t, s) \Gamma\big[ Q^b( s, \cdot) \big]\big[ \mu_{X^u, X^{\oSlash}} \big] ds + Z_t^u
            \\
            &\mu_{x,y} = \bP\big[ \cdot \big| (X^\oSlash, X^u) = (x, y) \big] \circ \Big( \overline{V}, \overline{E}, \overline{\oSlash}, (X^u)_{u \in \overline{V}} \Big)^{-1}
        \end{aligned}
        \right\}
    \end{equation}
    Given that we consider processes of the form described in Equation \eqref{eq:Interacting_Equation-U} (rather than Equation \eqref{eq:Interacting_Equation-E}), we need a new set of assumptions to work with:
	\begin{assumption}
	 	\label{assumption:localU}
		Let $\kappa\geq 2$. Let $(V, E, \oSlash)$ be the $\kappa$-regular rooted tree and let $M \in L^2\big( [0,T]; \bR \big)$. Suppose that
		\begin{enumerate}
			\item Let $K: [0,T] \to L^2\big( [0,T]; \lin(\bR^d, \bR^d) \big)$ be a Volterra kernel that satisfies Assumption \ref{assumption:VolterraK};
            \item Let $\lambda_0 \in \cP(\bR^d)$ and suppose that there exists $\delta>0$ such that
            \begin{equation*}
                \int_{\bR^d} \exp\big( \delta \cdot |x|^2 \big) \lambda_0(dx) < \infty; 
            \end{equation*}
            \item Let $b_0: [0,T] \times \cC_T^d \to \bR^d$ be progressively measurable and for every $X \in \cC_T^d$,
            \begin{align*}
                \forall t\in [0, T], \quad &\Big| Q^{b_0}\big(t, X[t] \big) ds \Big| \leq M_t \cdot \Big( 1 + \| X \big\|_{\infty, t} \Big) < \infty;
            \end{align*}
			\item Let $b:[0,T] \times \cC_T^d \times \cP_1(\cC_T^d) \to \bR^d$ be progressively measurable and for every $X \in \cC_T^d$ and $(\mu, \nu) \in \cP(\cC_T^d)$, 
			\begin{align}
                \nonumber
                &\Big| Q^{b}\big(t, X[t], \mu[t] \big) \Big| \leq M_t \cdot \Big( 1 + \| X\|_{\infty, t} + \bW_{\infty, t}^{(1)}\big[ \mu, \delta_0 \big] \Big) \quad \mbox{and}\quad
                \\
                \label{eq:assumption:localU-Lip}
                &\Big| Q^b\big( t, X[t], \mu[t] \big) - Q^b\big(t, X[t], \nu[t] \big) \Big| \leq M_t \cdot \bW_{\infty, t}^{(1)}\big[ \mu, \nu \big]; 
			\end{align}
		\end{enumerate}
	\end{assumption}

    \subsubsection*{Talagrand inequalities for locally interacting processes}
	We address whether the solution law to Equation \eqref{eq:Interacting_Equation-U} and \eqref{eq:Local-Equation-U} are contained in $\cT_d^1(\cC_T^d )$:
	\begin{proposition}
		\label{proposition:Talagrand2}
		Let $(V, E, \oSlash)$ be the $\kappa$-regular rooted graph and let $(\overline{V}, \overline{E}, \overline{\oSlash})$ be the $B_1(\oSlash)$-subgraph of $(V, E, \oSlash)$. 
  
        Let $\lambda_0 \in \cP_2(\bR^d)$, let $K:[0,T] \to L^2\big( [0,T]; \lin(\bR^d, \bR^d) \big)$, let $b_0:[0,T] \times \cC_T^d \to \bR^d$ and let $b:[0,T] \times \cC_T^d \times \cC_T^d \to \bR^d$. Suppose that Assumption \ref{assumption:localU} is satisfied. 
        
        Then for any $u\in N_{\oSlash}$ and any $(x, y) \in \cC_T^d \times \cC_T^d$, 
		\begin{equation}
			\label{eq:proposition:Talagrand2}
			\mu_{x, y} := \bP\Big[ \cdot \Big| (X^{\oSlash}, X^u)=(x, y) \Big] \circ \Big( (X^v)_{v\in \overline{V}} \Big)^{-1} \in \cT_{\infty}^1\Big( \cG_{\ast}^{(\overline{V}, \overline{E})}(\cC_T^d) \Big). 
		\end{equation} 
	\end{proposition}
	
	\begin{proof}
        We follow the same techniques as those of Proposition \ref{proposition:TalagrandIQ}. For $u \in N_{\oSlash}$ and for every $(x, y) \in \cC_T^d \times \cC_T^d$, we denote
        \begin{align*}
            \mu_{x,y}^u := \bP\Big[ \cdot \Big| (X^{\oSlash}, X^u)=(x, y) \Big] \circ \Big( (X^v)_{v\in \overline{V}} \Big)^{-1}
            \quad \mbox{and}\quad
            \cL^{\oSlash, u}:= \bP \circ \Big( X^\oSlash, X^u\Big)^{-1}. 
        \end{align*}
        Our goal is to verify that the measure $\mu_{x, y}$ satisfies Equation \eqref{eq:theorem:Djellout}, that is there exists $\delta>0$ such that
		\begin{equation*}
			\int_{\cG_{\ast}^{(\overline{V}, \overline{E})}(\cC_T^d)} \exp\Big( \delta \cdot \sum_{v\in \overline{V}} \big\| Z^v \big\|_{\infty, T}^2 \Big) d\mu_{x, y}^u\Big( (Z^v)_{v \in \overline{V}} \Big) < \infty \quad \cL^{\oSlash, u}\mbox{-almost surely. }
		\end{equation*}
        First, we start by remarking that thanks to Assumption \ref{assumption:localU} we have that
        \begin{align*}
            \sum_{v \in B_1(\oSlash)} \big\| X^v \big\|_{\infty, t}^2 \leq& 4 \Bigg( \sum_{v \in B_1(\oSlash)} \big| X_0^v \big|^2 + \sum_{v\in B_1(\oSlash)} \big\| Z^v \big\|_{\infty, t}^2 
            \\
            &+ \sum_{v\in B_1(\oSlash)} \bigg\| \int_0^\cdot b_0\big( s, X^v[s] \big) ds \bigg\|_{\RKHS_t}^2 + \sum_{v \in B_1(\oSlash)} \bigg\| \int_0^\cdot b\Big( s, X^v[s], \mu^v\big(X[s]) \Big) ds \bigg\|_{\RKHS_t}^2 \Bigg)
            \\
            \leq& 4 \Bigg( \sum_{v \in B_1(\oSlash)} \big| X_0^v \big|^2 + \sum_{v\in B_1(\oSlash)} \big\| Z^v \big\|_{\infty, t}^2 + 2(\kappa+1) \int_0^t \big| M_s\big|^2 ds
            \\
            &+ \int_0^t \big| M_s\big|^2 \cdot \bigg( \sum_{v\in B_1(\oSlash)} \big\| X^v \|_{\infty, s}^2  + \sum_{v \in B_1(\oSlash) } \Big( \tfrac{1}{\kappa} \sum_{w\in N_{v}} \big\| X^w \|_{\infty, s}^2 \Big) \bigg) ds \Bigg). 
        \end{align*}
        Therefore, for a choice of $\delta>0$ we have that
        \begin{align*}
            \exp\bigg( \delta& \cdot \sum_{v \in B_1(\oSlash)} \big\| X^v \big\|_{\infty, t}^2 \bigg) 
            \\
            \leq& \frac{\exp\bigg( \delta \cdot 20 \cdot \sum_{v \in B_1(\oSlash)} \big| X_0^v \big|^2 \bigg)}{5} + \frac{\exp\bigg( \delta \cdot 20 \cdot \sum_{v\in B_1(\oSlash)} \big\| Z^v \big\|_{\infty, t}^2 \bigg)}{5} 
            \\
            &+ \frac{\exp\bigg( \delta\cdot20\cdot \int_0^t \big| M_s\big|^2 ds\bigg)}{5}
            \\
            &+ \frac{\exp\bigg( 20 \cdot \int_0^t \big| M_s\big|^2 \cdot \delta \cdot\sum_{v\in B_1(\oSlash)} \big\| X^v \|_{\infty, s}^2 ds \bigg)}{5}
            \\
            &+ \frac{\exp\bigg( 20 \cdot \int_0^t \big| M_s \big|^2 \cdot \tfrac{1}{\kappa} \sum_{v \in B_1(\oSlash) } \left(  \delta \cdot \sum_{w\in B_1(v)} \big\| X^w \|_{\infty, s}^2  \right) ds \bigg)}{5}. 
        \end{align*}
        By defining
        \begin{equation*}
            t' = \inf\bigg\{ t>0: \quad \sup_{s\in [0,T]} \int_s^{s+t} 20 \cdot | M_s |^2 ds = 1 \bigg\}, 
        \end{equation*}
        we apply Jensen's inequality to obtain that for any $t\in [0,t']$
        \begin{align*}
            \exp\bigg( \delta& \cdot \sum_{v \in B_1(\oSlash)} \big\| X^v \big\|_{\infty, t}^2 \bigg) 
            \\
            \leq& \frac{\exp\bigg( \delta \cdot 20 \cdot \sum_{v \in B_1(\oSlash)} \big| X_0^v \big|^2 \bigg)}{5} + \frac{\exp\bigg( \delta \cdot 20 \cdot \sum_{v\in B_1(\oSlash)} \big\| Z^v \big\|_{\infty, t}^2 \bigg)}{5} 
            \\
            &+ \frac{\exp\bigg( \delta\cdot20\cdot \int_0^t \big| M_s\big|^2 ds\bigg)}{5}
            \\
            &+ 4 \cdot \int_0^t \big| M_s\big|^2 \cdot \exp\bigg( \delta \cdot\sum_{v\in B_1(\oSlash)} \big\| X^v \|_{\infty, s}^2 \bigg)ds 
            \\
            &+ 4 \cdot \int_0^t \big| M_s \big|^2 \cdot \tfrac{1}{\kappa} \sum_{v \in B_1(\oSlash) } \exp\bigg( \delta \cdot \sum_{w\in B_1(v)} \big\| X^w \|_{\infty, s}^2 \bigg) ds. 
        \end{align*}
        Choosing $\delta$ small enough, we take expectations followed by supremum over $V$ to conclude that
        \begin{align*}
            \sup_{u \in V} \bE\bigg[& \exp\Big( \delta \cdot \sum_{v \in B_1(u)} \big\| X^v \big\|_{\infty, t}^2 \Big) \bigg]
            \\
            \leq& \sup_{u\in V} \tfrac{1}{5} \bE\bigg[ \exp\Big( \delta \cdot 20 \cdot \sum_{v \in B_1(u)} \big| X_0^v \big|^2 \Big) \bigg] + \sup_{u \in V} \tfrac{1}{5} \bE\bigg[ \exp\Big( \delta \cdot 20 \cdot \sum_{v\in B_1(\oSlash)} \big\| Z^v \big\|_{\infty, t'}^2 \Big)\bigg]
            \\
            &+ \tfrac{1}{5} \exp\Big( \delta\cdot20\cdot \int_0^{t'} \big| M_s\big|^2 ds \Big)
            \\
            &+ 8 \cdot \int_0^t \big| M_s\big|^2 \cdot \sup_{u\in V} \bE\bigg[ \exp\Big( \delta \cdot\sum_{v\in B_1(u)} \big\| X^v \|_{\infty, s}^2 \Big) \bigg] ds. 
        \end{align*}
        Applying Gr\"onwall's inequality, we conclude that for $t\in [0, t']$ and $\delta$ small enough
        \begin{align*}
            \sup_{u \in V} \bE\Bigg[& \exp\bigg( \delta \cdot \sum_{v \in B_1(u)} \big\| X^v \big\|_{\infty, t}^2 \bigg) \Bigg]
            \\
            \leq& \tfrac{1}{5} \Bigg( \sup_{u\in V} \bE\bigg[ \exp\Big( \delta \cdot 20 \cdot \sum_{v \in B_1(u)} \big| X_0^v \big|^2 \Big) \bigg] + \sup_{u \in V} \bE\bigg[ \exp\Big( \delta \cdot 20 \cdot \sum_{v\in B_1(\oSlash)} \big\| Z^v \big\|_{\infty, t'}^2 \Big)\bigg]
            \\
            &+ \exp\Big( \delta\cdot20\cdot \int_0^{t'} \big| M_s\big|^2 ds\Big) \Bigg) \cdot \exp\Big( 8 \cdot \int_0^t \big| M_s\big|^2 \Big)< \infty. 
        \end{align*}

        Now, suppose that for some $t\in [0,T]$ and $\delta>0$ that
        \begin{equation*}
            \sup_{u \in V} \bE\Bigg[ \exp\bigg( \delta \cdot \sum_{v \in B_1(u)} \big\| X^v \big\|_{\infty, t}^2 \bigg) \Bigg] < \infty
        \end{equation*}
        and consider the random variable $\sum_{v \in B_1(u)} \big\| X^v \big\|_{\infty, t+t'}^2$. Arguing as before, we conclude that
        \begin{align*}
            \exp\bigg( \delta& \cdot \sum_{v \in B_1(\oSlash)} \big\| X^v \big\|_{\infty, t+t'}^2 \bigg) 
            \\
            \leq& \frac{1}{5} \exp\bigg( \delta \cdot 20 \cdot \sum_{v \in B_1(\oSlash)} \big\| X^v \big\|_{\infty, t}^2 \bigg) + \frac{1}{5} \exp\bigg( \delta \cdot 20 \cdot \sum_{v\in B_1(\oSlash)} \big\| Z^v \big\|_{\infty, T}^2 \bigg)
            \\
            &+ \frac{1}{5} \exp\bigg( \delta\cdot20\cdot \int_{t}^{t+t'} \big| M_s\big|^2 ds\bigg)
            \\
            &+ 4 \cdot \int_{t}^{t+t'} \big| M_s\big|^2 \cdot \exp\bigg( \delta \cdot\sum_{v\in B_1(\oSlash)} \big\| X^v \|_{\infty, s}^2 \bigg) ds 
            \\
            &+ 4 \cdot \int_t^{t+t'} \big| M_s \big|^2 \cdot \tfrac{1}{\kappa} \sum_{v \in B_1(\oSlash) } \exp\bigg( \delta \cdot \sum_{w\in B_1(v)} \big\| X^w \|_{\infty, s}^2 \bigg) ds. 
        \end{align*}
        Taking expectations and a supremum over $V$, 
        \begin{align*}
            &\sup_{u \in V} \bE\bigg[ \exp\Big( \delta \cdot \sum_{v \in B_1(u)} \big\| X^v \big\|_{\infty, t+t'}^2 \Big) \bigg]
            \\
            &\leq \frac{1}{5}\bigg( \sup_{u\in V} \bE\bigg[ \exp\Big( \delta \cdot 20 \cdot \sum_{v \in B_1(u)} \big\| X^v \big\|_{\infty, t}^2 \Big) \bigg] + \sup_{u \in V} \bE\bigg[ \exp\Big( \delta \cdot 20 \cdot \sum_{v\in B_1(\oSlash)} \big\| Z^v \big\|_{\infty, T}^2 \Big) \bigg] +  \exp\Big( \delta \Big) \bigg)
            \\
            &\quad + 8 \cdot \int_{t}^{t+t'} \big| M_s\big|^2 \cdot \sup_{u\in V} \bE\bigg[ \exp\Big( \delta \cdot\sum_{v\in B_1(u)} \big\| X^v \|_{\infty, s}^2 \Big) \bigg] ds
        \end{align*}
        so that applying Gr\"onwall's inequality, we conclude that for $\delta$ chosen small enough, 
        \begin{equation*}
            \sup_{u \in V} \bE\Bigg[ \exp\bigg( \delta \cdot \sum_{v \in B_1(u)} \big\| X^v \big\|_{\infty, t+t'}^2 \bigg) \Bigg] < \infty. 
        \end{equation*}
        As such, we conclude that there exists $\delta>0$ such that
        \begin{equation*}
            \int_{\cG_{\ast}^{(\overline{V}, \overline{E})}(\cC_T^d)} \exp\Big( \delta \cdot \sum_{v \in \overline{V}} \big\| Z \big\|_{\infty, T}^2 \Big) d\mu_{x, y}^u\Big( (Z^v)_{v\in \overline{V}} \Big) < \infty \quad \cL^{\oSlash, u}\mbox{-almost surely, }
        \end{equation*}
        and we conclude. 
	\end{proof}
	
	\subsubsection*{Uniqueness of the local equation}
	
	This next theorem implies the uniqueness result from Theorem \ref{theorem:localEquation-fbm}:
	\begin{theorem}
        \label{theorem:Local-Uniq}
        Let $(V, E, \oSlash)$ be the $\kappa$-regular rooted tree, let $(\overline{V}, \overline{E}, \overline{\oSlash})$ be the $B_1(\oSlash)$-subgraph of $(V, E, \oSlash)$. 
        
        Let $\lambda_0 \in \cP_2(\bR^d)$, let $K:[0,T] \to L^2\big( [0,T]; \lin(\bR^d,\bR^d) \big)$, let $b_0:[0,T] \times \cC_T^d \to \bR^d$ and let $b:[0,T] \times \cC_T^d \times \cC_T^d\to \bR^d$. 
        
        Suppose that Assumption \ref{assumption:localU} is satisfied. Then there exists a \emph{unique} weak solution
        \begin{equation*}
            \Big( (\Omega, \cF, \bP), \big( \overline{V}, \overline{E}, \overline{\oSlash}, (X_0^u, Z^u, X^u)_{u\in \overline{V}} \big) \Big)
        \end{equation*}
        to Equation \eqref{eq:Local-Equation-U}. 
	\end{theorem}
 
	\begin{remark}
	 	We emphasise that while we prove these results for a class of locally interacting processes driven by additive Gaussian Volterra processes (motivated by the choice of fractional Brownian motion), these results are new even when the additive Gaussian process is taken to be Brownian motion and a generalisation of the results found in \cite{lacker2020marginal}. 
	\end{remark}
 
	\begin{proof}
		Let us start by supposing that there exist two weak solutions to Equation \eqref{eq:Local-Equation-U}
        \begin{equation*}
            \big( \Omega_1, \cF_1, \bP_1, \overline{V}, \overline{E}, \overline{\oSlash}, (X_0^{u}, Z^{u}, X^{u,1})_{u \in \overline{V}} \big)
            \quad \mbox{and}\quad
            \big( \Omega_2, \cF_2, \bP_2, \overline{V}, \overline{E}, \overline{\oSlash}, (X_0^{u}, Z^{u}, X^{u,2})_{u \in \overline{V}} \big)
        \end{equation*}
        and denote
        \begin{equation*}
            \left.
            \begin{aligned}
                \cL_1^{X}[t] =& \bP_1 \circ \Big( \overline{V}, \overline{E}, \overline{\oSlash}, (X^{u,1}[t])_{u\in \overline{V}} \Big)^{-1}, 
                \\
                \cL_2^X[t] =& \bP_2\circ \Big( \overline{V}, \overline{E}, \overline{\oSlash}, (X^{u, 2}[t])_{u \in \overline{V}} \Big)^{-1}
            \end{aligned}
            \right\}
            \in 
            \cP_2\Big( \cG_{\ast}^{(\overline{V}, \overline{E})}(\cC_T^d) \Big). 
        \end{equation*}
        Therefore, for every $v \in N_{\oSlash}$ we also denote the measurable mapping
        \begin{equation*}
            \cC_T^d \times \cC_T^d \ni (x, y) \mapsto 
            \left.
            \begin{aligned}
                \mu_{x, y}^{1}:=& \bP_1\big[ \cdot \big| (X^{\oSlash,1}, X^{u,1}) = (x, y) \big]\circ \Big( \big( \overline{V}, \overline{E}, \overline{\oSlash}, (X^{v,1})_{v\in \overline{V}} \big) \Big)^{-1}
                \\
                \mu_{x, y}^{2}:=& \bP_2\big[ \cdot \big| (X^{\oSlash,2}, X^{u,2}) = (x, y) \big]\circ \Big( \big( \overline{V}, \overline{E}, \overline{\oSlash}, (X^{v,2})_{v\in \overline{V}} \big) \Big)^{-1}
            \end{aligned}
            \right\} \in \cP_2(\cC_T^d)
        \end{equation*}
        and
        \begin{equation*}
            \gamma_1(t, x, y)= \Gamma\big[ Q^b(t, \cdot) \big]\big[ \mu_{x, y}^{1} \big]
            \quad\mbox{and}\quad
            \gamma_2(t, x, y)= \Gamma\big[ Q^b(t, \cdot) \big]\big[ \mu_{x, y}^{2} \big]. 
        \end{equation*}
        Next, we apply Proposition \ref{fBmEntropy} to conclude that for any $u \in N_{\oSlash}$, 
		\begin{align*}
			\bH\Big[ \cL_1^{X}[t] \Big| \cL_2^{X}[t] \Big] =& \frac{1}{2} \bE^{\bP_1}\bigg[ \int_0^t \sum_{v\in N_{\oSlash}} \Big| \gamma_1\big( s, X^v[s], X^{\oSlash}[s] \big) - \gamma_2\big( s, X^v[s], X^{\oSlash}[s] \big) \Big|^2 ds \bigg]
			\\
			=& \frac{\kappa}{2} \bE^{\bP_1}\bigg[ \int_0^t \Big| \gamma_1\big( s, X^u[s], X^{\oSlash}[s] \big) - \gamma_2\big( s, X^{u}[s], X^{\oSlash}[s] \big) \Big|^2 ds \bigg]
			\\
			=& \frac{\kappa}{2} \bE^{\bP_1}\bigg[ \int_0^t \Big| \Gamma\big[ Q^b(t, \cdot) \big]\big[ \mu_{X^u, X^{\oSlash}}^{1} \big] - \Gamma\big[ Q^b(t, \cdot) \big]\big[ \mu_{X^u, X^{\oSlash}}^{2} \big] \Big|^2 ds \bigg]
            \\
            =& \frac{\kappa}{2} \bE^{\bP_1}\Bigg[ \int_0^t \bigg| \Big\langle \mu_{X^u, X^{\oSlash}}^{1} - \mu_{X^u, X^{\oSlash}}^{2}, Q^b\Big(s, X^{\oSlash}[s], \tfrac{1}{\kappa} \sum_{v\in N_{\oSlash}} \delta_{\cdot} \Big) \Big\rangle \bigg|^2 ds \Bigg]
		\end{align*}
		Using that for any choice of $(t, x)\in [0,T] \times \cC_T^d$ we have that the function 
        \begin{align*}
            &Q^b\Big( t, x[t], \sum_{v \in N_{\oSlash}} \delta_{\cdot} \Big): (\cC_T^d)^{\times \kappa} \to \bR^d \quad \mbox{is Lipschitz since by Equation \eqref{eq:assumption:localU-Lip}}
            \\
			&\bigg| Q^b\Big( t, x[t], \tfrac{1}{|N_{\oSlash}|} \sum_{v\in N_{\oSlash}} \delta_{y^v[t]} \Big) - Q^b\Big( t, x[t], \tfrac{1}{|N_{\oSlash}|} \sum_{v\in N_{\oSlash}} \delta_{z^v[t]} \Big) \bigg| \leq M_t \cdot \tfrac{1}{\kappa} \sum_{v\in N_{\oSlash}} \big\| y^v - z^v \big\|_{\infty, t}. 
		\end{align*}
		Therefore, by the Kantorovich dual representation for the Wasserstein distance, the progressive measurability of $Q^b$ and Proposition \ref{proposition:Talagrand2}, we get that
		\begin{align*}
			\frac{\kappa}{2} \bE^{\bP_1}\Bigg[& \int_0^t \bigg| \Big\langle \mu_{X^{\oSlash}, X^{u}}^1 - \mu_{X^{\oSlash}, X^{u}}^2, Q^b\Big( s, X^{\oSlash}[s], \tfrac{1}{\kappa} \sum_{u\in N_{\oSlash}} \delta_{\cdot} \Big) \Big\rangle \bigg|^2 ds \Bigg]
			\\
			\leq& \frac{1}{2} \bE^{\bP_1}\bigg[ \int_0^t |M_s|^2 \cdot \bW_{\infty, s}^{(1)} \Big[ \mu_{X^{\oSlash}, X^{u}}^1[s], \mu_{X^{\oSlash}, X^{u}}^2[s] \Big]^2 ds \bigg]
			\\
			\leq& \frac{1}{2} \int_0^t |M_s|^2 (\gamma')^2 \cdot \bE^{\bP_1}\bigg[ \bH \Big[ \mu_{X^{\oSlash}, X^{u}}^1[s]\Big| \mu_{X^{\oSlash}, X^{u}}^2[s] \Big] \bigg] ds. 
		\end{align*}
		Finally, we denote 
        \begin{equation*}
            \left.
            \begin{aligned}
                \cL_1^{\oSlash,u}= \bP_1 \circ \big( X^{\oSlash,1}, X^{u,1} \big)^{-1} 
                \\
                \cL_2^{\oSlash,u}= \bP_2 \circ \big( X^{\oSlash,2}, X^{u,2} \big)^{-1} 
            \end{aligned}
            \right\} \in \cP_2(\cC_T^d \times \cC_T^d). 
        \end{equation*}
        By the chain rule of relative entropy, we get
		\begin{align*}
			\bE^{\bP_1}\bigg[ \bH \Big[& \mu_{X^{\oSlash}, X^{u}}^1[s] \Big| \mu_{X^{\oSlash}, X^{u}}^2[s] \Big] \bigg] 
			\\
			=& \int_{(\cC_T^d)^{\times 2}} \bigg(\int_{\cG_{\ast}^{(\overline{V}, \overline{E})}(\cC_T^d)} \log\bigg( \frac{ d\mu_{x, y}^1}{d\mu_{x,y}^2}\Big( (z)_{v\in \overline{V}} \Big) \bigg) d\mu_{x, y}^1\big( (z)_{v\in \overline{V}} \big) \bigg) d\cL_1^{\oSlash,u} (x, y)
			\\
			=& \int_{\cG_{\ast}^{(\overline{V}, \overline{E})}(\cC_T^d)} \Bigg( \log\bigg( \frac{ d\cL_1^{X}}{d\cL_2^{X}}\Big( (z^v)_{v\in \overline{V}} \Big) \bigg)  
            - \log\bigg( \frac{ d\cL_1^{\oSlash,u}}{d\cL_2^{\oSlash, u}}\Big( z^{\oSlash}, z^u \Big) \bigg) \Bigg) d\cL_1^{X} \Big( (z^v)_{v\in \overline{V}} \Big)
			\\
			=& \bH\Big[ \cL_1^{X}\Big| \cL_2^X \Big] - \bH\Big[ \cL_1^{\oSlash,u} \Big| \cL_2^{\oSlash,u} \Big]. 
		\end{align*}
		Therefore, we conclude that
		\begin{align*}
			\bH\Big[ \cL_1^{X}[t] \Big| \cL_2^{X}[t] \Big] \leq \int_0^t \gamma^2 |M_s|^2 \cdot \bH\Big[ \cL_1^{X}[s] \Big| \cL_2^{X}[s] \Big] ds
		\end{align*}
		and as such 
		\begin{equation*}
			\bH\Big[ \cL_1^{X}[t] \Big| \cL_2^{X}[t] \Big] = 0. 
		\end{equation*}
	\end{proof}
    
	
	\bibliographystyle{plain}

 
    \appendix
    \section{Appendix}
    \label{Appendix:Entropy}

    \begin{lemma}[\cite{lacker2022Hierarchies}*{Lemma 4.8}]
        \label{Lemma:Lacker-ODE}
        Let $A_k^l$ and $B_k^l$ be defined as in Equation \eqref{eq:Lacker-functions}. For $l> k\geq 1$ and $t \in \bR^+$, we have
        \begin{equation*}
            A_k^l(t) \leq \exp\bigg( -2(l+1) \Big( e^{-\gamma t} - \frac{k}{l+1} \Big)_+^2 \bigg). 
        \end{equation*}
        Moreover, for integers $r\geq 0$, 
        \begin{align*}
            \sum_{l=k}^{\infty} l^r A_k^l(t) \leq \frac{(k+r)!}{(k-1)!} \cdot \frac{e^{\gamma(r+1)t} - 1}{r+1}, 
            \\
            \sum_{l=k}^\infty l^2 B_k^l(t) \leq k(k+1) e^{2\gamma t} \leq 2k^2 e^{2\gamma t}. 
        \end{align*}
    \end{lemma}
	
\end{document}